\documentclass[12pt,oneside,british,a4wide]{amsart}
\usepackage{amsmath, amssymb,verbatim,appendix}
\usepackage{hyperref}
\usepackage[mathscr]{eucal}
\usepackage{amscd}
\usepackage{amsthm}
\usepackage{stmaryrd}
\usepackage{enumerate}
\usepackage{comment}
\usepackage[latin1]{inputenc} 
\usepackage{tikz}
\usetikzlibrary{shapes,arrows}
\usepackage{cite}
\usepackage{url}

\makeatletter
\newcommand{\dotminus}{\mathbin{\text{\@dotminus}}}

\newcommand{\@dotminus}{%
  \ooalign{\hidewidth\raise1ex\hbox{.}\hidewidth\cr$\m@th-$\cr}%
}
\makeatother

\usepackage{setspace,a4wide,color,xcolor,graphicx}

\def\showauthornotes{1}

\ifnum\showauthornotes=1
\newcommand{\Authornote}[2]{{\sf\small\color{red}{[#1: #2]}}}
\else
\newcommand{\Authornote}[2]{}
\fi

\newtheorem{theorem}{Theorem}[section]
\newtheorem{lemma}[theorem]{Lemma}
\newtheorem{corollary}[theorem]{Corollary}
\newtheorem{proposition}[theorem]{Proposition}

\newtheorem{question}[theorem]{Question}

\newtheorem{problem}[theorem]{Problem}
\newtheorem{observation}[theorem]{Observation}

\newtheorem{claim}[theorem]{Claim}
\newtheorem{fact}[theorem]{Fact}

\theoremstyle{definition}

\newtheorem{definition}[theorem]{Definition}

\def\e{\epsilon}

\def\bip{\text{Bip}}
\def\trip{\text{Trip}}

\newcommand{\VC}{\textnormal{VC}}

\newcommand{\disc}{\textnormal{disc}}

\newcommand{\dev}{\textnormal{dev}}

\newcommand{\calF}{\mathcal{F}}
\newcommand{\calH}{\mathcal{H}}

\newcommand{\calN}{\mathcal{N}}

\newcommand{\calP}{\mathcal{P}}

\newcommand{\calI}{\mathcal{I}}

\newcommand{\calQ}{\mathcal{Q}}
\newcommand{\calR}{\mathcal{R}}
\newcommand{\calS}{\mathcal{S}}

\newcommand{\calU}{\mathcal{U}}
\newcommand{\calV}{\mathcal{V}}

\newcommand{\calX}{\mathcal{X}}

\newcommand{\calZ}{\mathcal{Z}}

\def\triads{\operatorname{Triads}}

\def\SVC{\operatorname{SVC}}

\begin{document}
\title[]{Growth of regular partitions of $3$-uniform hypergraphs: strong regularity and the vertex partition}
\author{C. Terry}
\thanks{The author was partially supported by NSF CAREER Award DMS-2115518 and a Sloan Research Fellowship}
\address{Department of Mathematics, Statistics, and Computer Science, University of Illinois at Chicago, Chicago, IL 60607, USA}
\email{caterry@uic.edu}
\maketitle

\begin{abstract}
We consider here the strong regularity for $3$-uniform hypergraphs developed by Frankl, Gowers, Kohayakawa, Nagle,  R\"{o}dl, Skokan, and Schacht.  This type of regular decomposition comes with two components, a partition of the vertices, and a partition of the pairs of vertices. The data of a regular decomposition also includes two parameters measuring quasirandomness, a fixed constant $\e_1>0$, and a function $\e_2:\mathbb{N}\rightarrow (0,1]$.  We define two growth functions associated to a hereditary property $\calH$ of $3$-uniform hypergraphs: $T_{\calH}(\e_1,\e_2)$ which measures the size of the vertex component, and $L_{\calH}(\e_1,\e_2)$ which measures the size of the pairs component.  We introduce the following question.  What are the possible asymptotic growth rates of functions of the form $T_{\calH}$ and $L_{\calH}$?  In this paper, we consider this question for $T_{\calH}$, proving a separation into four classes: constant, polynomial, exponential, or at least wowzer.  The separations among the constant, polynomial and exponential ranges require only slow growing (namely polynomial) choices for $\e_2$. The jump to the wowzer range uses a very fast growing $\e_2$ and makes crucial use of a  lower bound construction for strong graph regularity due to Conlon and Fox. 
\end{abstract}

\section{Introduction}

 Szemer\'{e}di's regularity lemma is a tool for obtaining decompositions of finite graphs.  It says that for all $\e>0$, there is an integer $M=M(\e)$ such that any finite graph can be partitioned into at most $M$ parts, so that all but an $\e$ proportion of the pairs of parts are $\e$-regular.  First used to prove Szemer\'{e}di's Theorem in the integers, the regularity lemma eventually became a central tool in extremal combinatorics.  Beginning with the original proof of Szemer\'{e}di \cite{Szemeredi.1975, Szemeredi.1978}, all known upper bounds for $M(\e)$ took the form $Tw(\e^{-O(1)})$, where $Tw:\mathbb{N}\rightarrow \mathbb{N}$ is the \emph{tower function}, defined by setting $Tw(1)=1$ and $Tw(x)=2^{Tw(x-1)}$ for $x>1$.  Because $M(\e)$ is a limiting factor in applications, the following question became important.  

\begin{question}\label{q:1}
Can the bound $M(\e)$, as a function of $\e$, be improved? 
\end{question}

Question \ref{q:1} was famously answered in the negative by Gowers,  who proved $M(\e)\geq Tw(\e^{-O(1)})$ \cite{Gowers.1997}.  Tight upper and lower bounds were later obtained by Fox and Lov\'{a}sz \cite{Fox.2014} for one  formulation of the regularity lemma. Translating into the formulation used in this paper, their results yield bounds roughly of the form $Tw(\Omega(\e^{-2}))\leq M(\e)\leq Tw(\e^{-4})$  (see also \cite{Moshkovitz.2013, Conlon.2012}).   

Despite the negative answer to Question \ref{q:1} in general, researchers began showing there are special classes of graphs where the bound $M(\e)$ can be improved (for example \cite{Alon.2000, Lovasz.2010, Malliaris.2014, Alon.20058t, FPS2}). The most general results of this kind apply to graphs of bounded VC-dimension, where it was shown by Lov\'{a}sz-Szegedy \cite{Lovasz.2010} and Alon-Fisher-Newman \cite{Alon.2000} that the number of parts can be taken to be polynomial in $\e^{-1}$.  

It was first noted in a paper of Alon, Fox, and Zhao \cite{Alon.2018is} that the results cited above give rise to a dichotomy in the setting of hereditary graph properties. A \emph{hereditary graph property} is a class of finite graphs closed under isomorphism and induced subgraphs.  Using the regularity lemma, we can associate a growth function to every hereditary property.

\begin{definition}\label{def:M}
Given a hereditary graph property $\calH$, define $M_{\calH}:(0,1)\rightarrow \mathbb{N}$ as follows.  For all $\e\in (0,1)$, let $M_{\calH}(\e)$ be the smallest integer so that any sufficiently large  graph  in $\calH$ has an $\e$-regular partition with at most $M_{\calH}(\e)$ parts.  
\end{definition}

Note that $M(\e)$ corresponds to $M_{\calH_{0}}(\e)$, where $\calH_{0}$ is the class of all finite graphs. Alon, Fox, and Zhao \cite{Alon.2018is} observed a dichotomy  in the asymptotic behavior of functions of the form $M_{\calH}$. Here, by ``asymptotic behavior," we mean as $\e$ tends to $0$.  The dichotomy states that either $\calH$ contains graphs of arbitrarily large VC-dimension, in which case $M_{\calH}(\e)$ is bounded below by a tower function, or there is a uniform bound on the VC-dimension of any graph in $\calH$, in which case $M_{\calH}(\e)$ is bounded above by a polynomial.  In part 2 of this series \cite{Terry.2024b}, the author showed one more such dichotomy exists, leading to the following possibilities for the asymptotic behavior of a function of the form $M_{\calH}$. 

\begin{theorem}\label{thm:alljump}
Suppose $\calH$ is a hereditary graph property.  Then one of the following holds. 
\begin{enumerate}
\item (Tower) $Tw(\Omega(\e^{-2}))\leq M_{\calH}(\e)\leq O(Tw(\e^{-4}))$.
\item (Exponential) For some $C>0$, $\e^{-1+o(1)}\leq M_{\calH}(\e)\leq O(\e^{-C})$.
\item (Constant) For some $C\geq 1$, $M_{\calH}(\e)=C$.
\end{enumerate}
\end{theorem}

 Theorem \ref{thm:alljump} combines results of Alon--Fischer--Newman \cite{Alon.2000}, Lov\'{a}sz--Szegedy \cite{Lovasz.2010}, Fox--Lov\'{a}sz, \cite{Fox.2014} and the author \cite{Terry.2024b}.  For more details, we refer the reader to \cite{Terry.2024b}, where a full account of the proof appears.

The goal of the current series of papers is to explore analogues of Theorem \ref{thm:alljump} for hereditary properties of $3$-uniform hypergraphs.   Parts 1 and 2 consider a version of regularity for hypergraphs called \emph{weak regularity} (see Section 3 of \cite{Terry.2024b} for precise definitions).   This was the first type of hypergraph regularity defined in the literature, and was first developed independently by Chung \cite{Chung.1991} and Haviland and Thomason \cite{Haviland.1989}.  While weak hypergraph regularity generalizes graph regularity in several natural ways, it does not imply the full hypergraph analogue of the graph counting lemma (see \cite{Kohayakawa.2010}). 

Several years after the work of Chung and Haviland--Thomason, two stronger versions of hypergraph regularity were developed by Frankl, Kohayakawa, Nagle,  R\"{o}dl, Skokan, and Schacht \cite{Frankl.2002, Rodl.2005a, Rodl.2004,Nagle.2006} on the one hand, and Gowers \cite{Gowers.20063gk, Gowers.2007} on the other.  These types of regularity did produce general counting lemmas, which allowed for important applications, including Gowers' quantitative proof of the multidimensional Szemer\'{e}di's theorem \cite{Gowers.2007}, and the first proofs of hyergraph removal \cite{Gowers.2007, Nagle.2006, Rodl.2009,Rodl.2004,Rodl.2006}.  The notions developed in \cite{Gowers.20063gk, Gowers.2007} and  \cite{Frankl.2002,  Rodl.2005a, Rodl.2004, Nagle.2006} are formally distinct, however, they were shown to be roughly equivalent in the setting of $3$-uniform hypergraphs by Nagle, Poerschke, R\"{o}dl, and Schacht \cite{Nagle.2013} (see also \cite{NRS} for quantitative improvements). 

The goal of the current paper  is to consider the growth of regular decompositions of $3$-uniform hypergraphs for the version of regularity developed by Gowers in \cite{Gowers.2007,Gowers.20063gk}, which is roughly equivalent to that developed by Frankl and R\"{o}dl in \cite{Frankl.2002}. Perhaps surprisingly, we will show in this paper that the growth of these types of regular decompositions is closely related to graph regularity, weak hypergraph regularity, and strong graph regularity. 

We now give a brief description of the strong hypergraph regularity we study in this paper.  For a $3$-uniform hypergraph $H=(V,E)$, a regular decomposition comes with two error parameters $\e_1>0$ and $\e_2:\mathbb{N}\rightarrow (0,1]$, and two complexity parameters $t,\ell$.  The decomposition itself  consists of  a vertex partition $\calP_1=\{V_1,\ldots, V_t\}$, and a set of the form 
$$
\calP_2=\{P_{ij}^{\alpha}:1\leq i,j\leq t , 1\leq \alpha\leq \ell\},
$$
 where for each $1\leq i,j\leq t$, $P_{ij}^1\cup \ldots \cup P_{ij}^{\ell}$ is a partition of $V_i\times V_j$.  We will refer to the pair $(t,\ell)$ as the  \emph{complexity} of $\calP$. One can view $t$ as measuring the complexity of the ``unary" part of the decomposition, and $\ell$ as measuring the complexity of the ``binary" part of the decomposition.  Informally speaking, given a $3$-uniform hypergraph $H=(V,E)$, a decomposition $\calP=(\calP_1,\calP_2)$  is called \emph{$\dev_{2,3}(\e_1,\e_2(\ell))$-regular for  $H$} if the elements of $\calP_2$ are $\e_2(\ell)$-quasirandom as graph relations, and further, the ternary edges $H$ are $\e_1$-quasirandom relative to the underlying graph relations in $\calP_2$.   We give a formal statement of the regularity lemma here, and refer the reader to Section \ref{sec:regularity} for detailed definitions of the terms employed.  The version we state here is due to Gowers.
\begin{theorem}[Gowers \cite{Gowers.2007}]\label{thm:reg2intro} For all $\e_1>0$ and every function $\e_2:\mathbb{N}\rightarrow (0,1]$, there exist positive integers $T=T(\e_1,\e_2)$ and $L=L(\e_1,\e_2)$, such that for every sufficiently large $3$-graph $H=(V,E)$, there exists a $\dev_{2,3}(\e_1,\e_2(\ell))$-regular, $(t,\ell,\e_1,\e_2(\ell))$-decomposition $\calP$ for $H$ with $1\leq t\leq T$ and $1\leq \ell \leq L$.
\end{theorem}
As we see from the statement of Theorem \ref{thm:reg2intro}, there are two size parameters to be considered, the bound on the vertex partition $T$, and the bound on the ``pairs" partition, $L$.  We now define two functions measuring the growth of these bounds in a hereditary property $\calH$ of $3$-uniform hypergraphs.  Defining the desired growth functions requires some additional notation due to the number of quantifiers appearing in Theorem \ref{thm:reg2intro}.  To this end, given, $\e_1>0$,  $\e_2:\mathbb{N}\rightarrow (0,1]$, integers $T,L\geq 1$, and a hereditary property $\calH$ of $3$-uniform hypergraphs, let $\psi(\e_1,\e_2,T,L,\calH)$ be the statement:
\begin{align*}
&\text{For every sufficiently large $H\in \calH$, there exists $1\leq t\leq T$ and $1\leq \ell\leq L$}\\ &\text{and a $\dev_{2,3}(\e_1,\e_2(\ell))$-regular $(t,\ell,\e_1,\e_2(\ell))$-decomposition $\calP$ for $H$}.
\end{align*}
We now define functions $T_{\calH}$ and $L_{\calH}$ with the intent of capturing the size of the most efficient vertex partitions and pairs partitions, respectively.

\begin{definition}\label{def:vertM}
Suppose $\calH$ is a hereditary property of $3$-uniform hypergraphs.  Given $\e_1>0$ and $\e_2:\mathbb{N}\rightarrow (0,1)$, let $T_{\calH}(\e_1,\e_2)$ be minimal so that $\psi(\e_1,\e_2,T_{\calH}(\e_1,\e_2),L,\calH)$ is true for some $L\geq 1$.
\end{definition}

\begin{definition}\label{def:pairsM}
Suppose $\calH$ is a hereditary property of $3$-uniform hypergraphs.  Given $\e_1>0$ and $\e_2:\mathbb{N}\rightarrow (0,1)$, let $L_{\calH}(\e_1,\e_2)$ be minimal so that $\psi(\e_1,\e_2,T,L_{\calH}(\e_1,\e_2),\calH)$ is true for some $T\geq 1$.
\end{definition}

Note that for any hereditary property $\calH$ of $3$-uniform hypergraphs, $T_{\calH}(\e_1,\e_2)\leq T(\e_1,\e_2)$ and $L_{\calH}(\e_1,\e_2)\leq L(\e_1,\e_2)$, where $T(\e_1,\e_2)$ and $L(\e_1,\e_2)$ are from Theorem \ref{thm:reg2intro}.  

The central topic of this paper and part 4 \cite{Terry.2024d} is the possible asymptotic behavior of $T_{\calH}$ and $L_{\calH}$, where by ``asymptotic behavior,'' we mean for $\e_1>0$ which are sufficiently small, and functions $\e_2:\mathbb{N}\rightarrow (0,1]$ where $\e_2(n)$ tends to $0$ sufficiently fast as $n$ goes to infinity.  

\begin{problem}\label{prob:main}
Investigate the asymptotic behavior of $T_{\calH}(\e_1,\e_2)$ and $L_{\calH}(\e_1,\e_2)$ for $\e_1>0$ sufficiently small, and  $ \e_2:\mathbb{N}\rightarrow (0,1]$ tending to $0$ sufficiently fast.  
\end{problem}

In applications of Theorem \ref{thm:reg2}, it is usually sufficient to work with $\e_2:\mathbb{N}\rightarrow (0,1]$ tending to zero at only a polynomial rate.  In light of this, it is natural to consider Problem \ref{prob:main} restricted to polynomial $\e_2$. This problem was brought to the attention of the author by Asaf Shapira after an initial draft  of this paper appeared on the arXiv. 

\begin{problem}[Shapira]\label{prob:poly}
  $\text{ }$
\begin{enumerate}
\item Investigate the asymptotic behavior of $T_{\calH}(\e_1,\e_2)$  for $\e_1>0$ sufficiently small, and {\bf polynomial} $ \e_2:\mathbb{N}\rightarrow (0,1]$ tending to $0$ sufficiently fast.  
\item Investigate the asymptotic behavior of $L_{\calH}(\e_1,\e_2)$  for $\e_1>0$ sufficiently small, and {\bf polynomial} $ \e_2:\mathbb{N}\rightarrow (0,1]$ tending to $0$ sufficiently fast.  
\end{enumerate}
\end{problem}

The $T_{\calH}$ part of Problem \ref{prob:main} is the main topic of this paper, and all but one of our results are also relevant to Problem \ref{prob:poly}(1). We have taken care in this updated version to make explicit how our results relate to Problem \ref{prob:poly1}.  We also refer the reader to part 4 \cite{Terry.2024d} for results about $L_{\calH}$.

We next discuss what was known previously about $T_{\calH}$ and $L_{\calH}$. For this, we will require reference to the so-called Ackerman hierarchy.  Following the notation in \cite{Moshkovitz.2019}, we define $Ack_1:\mathbb{N}\rightarrow \mathbb{N}$ to be the function $x\mapsto 2^x$, and for $k>1$, we let $Ack_{k}:\mathbb{N}\rightarrow \mathbb{N}$  be the function satisfying $A_{k}(1)=1$ and for all $x>1$, $Ack_{k}(x)=Ack_{k-1}(Ack_{k}(x-1))$.  The functions $Ack_2$ and $Ack_3$ will be particularly important in the following discussion.  Note that $Ack_2$ is simply the tower function, $Tw$.  The function $Ack_3$ is also referred to as the \emph{wowzer function}, which we will denote by $W:\mathbb{N}\rightarrow \mathbb{N}$. With this notation, the wowzer function $W$ is defined by setting $W(1)=1$ and for all $x>1$, setting $W(x)=Tw(x-1)$.  Given an arbitrary function $f:\mathbb{N}\rightarrow \mathbb{N}$, we will informally refer to $f$ as \emph{$Ack_k$-type} if $f=o(Ack_{k+1})$ and $f=\Omega(Ack_{k})$. 

The proof of Theorem 8.10 in \cite{Gowers.2007} produces bounds $T(\e_1,\e_2)$ and $L(\e_1,\e_2)$ for  Theorem \ref{thm:reg2intro} which are roughly $poly(\e_1^{-1})$ many iterations of a function $f$, where $f$ is of the form   $f(x)=Tw(p(\e_2(x)^{-1},\e_1^{-1}))$, for some polynomial $p(x,y)$. More specifically, 
\begin{align}\label{bound}
T(\e_1,\e_2),L(\e_1,\e_2)\leq f^{(m)}(1),
\end{align}
where $m=2^{10}\e_1^{-2}$.  Since bounds are not spelled out explicitly in \cite{Gowers.2007}, we provide a sketch of a proof of Theorem \ref{thm:reg2intro} including the above stated bounds, in Appendix \ref{app:bounds}.  When $\e_2$ is a function of tower-type or below, one can see the above bounds are wowzer-type.  If $\e_2$ is of $Ack_k$-type for some $k>2$, then the bound $f^{(m)}(\e^{-1})$ will be of $Ack_{k+1}$-type. For applications, $\e_2$ is typically a polynomial,  in which case the bound in (\ref{bound}) for $T(\e_1,\e_2)$ and $L(\e_1,\e_2)$ is wowzer-type.  

Several years after Theorem \ref{thm:reg2intro} was proved, Moshkovitz and Shapira \cite{Moshkovitz.2019,MoshkovitzSimple} showed wowzer-type bounds are necessary for the vertex partition.  In particular, their work implies that there exists some hereditary property $\calH$, and a polynomially growing $\e_2$, so that  $T_{\calH}(\e_1,\e_2)\geq W(\e_1^{-C})$, for some $C>0$.  Their lower bound is essentially tight, in the sense that it matches the form of the upper bound $T(\e_1,\e_2)$ in Theorem \ref{thm:reg2}, for their choice of slow growing $\e_2$.  The only other previous result about these bounds, due to the author, says that when $\calH$ has finite $\VC_2$-dimension, $L_{\calH}$ is bounded above by a polynomial in $\e_1^{-1}$ \cite{Terry.2022}.

 The main theorem of this paper is Theorem \ref{thm:strong1} below, which gives substantial information about the possible asymptotic behavior of $T_{\calH}(\e_1,\e_2)$, with regards to both Problem \ref{prob:main} and Problem \ref{prob:poly}(1).

\begin{theorem}\label{thm:strong1}
Suppose $\calH$ is a hereditary property of $3$-uniform hypergraphs.  Then one of the following holds.
\begin{enumerate}
\item (At Least Wowzer)\footnote{An upper bound for this range comes from the proof of Theorem \ref{thm:reg2intro}. The growth rate of said upper bound depends on the growth rate of $\e_2$, as one can see from the discussion around inequality (\ref{bound}).}  For all sufficiently small $\e_1>0$ there exists $\e_2:\mathbb{N}\rightarrow (0,1]$ so that $ W(\Omega(\e_1^{-1/7}))\leq T_{\calH}(\e_1,\e_2)$.
\item (Exponential) There exist $C,C'>0$ and a polynomial $p(x,y)$ so that for all sufficiently small $\e_1>0$, and all $\e_2:\mathbb{N}\rightarrow(0,1]$ satisfying $\e_2(x)\leq p(\e_1,x^{-1})$, 
$$
2^{\e_1^{-C}}\leq T_{\calH}(\e_1,\e_2)\leq 2^{\e_1^{-C'}}.
$$
\item (Polynomial) There exist $C,C'>0$ and a polynomial $p(x,y)$ so that for all sufficiently small $\e_1>0$, and all $\e_2:\mathbb{N}\rightarrow(0,1]$ satisfying $\e_2(x)\leq p(\e_1,x^{-1})$, 
 $$
 \e_1^{-C}\leq T_{\calH}(\e_1,\e_2)\leq \e_1^{-C'}.
 $$
 
\item (Constant) There exists $C>0$ and a polynomial $p(x,y)$ so that for all sufficiently small $\e_1>0$, and all $\e_2:\mathbb{N}\rightarrow(0,1]$ satisfying $\e_2(x)\leq p(\e_1,x^{-1})$, 
$$
T_{\calH}(\e_1,\e_2)=C.
$$
\end{enumerate}
\end{theorem}

We note here two updates we have made to Theorem \ref{thm:strong1} compared to the original version.  First, in  the original version of this theorem, the upper bound in range (2) of Theorem \ref{thm:strong1}  was a double exponential in $\e_1^{-1}$, a bound which came from part 1 of this series \cite{Terry.2024a}. The main result of part 1 \cite{Terry.2024a} was recently improved to a single exponential by Gishboliner, Shapira, and Widgerson, \cite{GSW}, which yields the tight upper bound for range (2) as stated in Theorem \ref{thm:strong1} above.   Second, we have updated Theorem \ref{thm:strong1} to reflect the fact that many of our proofs require only polynomial $\e_2$.  In particular, as Theorem \ref{thm:strong1} makes clear, the ``jumps" in ranges (2)-(4) apply even when one restricts to polynomial choices of $\e_2$.  On the other hand, our proof of the lower bound in range (1) makes use of a very fast growing $\e_2$, suggesting the fastest growth rate may differ if one restricts to the case of polynomial $\e_2$.  This narrows Problem \ref{prob:poly}(1) down to just those $\calH$ which fall into range (1) in Theorem \ref{thm:strong1}.  The proof of Theorem \ref{thm:strong1} gives combinatorial characterizations of each class, including range (1), allowing for a completely explicit statement of what remains open for Problem \ref{prob:poly}(1) (see Problem \ref{prob:poly1}).

The proof of Theorem \ref{thm:strong1} makes crucial use of two connections with other types of regularity, namely strong graph regularity on the one hand, and weak hypergraph regularity on the other.  We will now outline these connections and their roles in the proof of Theorem \ref{thm:strong1}.

The first main ingredient in the proof of Theorem \ref{thm:strong1} is a sufficient condition for $T_{\calH}$ to fall into range (1).   To describe this in more detail, we require some notation. For any graph $G=(V,E)$, we can build a $3$-uniform hypergraph from $G$ by adding on $n$-many ``dummy vertices." In particular, we let $C$ be a new set of vertices of size $n$, and define 
$$
n\otimes G:=(V\cup C, \{xyz: xy\in E, z\in C\}).
$$
A crucial example of this construction is when the starting graph $G$ is the \emph{$k$-power set graph},
$$
U(k):=(\{a_i: i\in [k]\}\cup \{b_S:S\subseteq [k]\},\{a_ib_S:i\in S\}).
$$

Given a graph $G=(V,E)$, one can naturally define a corresponding bipartite graph by doubling its vertex set.  In particular, we let $\bip(G)$ have vertex set $\{a_v,b_v: v\in V\}$, and edge set $\{a_vb_{v'}:vv'\in E\}$.  Similarly, given a $3$-uniform hypergraph $H=(V,E)$, we define $\trip(G)$ to be the $3$-uniform hypergraph with vertex set $\{a_v,b_v,c_v: v\in V\}$ and edge set $\{a_vb_{v'}c_{v''}:vv'v''\in E\}$. The growth of regular partitions in a hereditary property $\calH$ is closely related to $\trip(\calH)$, which we define to be the hereditary closure of $\{\trip(H): H\in \calH\}$.  An important division for us will be between properties which are \emph{close to finite slicewise VC-dimension} and those which are \emph{far from finite slicewise VC-dimension}, where the latter class consists of exactly those $\calH$ satisfying $k\otimes U(k)\in \trip(\calH)$ for all $k\geq 1$ (see Section \ref{ss:homsl} for more details).  

The first main ingredient for Theorem \ref{thm:strong1} is showing that when $k\otimes U(k)\in \trip(\calH)$ for all $k\geq 1$,  we have that $T_{\calH}$ is bounded below by a wowzer type function. 
 
 \begin{theorem}\label{thm:main1}
 Suppose $\calH$ is a hereditary property of $3$-uniform hypergraphs. If $k\otimes U(k)\in \trip(\calH)$ for all $k\geq 1$, then for all sufficiently small $\e_1>0$, there exists $\e_2:\mathbb{N}\rightarrow (0,1]$ so that $T_{\calH}(\e_1,\e_2)\geq W(\Omega(\e_1^{-1/7}))$. 
 \end{theorem}
 
Theorem \ref{thm:main1} follows from Theorem \ref{thm:main0} below, which says there exist $3$-uniform hypergraphs which require $T_{\calH}$ to be large, and which have the form $N\otimes G$, for some graph $G$. Theorem \ref{thm:main0} is proved in Section \ref{sec:wowzer}.

\begin{theorem}\label{thm:main0}
There are constants $C,\e^*>0$ and $K\geq 1$ so that for all $0<\e_1<\e^*$, there exists $\e_2:\mathbb{N}\rightarrow (0,1]$ so that for any integers $\ell,N\geq 1$,  there exists a bipartite graph $G$ on at least $N$ vertices so that any $\dev_{2,3}(\e_1,\e_2(\ell))$-regular $(t,\ell,\e_1,\e_2(\ell))$-decomposition of $H:=N\otimes G$ has $t\geq W(C(\e_1^{-1/K}))$.
\end{theorem}

Theorem \ref{thm:main0} will be used to deduce Theorem \ref{thm:main1}.     Theorem \ref{thm:main0} also answers a question from \cite{Terry.2022}.  In particular, it shows there exist $3$-uniform hypergraphs of very small $\VC_2$-dimension (in this case, $\VC_2$-dimension $1$), which require wowzer-type bounds on the vertex partition (see Section \ref{sec:wowzer} for details on this).  We note here that the proof of Theorem \ref{thm:main0} is the only place in the paper which makes use of superpolynomial $\e_2$. Indeed, it uses a wowzer choice of $\e_2$.

The proof of Theorem \ref{thm:main0} leverages an interesting connection to strong graph regularity. First proved in \cite{Alon.2000}, the \emph{strong graph regularity lemma} says roughly the following: for any function $f:\mathbb{N}\rightarrow (0,1)$ and any $\e>0$, there exists a bound $M=M(f,\e)$ so that for any sufficiently large graph $G$, there are partitions $\calV$ and $\calU$ of $V(H)$ so that $\calV\preceq \calU$, so that $|\calV|\leq M$, so that $\calV$ is $f(|\calU|)$-regular, and so that for most pairs $(V_i,V_j)\in \calV^2$ the density of $G$ on $(V_i,V_j)$ is approximately the same as the density of $G$ on the $\calU$-pair containing $(V_i,V_j)$.  The proof produces a wowzer type bound for the sizes of $\calU,\calV$, which was later shown to be necessary by Conlon and Fox \cite{Conlon.2012} and independently, Kalyanasundaram and Shapira  \cite{KaSh}.    

The graph $G$ appearing in Theorem \ref{thm:main0} is $\bip(G')$, where $G'$ is a graph constructed by Conlon and Fox \cite{Conlon.2012} to prove a lower bound for the strong graph regularity lemma.  Crucially, Conlon and Fox proved this $G'$ requires wowzer bounds for a weakening of the usual strong graph regularity lemma (see Section \ref{ss:CF} for details).  To prove Theorem \ref{thm:main0}, we show that if $H=N\otimes \bip(G')$ had regular partitions with sub-wowzer-many vertex parts, it would produce regular partitions for the \emph{graph} $G'$, violating the theorem of Conlon and Fox.  The  proof of Theorem \ref{thm:main0} shows that strong graph regularity is  the main limiting factor with regards to the bound $T$ in Theorem \ref{thm:reg2intro}.

The second main ingredient in the proof of Theorem \ref{thm:strong1} is showing that if the hypotheses of Theorem \ref{thm:main1} fail for a property $\calH$, then $T_{\calH}$  falls into ranges (2), (3), or (4).  These results require only polynomial $\e_2$, which we make explicit in the statement below.

\begin{theorem}\label{thm:main2intro}
Suppose $\calH$ is a hereditary property of $3$-uniform hypergraphs, and assume that for some $k\geq 1$, $k\otimes U(k)\notin \trip(\calH)$.  Then one of the following holds.
\begin{enumerate} 
\item (Exponential) There exist $C,C'>0$ and a polynomial $p(x,y)$ so that for all $\e_2:\mathbb{N}\rightarrow (0,1]$ satisfying $\e_2(x)\leq p(\e_1,x^{-1})$, 
$$
2^{\e_1^{-C}}\leq T_{\calH}(\e_1,\e_2)\leq 2^{\e_1^{-C'}}.
$$
\item (Polynomial) There exist $C,C'>0$ and a polynomial $p(x,y)$ so that for all $\e_2:\mathbb{N}\rightarrow (0,1]$ satisfying $\e_2(x)\leq p(\e_1,x^{-1})$, 
$$
\e_1^{-C}\leq T_{\calH}(\e_1,\e_2)\leq \e_1^{-C'}.
$$
\item (Constant) There exist  $C\geq 1$ and a polynomial $p(x,y)$ so that for all $\e_2:\mathbb{N}\rightarrow (0,1]$ satisfying $\e_2(x)\leq p(\e_1,x^{-1})$, $T_{\calH}(\e_1,\e_2)=C$.
\end{enumerate}
\end{theorem}

Combining Theorems \ref{thm:main2intro} and \ref{thm:main0} immediately yields Theorem \ref{thm:strong1}.  While the proof of Theorem \ref{thm:main0} is connected to strong graph regularity, the proof of Theorem \ref{thm:main2intro} is connected to weak hypergraph regularity, as we now explain. The precise definition of weak regularity is not needed in this paper, but the interested reader can find it stated in Section 3 of \cite{Terry.2024b}. 

 It was shown by Chung \cite{Chung.1991} that all $3$-uniform hypergraphs admit weak $\e$-regular partitions, leading to the definition in \cite{Terry.2024b} of a corresponding growth function $M_{\calH}^{weak}:(0,1)\rightarrow \mathbb{N}$ in analogy to of Definition \ref{def:M}\footnote{In \cite{Terry.2024b}, this function is simply referred to as $M_{\calH}$ as there is only one type of hypergraph regularity appearing in that paper.}

 The subject of part 2 of this series is the asymptotic behavior of $M_{\calH}^{weak}$.  It was shown there that the behavior of $M^{weak}_{\calH}$ is closely related to the function $M_{\calH}^{hom}$, which measures the size of so-called \emph{homogeneous} partitions, when they exist.  Roughly speaking, an $\e$-homogeneous partition of a $3$-uniform hypergraph is a vertex partition in which almost all triples of parts have density within $\e$ of $0$ or $1$ (for precise definitions, see Section \ref{sec:hom}).  The main ingredient for Theorem \ref{thm:main2intro}  is Theorem \ref{thm:ingredient1} below, which says that when the hypotheses of Theorem \ref{thm:main2intro} hold, $T_{\calH}$ is roughly controlled by $M^{weak}_{\calH}$ and $M_{\calH}^{hom}$.

\begin{theorem}\label{thm:ingredient1}
For all $k\geq 1$, there is $K\geq 1$ so that the following holds.  Suppose $\calH$ is a hereditary property of $3$-uniform hypergraphs and assume $k\otimes U(k)\notin \trip(\calH)$.  Then there is a polynomial $p(x,y)$ so that for all sufficiently small $\e_1>0$ and all $\e_2:\mathbb{N}\rightarrow (0,1]$ satisfying $\e_2(x)\leq p(\e_1,x^{-1})$, 
$$
M^{weak}_{\calH}(\e_1^{1/K})\leq T_{\calH}(\e_1,\e_2)\leq M_{\calH}^{hom}(\e_1^{K}).
$$
\end{theorem}

We now state the characterization of the possible growth rates of functions of the form $M_{\calH}^{weak}$ and $M_{\calH}^{hom}$.  This is obtained from the main result of \cite{Terry.2024b}, updated to reflect the improvement of part 1 \cite{Terry.2024a} which was recently obtained by Gishboliner, Shapira, and Widgerson in \cite{GSW}.

\begin{theorem}[Theorem 1.11 of \cite{Terry.2024b} + Theorem 1.3 \cite{GSW}]\label{thm:weak}
Suppose $\calH$ is a hereditary property of $3$-uniform hypergraphs.  Then one of the following holds.
\begin{enumerate}
\item  (Tower) For some $C,C'>0$, $ Tw(\e^{-C})\leq M^{weak}_{\calH}(\e)\leq Tw(\e^{-C'})$.
\item (Exponential) For some $C,C'>0$ and $K\geq 1$, 
$$
2^{\e^{-C}}\leq M^{weak}_{\calH}(\e)\leq M_{\calH}^{hom}(\e^K)\leq 2^{\e^{-C'}}.
$$
\item (Polynomial) For some $C,C'>0$ and $K\geq 1$, $\e^{-C}\leq M^{weak}_{\calH}(\e)\leq M_{\calH}^{hom}(\e^K)\leq \e^{-C'}$.
\item (Constant) For some $C\geq 1$, $M^{weak}_{\calH}(\e)=M_{\calH}^{hom}(\e)=C$.
\end{enumerate}
\end{theorem}

It was shown in part 2 \cite{Terry.2024b} that a property $\calH$ falls into ranges (2), (3), or (4) in Theorem \ref{thm:weak} if and only if for some $k\geq 1$, $k\otimes U(k)\notin \trip(\calH)$.  Combining this with Theorem \ref{thm:ingredient1} yields Theorem \ref{thm:main2intro}.   

We observe here that the possibilities for $M^{weak}_{\calH}$ in Theorem \ref{thm:weak} are almost the same as those for $T_{\calH}$  in Theorem \ref{thm:strong1}, with the exception of the fastest growth rate,  where Theorem \ref{thm:strong1} contains a lower bound which is wowzer, while Theorem \ref{thm:weak} contains tower type upper and lower bounds.  The similarity of the growth rates in Theorem \ref{thm:strong1} and Theorem \ref{thm:weak} is no coincidence.  In particular, we will show that  a hereditary property of $3$-uniform hypergraphs $\calH$ falls into range (1), (2), (3), (4) (respectively) in Theorem \ref{thm:weak} if and only if it falls into range (1), (2), (3), (4)  (respectively) in Theorem \ref{thm:strong1}.

We now discuss open problems around $T_{\calH}$.  Our proof of Theorem \ref{thm:strong1} will show  the properties falling into range (1) are exactly those $\calH$ satisfying $k\otimes U(k)\in \trip(\calH)$ for all $k\geq 1$.   Thus, the remaining open portion of Problem \ref{prob:poly}(1) is as follows.

\begin{problem}[Open part of Problem \ref{prob:poly}(1)]\label{prob:poly1}
Answer Problem \ref{prob:poly}(1) for $\calH$ satisfying $k\otimes U(k)\in \trip(\calH)$ for all $k\geq 1$.
\end{problem}

Significant progress on Problem \ref{prob:poly1} by Gishboliner, Shapira, and Widgerson was recently communicated to the author \cite{GSW2}. Their work implies that if $k\otimes U(k)\in \trip(\calH)$ for all $k\geq 1$, \emph{ and $\calH$ has finite $\VC_2$-dimension}, then in the regime where $\e_2$ is polynomial, $T_{\calH}(\e_1,\e_2)$ is bounded above and below by tower type functions in $\e_1^{-1}$.  On the other hand, the work of Moshkovitz and Shapira \cite{Moshkovitz.2019} implies that when $\calH$ has infinite $\VC_2$-dimension, there are polynomial $\e_2$ for which $T_{\calH}(\e_1,\e_2)$ is bounded below by a wowzer type function.  Interestingly, this means that when restricting to polynomial $\e_2$,  there is an additional ``jump" in the growth rates for $T_{\calH}$, characterized by $\VC_2$-dimension, which does not exist when arbitrary $\e_2$ are allowed (as Theorem \ref{thm:strong1} shows).  

There also remain many open problems about range (1) with regards to Problem \ref{prob:main}, i.e., where arbitrary $\e_2$ are allowed. It seems likely that in this range $T_{\calH}(\e_1,\e_2)$ will depend more fundamentally on the growth rate of $\e_2$. 

\subsection{Relationship to Moshkovitz-Shapira}

As mentioned early in the introduction, the first lower bound construction for $T(\e_1,\e_2)$ of Theorem \ref{thm:reg2} is due to Moshkovitz and Shapira \cite{Moshkovitz.2019,MoshkovitzSimple}. The hypergraph $N\otimes\bip(G')$ used in Theorem \ref{thm:main0}  is simpler than the one constructed by Moshkovitz and Shapira in \cite{Moshkovitz.2019,MoshkovitzSimple}.  However, the Moshkovitz-Shapira example was shown in \cite{Moshkovitz.2019,MoshkovitzSimple} to require wowzer type bounds for a version of hypergraph (called $\langle \delta \rangle$-regularity) which is weaker than the notion of regularity that that used in Theorem \ref{thm:main0}.  Specifically they proved the following (see Section \ref{sec:weakreg} for definitions).

\begin{theorem}[Moshkovitz-Shapria \cite{MoshkovitzSimple}]\label{thm:MS}
For all $s\geq 1$ there exists a $3$-partite $3$-uniform hypergraph $H$ with edge density at least $2^{-s-3}$ and a partition $\calV_0$ of $V(H)$ of size at most $2^{300}$ so that if $\calP$ is a $\langle 2^{-73}\rangle$-regular partition with vertex partition $\calP_1\preceq \calV_0$, then $|\calP_1|\geq W(s)$.
\end{theorem}

The additional complexity in the construction of Moshkovitz and Shapira is necessary for the conclusion they proved in Theorem \ref{thm:MS}.  In particular,  we show the hypergraph from Theorem \ref{thm:main0} permits $\langle \delta \rangle$-regular partitions of much smaller size.

\begin{theorem}\label{thm:weakregUB}
Let $0<\delta<\rho<1/8$. Assume $H=(V_1\cup V_2\cup V_3,F)$ is a $3$-partite $3$-uniform hypergraph satisfying  $\rho=|E(H)|/|V_1||V_2||V_3|$, and so that $H=N\otimes G$ for some bipartite graph $G$.  Then for any partition $\calV_0$ of $V(H)$ of size at most $2^{300}$, there exists a $\langle \delta\rangle$-regular partition of $H$ with vertex partition $\calP_1\preceq \calV_0$ and with $|\calP_1|\leq O(Tw(\delta^{-1000}))$.
\end{theorem}

This suggests the question of characterizing the possible growth of $\langle \delta \rangle$-regular partitions in hereditary properties is interesting, as it would likely give a different stratification than Theorem \ref{thm:strong1}.

\subsection{Outline} We give here an outline of the rest of the paper.   In Section \ref{ss:notation} we set out basic notation used throughout the paper. In Section \ref{sec:regularity} we define the type of strong regularity for $3$-uniform hypergraphs which is the main topic of this paper.  In this section we also state the relevant regularity lemma and corresponding counting lemma, both due to Gowers.  In Section \ref{sec:hom}, we introduce preliminaries related to homogeneous decompositions, homogeneous partitions, and generalizations of VC-dimension to $3$-uniform hypergraphs.  In Section \ref{sec:strongvert}, we prove Theorem \ref{thm:ingredient1}, which reduces the ``slow" growth rates in Theorem \ref{thm:strong1} to   Theorem \ref{thm:weak}.  In Section \ref{sec:wowzer}, we prove Theorem \ref{thm:main0}, which allows us to finish the proof of Theorem \ref{thm:strong1}.    In Section \ref{sec:weakreg}, we prove Theorem \ref{thm:weakregUB} which shows the example constructed to prove Theorem \ref{thm:main0} is not complicated enough to produce wowzer lower bounds for the $\langle \delta\rangle$-regularity of \cite{MoshkovitzSimple}.

\subsection{Acknowledgements} The connection between the growth of $T_{\calH}$ and strong graph regularity became apparent to the author as a result of joint work with Julia Wolf. The author would like to thank Asaf Shapira for pointing out subtleties the author previously overlooked regarding the bounds in Theorem \ref{thm:reg2}, and  Jacob Fox for pointing out a difference between the notion of regularity used here and that appearing in \cite{Fox.2014}. Finally, the author thanks the anonymous referee who's comments prompted the more careful presentation here regarding polynomial vs. superpolynomial $\e_2$.

\section{Notation}\label{ss:notation}
Given a natural number $n\geq 1$, $[n]=\{1,\ldots, n\}$.  For real numbers $r_1,r_2$ and $\e>0$, we write $r_1=r_2\pm \e$ or $r_1\approx_{\e}r_2$ to mean $|r_1-r_2|\leq \e$.  An \emph{equipartition} of a set $V$ is a partition $V=V_1\cup \ldots \cup V_t$ with the property that for each $1\leq i,j\leq t$, $||V_i|-|V_j||\leq 1$. 

Given a set $V$ and $k\geq 1$, let ${V\choose k}=\{X\subseteq V: |X|=k\}$.  A \emph{$k$-uniform hypergraph} is a pair $(V,E)$ where $E\subseteq {V\choose k}$.  For a $k$-uniform hypergraph $G$, we let $V(G)$ denote the vertex set of $V$ and $E(G)$ denote the edge set of $G$. For $k\geq 3$, we will refer to $k$-uniform hypergraphs as \emph{$k$-graphs}. We refer to $2$-uniform hypergraphs as \emph{graphs}.  

Given a $k$-graph $G=(V,E)$ and a set $V'\subseteq V$, $G[V']:=(V', E\cap {V'\choose 2})$ is the \emph{induced sub-$k$-graph of $G$ with vertex set $V'$}. A \emph{hereditary $k$-graph property} is a class of finite $k$-graphs closed under induced sub-$k$-graphs and isomorphism.

Given integers $\ell,k\geq 1$, a $k$-graph $G=(V,E)$ is \emph{$\ell$-partite} if there is a partition $V=V_1\cup \ldots \cup V_{\ell}$ so that for all $e\in E$ and $1\leq i\leq \ell$, $|e\cap V_i|\leq 1$.  In this case, we write $G=(V_1\cup \ldots \cup V_{\ell},E)$ to denote that $G$ is $\ell$-partite with vertex partition given by $V_1\cup \ldots \cup V_{\ell}$.

Given distinct elements $x,y$, we will write $xy$ for the set $\{x,y\}$.   Given sets $X,Y$, we set 
\begin{align*}
K_2[X,Y]&=\{xy: x\in X, y\in Y, x\neq y\}.
\end{align*}
Suppose $G=(V,E)$ is a graph.   Given disjoint sets $X,Y\subseteq V$, we define
$$
G[X,Y]=(X\cup Y,E\cap K_2[X,Y]).
$$
In other words, $G[X,Y]$ is the bipartite graph with parts $X,Y$ induced by $G$.  For a graph $G=(V,E)$, we will  frequently refer to the set of ordered pairs coming from edges of $G$.  For this reason we define the following notation.
$$
\overline{E}=\{(x,y)\in V^2: xy\in E\}.
$$
Given $X,Y\subseteq V$, the \emph{density of $(X,Y)$ in $G$} is
$$
d_G(X,Y)=|\overline{E}\cap (X\times Y)|/|X||Y|.
$$ 
Note that if $X$ and $Y$ are disjoint, then $d_G(X,Y)=|E\cap K_2[X,Y]|/|X||Y|$.  Given $x\in V$, the \emph{neighborhood of $x$ in $G$} is $N_G(x)=\{y\in V: xy\in E\}$.  

We will use similar neighborhood notation in the following more general contexts.  Suppose $X$ is a set and $F\subseteq {X\choose 2}$. Given $x\in X$, we write $N_F(x)=\{y\in X: xy\in F\}$. Note that with the  notation we have defined, we could write the neighborhood of a vertex in a graph $G$ as either $N_G(x)$ or $N_E(x)$ where $E$ is the edge set of $G$.  Similarly, for $F\subseteq X\times Y$ and $x\in X$, write $N_F(x)=\{y\in V: (x,y)\in F\}$.

We now set similar notation for $3$-graphs.  To begin with, for three distinct elements $x,y,z$, we will write $xyz$ for the set $\{x,y,z\}$. Given sets $X,Y,Z$, we set 
\begin{align*}
K_3[X,Y,Z]&=\{xyz: x\in X, y\in Y, z\in Z, x\neq y, y\neq z, x\neq z\}.
\end{align*}
Suppose $G=(V,E)$ is a $3$-graph.  We define
$$
\overline{E}=\{(x,y,z)\in V^3: xy\in E\}.
$$
For sets $X,Y,Z\subseteq V$, the \emph{density of $(X,Y,Z)$ in $G$} is 
$$
d_G(X,Y,Z)=|\overline{E}\cap (X\times Y\times Z)|/|X||Y||Z|.
$$ 
For disjoint subsets $X,Y,Z\subseteq V$, we let $G[X,Y,Z]$ be the tripartite $3$-graph 
$$
(X\cup Y\cup Z, E\cap K_3[X,Y,Z]).
$$
 Given $x,y\in V$, we define neighborhoods 
 $$
 N_G(x)=\{uv\in V: xuv\in E\}\text{ and }N_G(xy)=\{v\in V: xyv\in E\}.
 $$

We will use similar notation in the following more general contexts.  Suppose $X,Y,Z$ are sets.  Given a set $F\subseteq {X\choose 3}$ and $x,y\in X$, we write $N_F(x)=\{yz\in X: xyz\in F\}$ and $N_F(xy)=\{z\in V: xyz\in F\}$.  With the notation we have defined, we could write neighborhoods in a $3$-graph $G$ as either $N_G(x)$, $N_G(xy)$ or $N_E(x)$, $N_E(xy)$, where $E$ is the edge set of $G$.  Similarly, for $F\subseteq X\times Y\times Z$ and $x\in X$, $y\in Y$, write 
\begin{align*}
N_F(x)=\{(y,z)\in V: (x,y,z)\in F\}\text{ and }N_F(x,y)=\{z\in Z: (x,y,z)\in F\}.
\end{align*}
  
  We use the following definitions to understand the growth of certain bounds.

\begin{definition}[Tower Function]\label{def:tower}
Given $m\in \mathbb{R}^{>0}$, define $T_m:\mathbb{N}\rightarrow \mathbb{N}$ by setting $T_m(1)=m$ and for all $x\geq 1$, $T_m(x+1)=T_m(x)2^{T_m(x)}$.
\end{definition}

\begin{definition}[Wowzer Function]\label{def:wowzer}
Define $W:\mathbb{N}\rightarrow \mathbb{N}$ by setting $W(1)=1$ and for all $x\geq 1$, $W(x+1)=T_2(W(x))$.  
\end{definition}

\section{Regularity for $3$-graphs}\label{sec:regularity}

In this section we give background on the type of regularity for $3$-graphs which is the topic of this paper.  The notion we consider was developed by Gowers \cite{Gowers.2007,Gowers.20063gk}.  It is equivalent to several other notions of hypergraph quasirandomness.  We refer the reader to \cite{Nagle.2013} for more history on this topic. 

Our particular presentation here is an adaptation of \cite{Terry.2022}, which in turn was based on \cite{Gowers.2007,Gowers.20063gk, Frankl.2002,Nagle.2013}.  We will deviate slightly from \cite{Terry.2022}  for the following reason.  For the purposes of our theorems, it is important that our definitions be meaningful for ``regular decompositions" with only one vertex part.  For example, if $\calH$ is the hereditary $3$-graph property consisting of all finite $3$-graphs with no edges, the ``correct" value for $T_{\calH}(\e_1,\e_2)$ should always be $1$.  In the set-up used in \cite{Terry.2022}, regular decompositions with one vertex part do not impart any information about a given hypergraph.  For this reason, we work with what we will refer to as ``bigraphs" (see the next subsection) whereas in \cite{Terry.2022} we worked with bipartite graphs.  This changes slightly the other notions involved, but the differences are largely cosmetic. 

In Subsection \ref{ss:bigraphs} we introduce the building blocks of regular decompositions, namely bigraphs, triads, and trigraphs.  In Subsection \ref{ss:qr}, we give the definition of a form of quasirandomness for $3$-graphs and a corresponding counting lemma, both due to Gowers. In Subsection \ref{ss:dec} we define $\dev_{2,3}$-regular decompositions and state the corresponding regularity lemma.

\subsection{Bigraphs,  Trigraphs, and Triads}\label{ss:bigraphs}

We begin with the definition of what we will call a bigraph, which is an ordered analogue of a bipartite graph.

\begin{definition}\label{def:bigraph}
A \emph{bigraph} is a tuple $G=(V_1,V_2;E)$ where $V_1,V_2$ are vertex sets and $E\subseteq V_1\times V_2$. 
\end{definition}

Note $V_1$ and $V_2$ in Definition \ref{def:bigraph} need not be disjoint.  Given a bigraph $G=(V_1,V_2; E)$, we call $(V_1,V_2)$ the \emph{vertex sets} of $G$ and we call $E$ the \emph{edge set} of $G$, denoted $E(G)$.  Every graph naturally corresponds to a bigraph. 

\begin{definition}\label{def:gbar}
Suppose $G=(V,E)$ is a graph.  Define $\overline{G}$ to be the bigraph $(V,V; \overline{E})$.
\end{definition}

Note that if $(V_1,V_2; E)$ is a bigraph where $V_1,V_2$ are disjoint, we naturally obtain a bipartite graph $(V_1\cup V_2, \{xy: (x,y)\in E\})$.   In this way, there is a correspondence between bipartite graphs and bigraphs with disjoint vertex sets.  In the next subsection, will use bigraphs in our definitions of regular decompositions instead of bipartite graphs (which are used in \cite{Terry.2022}).  We do this because bigraphs allow the vertex sets to possibly overlap.  We now define the density of a bigraph.

\begin{definition}\label{def:dens2}
Given a bigraph $G=(V_1,V_2; E)$ and $X\subseteq V_1$ and $Y\subseteq V_2$, define 
$$
d_G(X,Y)=\frac{|E\cap (X\times Y)|}{|X||Y|}.
$$
The \emph{density of $G$} is $d_G(V_1,V_2)$.
\end{definition}

Note that if $G=(V_1\cup V_2, E)$ is a bipartite graph and $X\subseteq V_1$ and $Y\subseteq V_2$, then $d_G(X,Y)$ (as defined in Subsection \ref{ss:notation}), agrees with the density $d_{\overline{G}}(X,Y)$ in the bigraph $\overline{G}$. 

In analogy to Definition \ref{def:bigraph}, we define a trigraph to be an analogue of a tripartite $3$-graph.

\begin{definition}
An \emph{trigraph} is a tuple $(X,Y,Z; E)$ where $E\subseteq X\times Y\times Z$.
\end{definition}

Given a trigraph $G=(V_1,V_2,V_3; E)$, we call $(V_1,V_2,V_3)$ the \emph{vertex sets} of $G$ and we call $E$ the \emph{edge set} of $G$, denoted $E(G)$. In analogy to Definition \ref{def:gbar}, $3$-graphs naturally give rise to trigraphs.  

\begin{definition}\label{def:hbar}
Suppose $G=(V,E)$ is a $3$-graph.  Define $\overline{G}$ to be the trigraph $(V,V,V; \overline{E})$.
\end{definition}

It is not hard to see there is a correspondence between tripartite $3$-graphs and trigraphs with disjoint vertex sets.  In the next subsection, trigraphs will appear in place of tripartite $3$-graphs in contexts where we want to allow vertex sets to overlap.  We now give an analogue of Definition \ref{def:dens2}.

\begin{definition}\label{def:dens3}
Given a trigraph $H=(V_1,V_2,V_3; E)$ and $X\subseteq V_1$, $Y\subseteq V_2$, and $Z\subseteq V_3$, define 
$$
d_H(X,Y,Z)=\frac{|E\cap (X\times Y\times Z)|}{|X||Y||Z|}.
$$
\end{definition}
Note that for a tripartite $3$-graph $H$, the density $d_H(X,Y,Z)$ defined in Subsection \ref{ss:notation} aggrees with $d_{\overline{H}}(X,Y,Z)$ from Definition \ref{def:dens3}.  Definition \ref{def:dens3} will be used heavily in Sections \ref{sec:hom} and \ref{sec:strongvert}, which are related to weak hypergraph regularity.  For the stronger regularity which is the main topic of the paper, we will be more interested in the density of a trigraph relative to the triangles formed from the edges of distinguished bigraphs.  This leads us to the notion of a \emph{triad}, which the reader should think of as an ordered analogue of a \emph{tripartite graph}.

\begin{definition}\label{def:triad}
A \emph{triad} is a tuple $G=(X,Y,Z; E_{XY},E_{YZ}, E_{XZ})$ where $(X,Y;E_{XY})$, $(X,Z; E_{XZ})$, and $(Y,Z;E_{YZ})$ are bigraphs.
\end{definition}

We refer to the tuple $(X,Y,Z)$ as the \emph{vertex sets} of $G$.  There is clearly a correspondence between tripartite graphs and triads with pairwise disjoint vertex sets. In the next section, we use the triads of Definition \ref{def:triad} instead of tripartite graphs (which are used in \cite{Terry.2022}) to allow vertex sets to overlap.  

We now give some definitions related to triads.  Suppose $G=(X,Y,Z; E_{XY},E_{YZ}, E_{XZ})$ is a triad.  The \emph{set of ordered triangles in $G$} is 
$$
K_3(G):=\{(x,y,z)\in X\times Y\times Z: xy\in E_{XY}, yz\in E_{YZ},xz\in E_{XZ}\}.
$$
We also define the \emph{component bigraphs of $G$} to be
$$
G[X,Y]:=(X,Y;E_{XY}),\text{ }G[X,Z]:=(X,Z;E_{XZ}),\text{ and }G[Y,Z]:=(Y,Z;E_{YZ}).
$$ 

The version of regularity we are interested in requires us to consider edges of trigraphs  relative to triangles of  triads.  For this reason, we now define what it means for a triad to underly a trigraph.  

\begin{definition}
Suppose $H=(X,Y,Z;R)$ is a trigraph and $G=(X,Y,Z; E_{XY},E_{YZ}, E_{XZ})$ is a triad. We say $G$ \emph{underlies $H$} if $R\subseteq K_3(G)$. 
\end{definition}

In other words, $G$ underlies $H$ if all ternary edges of $H$ form ordered triangles in $G$.    We will also use the following notation, which restricts a trigraph to the ordered triangles of a triad.

\begin{definition}\label{def:1}
Suppose $H=(V_1,V_2,V_3;R)$ is a trigraph, $X_1\subseteq V_1$, $X_2\subseteq V_2$, $X_3\subseteq V_3$,  and $G$ is a triad with vertex sets $(X_1,X_2,X_3)$.  Define $H|G$ to be the following trigraph.
$$
H|G:=(X_1,X_2,X_3;R\cap K_3(G)).
$$
\end{definition}

Note that in the notation of Definition \ref{def:1},  $G$ always underlies $H|G$.   We now define a notion of density for a trigraph relative to a triad.
 
 \begin{definition}\label{def:dens4}
 Suppose $H=(V_1,V_2,V_3;R)$ is a trigraph, $X_1\subseteq V_1$, $X_2\subseteq V_2$, $X_3\subseteq V_3$,  and $G$ is a triad with vertex sets $(X_1,X_2,X_3)$.  Define
 $$
 d_H(G)=|R\cap K_3(G)|/|K_3(G)|.
 $$
 \end{definition}
 
We will frequently need to apply the definitions above in the context of a fixed $3$-graph $H$. We do this by simply considering these notions applied to the trigraph $\overline{H}$ associated to $H$ (see Definition \ref{def:hbar}).  For example, we define the following notion of density for a $3$-graph relative to a triad.
 
  \begin{definition}\label{def:dens5}
 Suppose $H=(V,E)$ is a  $3$-graph, $X_1,X_2,X_3\subseteq V$, and $G$ is a triad with vertex sets $(X_1,X_2,X_3)$.  We then define $d_H(G)$ to be $d_{\overline{H}}(G)$ (from Defintion \ref{def:dens3} applied to $\overline{H}$).
 \end{definition}
 
We now set some notation for restricting the vertex sets of triads and trigraphs.  Given a triad $G=(A,B,C; E_{AB},E_{BC},E_{AC})$ and $A'\subseteq A$, $B'\subseteq B$, and $C'\subseteq C$, we define a triad
$$
G[A',B',C']=(A',B',C'; E'_{AB},E'_{BC},E'_{AC}),
$$
 where $E_{AB}'=E_{AB}\cap (A\times B)$, $E_{AC}'=E_{AC}\cap (A\times C)$, and $E_{BC}'=E_{BC}\cap (B\times C)$.  Similarly, given a trigraph $H=(A,B,C;F)$, we define 
 $$
 H[A',B',C']=(A',B',C'; F\cap (A'\times B'\times C')).
 $$
 Observe that if $H$ is underlied by $G$, then $H[A',B',C']$ is underlied by $G[A',B',C']$.

\subsection{$\dev_{2,3}$-quasirandomness}\label{ss:qr}

This section defines the notions of quasirandomness which appear in the type of strong regularity lemma we are interested in.  We begin with a definition of  quasirandomenss for bigraphs, adapted from \cite{Gowers.20063gk}.

\begin{definition}\label{def:dev2}
Suppose $B=(U,W; E)$ is a bigraph and $|E|=d_B|U||W|$.  We say $B$ \emph{has $\dev_2(\e,d)$} if $d_B=d\pm \e$ and 
$$
\sum_{u_0,u_1\in U}\sum_{w_0,w_1\in W}\prod_{i\in \{0,1\}}\prod_{j\in \{0,1\}}g_B(u_i,v_j)\leq \e |U|^2|V|^2,
$$
where $g_B(u,v)=1-d_B$ if $uv\in E$ and $g_B(u,v)=-d_B$ if $uv\notin E$. 

We say $B$ simply \emph{has $\dev_2(\e)$} if it has $\dev_2(\e, d_B)$.
\end{definition}

This definition is roughly equivalent to that of graph regularity (see Section \ref{ss:CF} for more details).    A crucial fact about the quasi-randomness of Definition \ref{def:dev2} is the counting lemma. The version we state below counts the number of ordered triangles in a triad with quasirandom component bigraphs. This statement follows immediately from Lemma 3.4 of  \cite{Gowers.20063gk}.

\begin{proposition}[Counting Lemma]\label{prop:counting}
Let $\e,d_{AB},d_{AC},d_{BC}>0$. Suppose we have a triad $G=(A,B,C; E_{AB},E_{AC}, E_{BC})$ such that $G[A,B]$, $G[B,C]$ and $G[A,C]$ have $\dev_2(\e, d_{AB})$, $\dev_2(\e, d_{BC})$, and $\dev_2(\e, d_{AC})$, respectively. Then 
$$
\Big| |K_3(G)|- d_{AB}d_{BC}d_{AC}|A||B||C||\Big|\leq 4\e^{1/4}|A||B||C|.
$$
\end{proposition}

We now define a notion of quasirandomness for a trigraph relative to an underlying triad. This definition is  due to Gowers \cite{Gowers.20063gk}. 

 \begin{definition}\label{def:regtriad}
Let $\e_1,\e_2>0$.  Assume $H=(X,Y,Z;E)$ is a trigraph and $G$ is a triad underlying $H$.  

We say that \emph{$(H,G)$ has $\dev_{2,3}(\e_1,\e_2)$} if there are $d_{XY},d_{YZ},d_{XZ}>0$ such that $G[X,Y]$, $G[X,Z]$, and $G[Y,Z]$ have $\dev_2(\e_2,d_{XY})$,  $\dev_2(\e_2,d_{XZ})$, and  $\dev_2(\e_2,d_{YZ})$ respectively, and
$$
\sum_{u_0,u_1\in X}\sum_{w_0,w_1\in Y}\sum_{z_0,z_1\in Z}\prod_{(i,j,k)\in \{0,1\}^3}h_{H,G}(u_i,w_j,z_k)\leq \e_1 d_{XY}^4d_{YZ}^4d_{XZ}^4|X|^2|Y|^2|Z|^2,
$$
where $h_{H,G}(x,y,z)=1-d_H(G)$ if $(x,y,z)\in E\cap K_3(G)$, $h_{H,G}(x,y,z)=-d_H(G)$ if $(x,y,z)\in K_3(G)\setminus E$, and $h_{H,G}(x,y,z)=0$ if $(x,y,z)\notin K_3(G)$.
\end{definition}

We now state a version of the counting lemma for $\dev_{2,3}$-quasirandomenss, due to Gowers (see Theorem 6.8 of \cite{Gowers.20063gk}). This theorem allows us to find configurations in $3$-graphs, given a sufficient amount of $\dev_{2,3}$-regularity.
  
\begin{proposition}\label{prop:countinggowers} 
For all $t\in \mathbb{N}$ there exist $K,C,L\geq 1$ so that for all $0<\e_2,d_2,\e_1, d_1<1$ satisfying $0<\e_2<\e_1^Kd_2^{t\choose 2}/C$,  the following holds.  

Let $F=([t],R_F)$ and $H=(V,R)$ be $3$-graphs. Suppose $V_1,\ldots, V_t$ are subsets of $V$, and for each $1\leq i,j\leq t$, assume $G_{ij}=(V_i,V_j;E_{ij})$ is a bigraph with density $d_{ij}\geq d_2$.  For each $1\leq i,j,k\leq t$, let $G_{ijk}=(V_i,V_j,V_k; E_{ij},E_{jk},E_{ik})$, let $H^{ijk}=\overline{H}|G_{ijk}$, and let $d_{ijk}=d_{H^{ijk}}(G^{ijk})$.  Suppose the following hold.
\begin{enumerate}
\item For each $1\leq i,j\leq t$, $G_{ij}$ has $\dev_2(\e_2)$,
\item For each $ijk\in R_F$, $d_{ijk}\geq d_1$, and for each  $ijk\in {[t]\choose 3}\setminus R_F$, $d_{ijk}\leq 1-d_1$.
\item For each $1\leq i,j,k\leq t$, $(H^{ijk},G^{ijk})$ satisfies $\dev_{2,3}(\e_1,\e_2)$,
\end{enumerate}
Let $m_F$ be the number of tuples $(v_1,\ldots, v_t)\in \prod_{i=1}^tV_i$ such that $(v_i,v_j,v_k)\in R$ if and only if $ijk\in R_F$.  Then  
$$
\Big|m_F - (\prod_{ijk\in {[t]\choose 3}}d_{ijk})(\prod_{(i,j)\in {[t]\choose 2}}d_{ij})\prod_{i=1}^t|V_i|\Big|\leq \e_1^{1/L}(\prod_{(i,j)\in {[t]\choose 2}}d_{ij})\prod_{i=1}^t|V_i|.
$$
\end{proposition}

We will frequently use the following embedding lemma, which is a corollary of Proposition \ref{prop:countinggowers}.

\begin{corollary}\label{cor:countingcor}
For all $t\geq 1$, there are $D\geq 1$, and a polynomial  $p(x,y)$ so that for all $0<\e_1,d_1,d_2,\e_2<1$ satisfying  $\e_2<p(\e_1,d_2)$ and $\e_1<d_1^D$, there is $n_0$ so that the following holds. 

Let $F=([t],R_F)$ and $H=(V,R)$ be $3$-graphs. Suppose $V_1,\ldots, V_t$ are subsets of $V$, each of size at least $n_0$ and for each $1\leq i,j\leq t$, assume $G_{ij}=(V_i,V_j;E_{ij})$ is a bigraph with density $d_{ij}\geq d_2$.  For each $1\leq i,j,k\leq t$, let $G_{ijk}=(V_i,V_j,V_k; E_{ij},E_{jk},E_{ik})$, let $H^{ijk}=\overline{H}|G_{ijk}$, and let $d_{ijk}=d_{H^{ijk}}(G^{ijk})$.  Suppose the following hold.
\begin{enumerate}
\item For each $1\leq i,j\leq t$, $G_{ij}$ has $\dev_2(\e_2)$,
\item For each $ijk\in R_F$, $d_{ijk}\geq d_1$, and for each  $ijk\in {[t]\choose 3}\setminus R_F$, $d_{ijk}\leq 1-d_1$.
\item For each $1\leq i,j,k\leq t$, $(H^{ijk},G^{ijk})$ satisfies $\dev_{2,3}(\e_1,\e_2)$,
\end{enumerate}
Then there exists a tuple $(v_1,\ldots, v_t)\in \prod_{i=1}^tV_i$ such that $(v_i,v_j,v_k)\in R$ if and only if $ijk\in R_F$.
\end{corollary}
\begin{proof}
Let $K,C,L$ be as in  Proposition \ref{prop:countinggowers} for $t$.  Let $p(x,y)=\frac{x^Ky^{t\choose 2}}{C}$ and set $D={t\choose 3}L$.   Assume $0<\e_1,d_1,d_2,\e_2<1$ satisfy  $\e_2<p(\e_1,d_2)$ and $\e_1<d_1^D$.  Let $n_0$ be sufficiently large.  

Suppose we have $F$, $H$, $V_1,\ldots, V_t$, and $G_{ij}$, $d_{ij}$, $H_{ijk}$, $G_{ijk}$, and $d_{ijk}$ as in the hypotheses.  Let $m_F$ be the number of tuples $(v_1,\ldots, v_t)\in \prod_{i=1}^tV_i$ such that $(v_i,v_j,v_k)\in R$ if and only if $ijk\in R_F$.  By Proposition \ref{prop:countinggowers} 
$$
m_F \geq \Big((\prod_{ijk\in {[t]\choose 3}}d_{ijk})-\e_1^{1/L}\Big)(\prod_{(i,j)\in {[t]\choose 2}}d_{ij})\prod_{i=1}^t|V_i|\geq (d_1^{t\choose 3}-\e_1^{1/L})d_2^{t\choose 2}n_0^t.
$$
Since $\e_1<d_1^D$, $d_1^{t\choose 3}-\e_1^{1/L}>0$. Thus, since $n_0$ is sufficiently large, the right hand side of the inequality above is at least $1$. 
\end{proof}

\subsection{Regular Decompositions}\label{ss:dec}

In this subsection, we define regular decompositions and state a regularity lemma for $\dev_{2,3}$-quasirandomness.  We begin with the definition of a $(t,\ell)$-decomposition of a vertex set $V$, which will be the analogue of the vertex partition in the graph regularity lemma.  A $(t,\ell)$-decomposition will roughly speaking partition not only the vertex set, but also pairs of vertices.

\begin{definition}\label{def:decomp}
Let $V$ be a vertex set and $t,\ell \in \mathbb{N}$. A \emph{$(t,\ell)$-decomposition} $\calP$ for $V$ consists of a pair $(\calP_1,\calP_2)$, where $\calP_1=\{V_1\cup \ldots \cup V_t\}$ is a partition of $V$, and where 
$$
\calP_2=\{P_{ij}^{\alpha}: 1\leq i,j\leq t, 1\leq \alpha\leq \ell \}
$$
satisfies the following. For each $1\leq i, j\leq t$, $P_{ij}^1,\ldots, P_{ij}^{\ell}$ are disjoint sets such that $V_i\times V_j=P^1_{ij}\cup \ldots \cup P^{\ell}_{ij}$.  
\end{definition}
We note that Definition \ref{def:decomp} differs slightly from the definition of a $(t,\ell)$-decomposition used in \cite{Terry.2022}. One difference is that, in the notation of Definition \ref{def:decomp}, some of the sets $P_{ij}^{\alpha}$ may be empty.  We do this mainly for notational convenience.  A less trivial difference is that here, $\calP_2$ contains sets $P_{ij}^{\alpha}$ where $i=j$.

The fundamental unit of a $(t,\ell)$-decomposition $\calP$ is a \emph{triad of $\calP$}, meaning a triad of the form 
$$
G_{ijk}^{\alpha\beta\gamma}:=(V_i,V_j,V_k; P_{ij}^\alpha,P_{ik}^\beta, P_{jk}^\gamma),
$$
 where $1\leq i,j,k\leq t$, and $\alpha,\beta,\gamma\in [\ell]$.  We say a triad $G_{ijk}^{\alpha\beta\gamma}$ is \emph{non-empty} if $K_3(G_{ijk}^{\alpha\beta\gamma})\neq \emptyset$.  We then let $\triads(\calP)$ denote the set of all non-empty triads of $\calP$.  Observe that if $\calP$ is a $(t,\ell)$-decomposition of a set $V$, then we have a partition 
 $$
V^3=\bigcup_{G\in \triads(\calP)}K_3(G).
 $$

We now define the notion of regular triad, relative to a $3$-graph, by combining Definition \ref{def:regtriad} with Definition \ref{def:hbar}.

\begin{definition}\label{def:regtriads}
Given a $3$-graph $H=(V,R)$, a $(t,\ell)$-decomposition $\calP$ of $V$, and $G\in \triads(\calP)$, we say $G$ \emph{has $\dev_{2,3}(\e_1,\e_2)$ with respect to $H$} if $(\overline{H}|G,G)$ has $\dev_{2,3}(\e_1,\e_2)$. 
\end{definition}

  To define a regular decomposition for a $3$-graph, we need one more definition.

\begin{definition}
Suppose $\calP=(\calP_1,\calP_2)$ is a $(t,\ell)$-decomposition with $\calP_1=\{V_1,\ldots, V_t\}$ and $\calP_2=\{P_{ij}^{\alpha}: 1\leq i,j\leq t, 1\leq \alpha\leq \ell\}$.  We say $\calP$ is a \emph{$(t,\ell, \e_1,\e_2)$-decomposition of $V$} if 
$$
|\bigcup_{P\in \Sigma}P|\geq (1-\e_1)|V|^2,
$$
where $\Sigma$ is the set of $P_{ij}^{\alpha}\in \calP_2$ such that $(V_i,V_j; P_{ij}^{\alpha})$ has $\dev_{2}(\e_2)$.
\end{definition}

We now define regular decompositions of $3$-graphs.

\begin{definition}\label{def:regdec}
Suppose $H=(V,E)$ is a $3$-graph and $\calP$ is an $(t,\ell, \e_1,\e_2)$-decomposition of $V$.  We say that $\calP$ is \emph{$\dev_{2,3}(\e_1,\e_2)$-regular} with respect to $H$ if 
$$
|\bigcup_{G\in \Omega}K_3(G)|\geq (1-\e_1)|V|^3,
$$
where $\Omega$ is the set of  $G\in \triads(\calP)$ satisfying $\dev_{2,3}(\e_1,\e_2)$ with respect to $H$.   
\end{definition}

We emphasize that in this paper we make no insistence on equitability conditions for our regular decompositions. We can now state a version of  the regularity lemma for $\dev_{2,3}$-quasirandomness, which was first proved by Gowers in \cite{Gowers.20063gk}.

\begin{theorem}\label{thm:reg2} For all $\e_1>0$, every function $\e_2:\mathbb{N}\rightarrow (0,1]$,   there exist positive integers $T=T(\e_1,\e_2)$ and $L=L(\e_1,\e_2)$, such that for every sufficiently large $3$-graph $H=(V,E)$, there exists a  $\dev_{2,3}(\e_1,\e_2(\ell))$-regular, $(t,\ell,\e_1,\e_2(\ell))$-decomposition $\calP$ for $H$ with $1\leq t\leq T$ and $1\leq \ell \leq L$.
\end{theorem}

For the sake of completeness, we have included a proof of Theorem \ref{thm:reg2} with explicit bounds in the appendix. Since our decompositions may not be equitable in this paper, we frequently need to restrict our attention to only those parts of a decomposition which are sufficiently large.  This is the motivation for the following terminology.

\begin{definition}\label{def:nontrivialtriad}
Suppose $\calP$ is a $(t,\ell, \e_1,\e_2)$-decomposition of $V$.  We say a triad $G=(V_i,V_j,V_k; P_{ij}^{\alpha},P_{ik}^{\beta},P_{jk}^{\gamma})$ of $\calP$ is \emph{$\mu$-non-trivial} if the following hold.
\begin{enumerate}
\item $\min\{|V_i|,|V_j|,|V_k|\}\geq \mu |V|/t$,
\item $|P_{ij}^{\alpha}|\geq \mu|V_i||V_j|/\ell$, $|P_{ik}^{\beta}|\geq \mu|V_i||V_k|/\ell$, and $|P_{jk}^{\gamma}|\geq \mu|V_j||V_k|/\ell$.
\end{enumerate}
\end{definition}

It is not difficult to show that for any decomposition $\calP$ of a set $V$, most triples from $V^3$ come from a non-trivial triad of $\calP$.

\begin{lemma}\label{lem:nontrivialtriad}
Suppose $\calP$ is a $(t,\ell, \e_1,\e_2)$-decomposition of $V$, and $\Omega$ is the set of $\mu$-non-trivial triads of $\calP$.  Then $|\bigcup_{G\in \Omega}K_3(G)|\geq (1-2\mu)|V|^3$.
\end{lemma}
\begin{proof}
Say $\calP_1=\{V_1,\ldots, V_t\}$ and $\calP_2=\{P_{ij}^{\alpha}: 1\leq i,j\leq t, 1\leq \alpha\leq \ell\}$.  Set
\begin{align*}
\calR_1&=\{X\in \calP_1: |X|<\mu|V|/t\}\text{ and }\\
\calR_2&=\{P_{ij}^{\alpha}\in \calP_2: |P_{ij}^{\alpha}|<\mu|V_i||V_j|/t\}.
\end{align*}  
Then 
\begin{align*}
\Big|\bigcup_{G\in \Omega}K_3(G)\Big|&\leq \sum_{X\in \calR_1}|X||V^2|+\sum_{V_i,V_j\in \calP_1}\Big(\sum_{\{\alpha\in [\ell]: P_{ij}^{\alpha}\in \calR_2\}}|P_{ij}^{\alpha}||V| \Big)\\
&\leq |\calR_1|\frac{\mu}{t} |V|^3+\sum_{V_i,V_j\in \calP_1}|\{\alpha\in [\ell]: P_{ij}^{\alpha}\in \calR_2\}|\frac{\mu}{\ell}|V_i||V_j||V|\\
&\leq \mu|V|^3+\sum_{V_i,V_j\in \calP_1}\mu|V_i||V_j||V|\\
&=2\mu|V|^3.
\end{align*}
\end{proof}

\section{Homogeneity and analogues of $\VC$-dimension in $3$-graphs}\label{sec:hom}

In this section we discuss homogeneous partitions of various kinds, and their relationship to generalizations of VC-dimension in the hypergraph setting.  In Subsection \ref{ss:hom}, we define homogeneous decompositions and explain their connection to $\VC_2$-dimension.  In Subsection \ref{ss:homsl}, we define homogeneous partitions and explain their connection to slicewise VC-dimension.

\subsection{Homogeneous Decompositions}\label{ss:hom}

In this subsection we discuss homogeneous $(t,\ell)$-decompositions and $\VC_2$-dimension.  We begin by observing that if a trigraph $H$ is underlied by a triad $G$, and $d_H(G)$ is close to $0$ or $1$, then $(H,G)$ will be regular in the sense of Definition \ref{def:regtriad}.

\begin{proposition}\label{prop:homimpliesrandome}
For all $0<\e<1/2$, $d_2>0$, and $0<\delta\leq (d_2/2)^{48}$, there is $N$ such that the following holds.  Suppose $H=(V_1,V_2,V_3;R)$ is a trigraph underlied  by a bigraph $G=(V_1,V_2,V_3;E_{12},E_{13},E_{23})$ so that for each $1\leq i<j\leq 3$, $G[V_i,V_j]$ has $\dev_2(\delta)$ and density at least $d_2$.  If $d_H(G)\in [0,\e)\cup (1-\e,1]$, then $(H,G)$ has $\dev_{2,3}(\delta,6\e)$.
\end{proposition}

For a proof, see Proposition 2.24 in  \cite{Terry.2022}.  Proposition \ref{prop:homimpliesrandome} tells us that having density close to $0$ or $1$ is a special way to ensure regularity in the sense of Definition \ref{def:regtriad}.  This leads us to the definition of a homogeneous decomposition, which means a decomposition where most triples are in a homogeneous triad. 

\begin{definition}\label{def:homdec}
Suppose $H=(V,E)$ is a $3$-graph with $|V|=n$ and $\mu>0$. Suppose $t,\ell\geq 1$ and $\calP$ is a $(t,\ell)$-decomposition of $V$.  
\begin{enumerate}
\item Given a triad $G\in \triads(\calP)$, we say $G$ is \emph{$\mu$-homogeneous for $H$} if 
$$
d_H(G)\in [0,\e)\cup (1-\e,1].
$$ 
\item We say that $\calP$ is \emph{$\mu$-homogeneous with respect to $H$} if 
$$
\Big|\bigcup_{G\in \Omega_{hom}}K_3(G)\Big|\geq (1-\mu)|V|^3,
$$
where $\Omega_{hom}=\{G\in \triads(\calP): G$ is $\mu$-homogeneous with respect to $H\}$.  
\end{enumerate}
\end{definition}

Proposition \ref{prop:homimpliesrandome} implies that for any $3$-graph $H$, any $(t,\ell,\e_1,\e_2(\ell))$-decomposition of $V(H)$ which is $\e_1$-homogeneous with respect to $H$ is automatically $\dev_{2,3}(\e_2(\ell),6\e_1)$-regular with respect to $H$.  Of course, the converse of this is false as there are certainly $3$-graphs $H$ which do not admit homogeneous decompositions.  On the level of hereditary properties, there is a combinatorial characterization of when this happens, and it has to do with the following analogue of VC-dimension for $3$-graphs.

\begin{definition}\label{def:vc2}
Suppose $H=(V,E)$ is a $3$-graph.  The \emph{$\VC_2$-dimension of $H$}, $\VC_2(H)$, is the largest integer $k$ so that there exist vertices $a_1,\ldots, a_k,b_1,\ldots, b_k\in V$ and $c_S\in V$ for each $S\subseteq [k]^2$, such that $a_ib_jc_S\in E$ if and only if $(i,j)\in S$.
\end{definition}

This definition is due to Shelah \cite{Shelah.2007, ShelahSD}, and has been studied in several papers in model theory and its connections to combinatorics \cite{Chernikov.2020,Hempel.2016,Chernikov.2019b, Terry.2018,Terry.2021a,Terry.2021b,Terry.2022, Aldaim.2023}.  Definition \ref{def:vc2} is closely related to homogenous decompositions.  First, non-trivial regular triads in $3$-graphs of small $\VC_2$-dimension have density close to $0$ or close to $1$ (see Proposition \ref{prop:suffvc2} below). This is a straightforward corollary of Gower's counting lemma, and has appeared in various forms in \cite{Terry.2021b, Terry.2022}.  We provide a sketch of the proof here for the sake of completeness, and because we will use similar ideas in other parts of the paper.

\begin{proposition}\label{prop:suffvc2}
For all $k\geq 1$, there is a polynomial $p(x,y)$ and $D\geq 1$  so that for all $0<\e_1<2^{-D}$, $\mu>0$, and $\e_2:\mathbb{N}\rightarrow (0,1]$ satisfying $\e_2(x)\leq p(\e_1,\mu x^{-1})$  the following holds.  Let $t,\ell\geq 1$, and suppose $H=(V,E)$ is a sufficiently large $3$-graph with $\VC_2$-dimension less than $k$.  Suppose $\calP$ is a $\dev_{2,3}(\e_1,\e_2(\ell))$-regular $(t,\ell, \e_1,\e_2(\ell))$-decomposition of $V$. Every $\mu$-non-trivial $G\in \triads(\calP)$ satisfying $\dev_{2,3}(\e_1,\e_2(\ell))$ with respect to $H$  is $\e_1^{1/D}$-homogeneous with respect to $H$.
\end{proposition}
\begin{proof}
Let $D$ and $p(x,y)$ be as in Corollary \ref{cor:countingcor} for $T=2k+2^{k^2}$. Fix $0<\e_1<2^{-D}$ and $\mu>0$, and suppose $\e_2:\mathbb{N}\rightarrow (0,1]$ satisfies $\e_2(x)\leq p(\e_1,\mu x^{-1})$.  

 Fix $t,\ell\geq 1$ and suppose $H=(V,E)$ is a sufficiently large $3$-uniform $3$-graph with $\VC_2$-dimension less than $k$. Let $\calP$ be a $\dev_{2,3}(\e_1,\e_2(\ell))$-regular $(t,\ell, \e_1,\e_2(\ell))$-decomposition for $H$.  Suppose $G=(V_i,V_j,V_s; P_{ij}^{\alpha},P_{ik}^{\beta},P_{jk}^{\gamma})$ is $\mu$-non-trivial $\dev_{2,3}(\e_1,\e_2(\ell))$-regular triad with respect to $H$, and suppose towards a contradiction $\e_1^{1/D}<d_H(G)<\e_1^{1/D}$.  Our next goal is to define an auxiliary bigraph $G'$ by duplicating vertices and edges from $G$. First, we define a new vertex set $V'$ as follows.  Set
$$
V'=A_1\cup \ldots A_k\cup B_1\cup \ldots \cup B_k\cup \bigcup_{S\subseteq [k]^2}C_S,
$$
 where for each $1\leq u\leq k$, $A_u$ is a copy of $V_i$ and $B_u$ is a copy of $V_j$, and for each $S\subseteq [k]^2$, $C_S$ is a copy of $V_s$.  We then define $G'$ to have vertex sets $(V',V')$ and so that the following hold. For each $1\leq u,v\leq k$, let $G'[A_u,B_v]$ be a copy of $G[V_i,V_j]$. For each $1\leq u\leq k$ and $S\subseteq [k]^2$, let $G'[A_u,C_S]$ be a copy of $G[V_i,V_S]$ and let $G'[B_u,C_S]$ be a copy of $G[V_j,V_S]$.  We finish the definition of $G'$ by declaring it has no other edges.  Let $G'_{uvS}=G'[A_u,B_v,C_S]$ (this notation is defined at the end of Subsection \ref{ss:bigraphs}).

We next define an auxiliary $3$-graph $H'$ by duplicating the behavior of $H$.  In particular, we let $H'$ be the $T$-partite $3$-graph with $V(H')=V'$ and edge set $E(H')$ satisfying the following.  For each $1\leq u,v\leq k$ and each $S\subseteq [k]^2$, $\overline{E(H')}\cap K_3(G_{uvS})$ is a copy of $\overline{E(H)}\cap K_3(G)$.  

For each $1\leq u,v\leq k$ and each $S\subseteq [k]^2$, let $H'_{uvS}=H'[A_u,B_v,C_S]$ (again this notation is defined at the end of Subsection \ref{ss:bigraphs}).  We then have that for all $1\leq u,v\leq k$ and $S\subseteq[k]^2$,  the pair $(H'_{uvS}, G'_{uvS})$ has $\dev_{2,3}(\e_1,\e_2(\ell))$, since it is just a copy of $(\overline{H}|G, G)$. Moreover, if $(u,v)\in S$, its density is at least $\e_1^{1/D}$, and if $(u,v)\notin S$, its density is at most $1-\e_1^{1/D}$.  Consequently, since $V_i,V_j,V_s$ are sufficiently large, Corollary \ref{cor:countingcor} implies there exists $(a_1,\ldots, a_k, b_1,\ldots, b_k, (c_S)_{S\subseteq [k]^2})\in \prod_{u\in [k]}|A_u|\prod_{v\in [k]}|B_v|\prod_{S\subseteq [k]^2}C_S$ with the property that $(a_u,b_v,c_S)\in E(H')$ for $(u,v)\in S$ and $(a_u,b_v,c_S)\notin E(H')$ for $(u,v)\notin S$.  By construction of $H'$, this implies  $H$ has $\VC_2$-dimension at least $k$, a contradiction.
\end{proof}

We extend the definition of $\VC_2$-dimension to hereditary properties as follows. 

\begin{definition}
Given a hereditary $3$-graph property, define $\VC_2(\calH)\in \mathbb{N}\cup \{\infty\}$ as follows.
$$
\VC_2(\calH)=\sup\{\VC_2(H): H\in \calH\}.
$$
\end{definition}
We will use the following observation later in the paper (for a proof, see Section 2.3 in \cite{Terry.2021b}).

\begin{observation}\label{ob:universal2}
Suppose $\calH$ is a hereditary $3$-graph property and $\VC_2(\calH)=\infty$.  Then for all tripartite $3$-graphs $H=(A\cup B\cup C, E)$, there is a $3$-graph $H'\in \calH$ with vertex set $A\cup B\cup C$ and edge set $E(H')$ satisfying $E(H')\cap K_3[A,B,C]=E$.  
\end{observation}

On the level of hereditary properties, $\VC_2$-dimension characterizes when homogeneous decompositions exist (see Theorem \ref{thm:vc2hom} below). The version of this characterization stated below was proved in \cite{Terry.2021b}.  A non-quantitative analogue was proved independently under a different formalism  in \cite{Chernikov.2020}.

\begin{theorem}\label{thm:vc2hom}
Suppose $\calH$ is a hereditary $3$-graph property.  The following are equivalent.
\begin{enumerate}
\item $\calH$ has finite $\VC_2$-dimension,
\item $\calH$ admits homogeneous decompositions in the following sense: for all $\e_1>0$ and $\e_2:\mathbb{N}\rightarrow (0,1)$, there are $T, L\geq 1$ so that all sufficiently large $H\in \calH$, there exists $1\leq t\leq T$ and $1\leq \ell\leq L$, and an  $\e_1$-homogeneous $(t,\ell,\e_1,\e_2(\ell))$-decomposition of $H$.  
\end{enumerate}
\end{theorem}

We note Theorem \ref{thm:vc2hom} as proved in \cite{Terry.2021b} used equitable decompositions.  However it is not difficult to see that the proof can be adjusted to work without the equitability assumption (see the proof of Theorem 2.26 in \cite{Terry.2021b}).

\subsection{Homogeneous partitions and slicewise VC-dimension}\label{ss:homsl}

In this section we discuss homogeneous partitions and slicewise VC-dimension.  We begin with the definition of a homogeneous partition. 

\begin{definition}\label{def:hom}
Suppose $H=(V,E)$ is an $3$-graph. 
\begin{enumerate}
\item Given $X,Y,Z\subseteq V$, the tuple $(X,Y,Z)$ is \emph{$\e$-homogeneous for $H$} if 
$$
d_H(X,Y,Z)\in [0,\e)\cup (1-\e,1].
$$
\item A partition $\calP$ of $V(G)$ is \emph{$\e$-homogeneous for $H$} if 
$$
\Big|\bigcup_{(X,Y,Z)\in \Sigma_{hom}}X\times Y\times Z\Big|\geq (1-\e)|V|^3,
$$
where $\Sigma_{hom}=\{(X,Y,Z)\in \calP^3: (X,Y,Z)\text{ is $\e$-homogeneous for }H\}$. 
\end{enumerate}
\end{definition}

It is an exercise to show homogeneous partitions are also regular partitions in the sense the ``weak regularity" of \cite{Chung.1991,Haviland.1989}.  As such, they play a crucial role in Parts 1 and 2 of this series, which consider the growth of weak regular partitions.  Homogeneous partitions are closely related to \emph{slicewise VC-dimension}, which is another analogue of VC-dimension for $3$-graphs.\footnote{Slicewise VC-dimension is called \emph{weak VC-dimension} in \cite{Terry.2021b}.}  This notion is defined in terms of ``slice graphs," also called ``link graphs." 

\begin{definition}
Suppose $H=(V,E)$ is a $3$-graph and $x\in V$.  The \emph{slicegraph of $H$ at $x$} is the graph 
$$
H_x=(V, \{yz\in V: xyz\in E\}).
$$
\end{definition} 

The slicewise VC-dimension of a $3$-graph is then defined to be the maximum VC-dimension of its slice graphs.

\begin{definition}
Suppose $H=(V,E)$ is a $3$-graph.  The \emph{slicewise VC-dimension of $H$} is 
$$
\SVC(H):=\max\{\VC(H_x): x\in V\},
$$
\end{definition}

We extend this definition to hereditary properties as follows.

\begin{definition}
For a hereditary $3$-graph property $\calH$, let $\SVC(\calH)\in \mathbb{N}\cup \{\infty\}$ be 
 $$
 \SVC(\calH)=\sup\{\SVC(H):H\in \calH\}.
 $$
\end{definition}

We will use the following observation about slicewise VC-dimension.  We leave the proof as an exercise to the reader.

\begin{observation}\label{ob:universal3}
Suppose $\calH$ is a hereditary $3$-graph property and assume that for all $k\geq 1$, $k\otimes U(k)\in \trip(\calH)$.  Then for all bipartite graphs $G$ and all $n\geq 1$, $n\otimes G\in \trip(\calH)$.
\end{observation}

The following theorem characterizes when a hereditary $3$-graph property admits homogeneous partitions.  The version stated below is due to the author and Wolf  \cite{Terry.2021b}. A non-quantitative version follows from  independent work of Chernikov and Towsner  \cite{Chernikov.2020}. 

\begin{theorem}\label{thm:vdischom}
Suppose $\calH$ is a hereditary $3$-graph property.  The following are equivalent.
\begin{enumerate}
\item $\trip(\calH)$ omits $k\otimes U(k)$ for some $k$,
\item $\calH$ is close\footnote{See Definition 2.6 in \cite{Terry.2024b} for a definition of closeness.} to some $\calH'$ with finite slice-wise VC-dimension,
\item $\calH$ admits homogeneous partitions in the following sense: for all $\e>0$ there is $M(\e)$ so that all sufficiently large elements in $\calH$ have $\e$-homogneous equipartitions with at most $M(\e)$ parts. 
\item $\calH$ admits homogeneous partitions in the following sense: for all $\e>0$ there is $M(\e)$ so that all sufficiently large elements in $\calH$ have $\e$-homogneous partitions with at most $M(\e)$ parts. 
\end{enumerate}
\end{theorem}

We note that  \cite{Terry.2021b} in fact proves the equivalence of (1)-(3).  The equivalence of (3) and (4) is an easy exercise using Lemma \ref{lem:averaging}.   We will also essentially reprove (1) implies (4) in the next section.

Theorem \ref{thm:vdischom} leads us to the following definition.

\begin{definition}\label{def:Mhom}
Suppose $\calH$ is a hereditary $3$-graph property and assume that for some $k\geq 1$, $k\otimes U(k)\notin \trip(\calH)$.  Define $M_{\calH}^{hom}:(0,1)\rightarrow \mathbb{N}$ by setting $M_{\calH}^{hom}(\e)$ to be the smallest integer so that any sufficiently large element in $\calH$ has an $\e$-homogeneous partition with at most $M_{\calH}^{hom}(\e)$ parts.
\end{definition}

Note Definition \ref{def:Mhom} makes sense because of Theorem \ref{thm:vdischom}.  A corollary of the work in part 2, and the improvement of part 1 \cite{Terry.2024a} recently proved in \cite{GSW}, is the following theorem, where $M_{\calH}^{weak}$ denotes the minimal size of weakly $\e$-regular partitions in $\calH$ (see part 2 \cite{Terry.2024b} for a detailed definition of this). 

\begin{theorem}\label{thm:weakhom}
Suppose $\calH$ is a hereditary $3$-graph property and assume that for some $k\geq 1$, $k\otimes U(k)\notin \trip(\calH)$.  Then one of the following hold.
\begin{enumerate}
\item (Exponential) For some $C,C'>0$, $2^{\e^{-C}}\leq M^{weak}_{\calH}(\e^{1/8})\leq M^{hom}_{\calH}(\e)\leq 2^{\e^{-C'}}$,
\item (Polynomial) For some $C,C'>0$, $\e^{-C}\leq M^{weak}_{\calH}(\e^{1/8})\leq M^{hom}_{\calH}(\e)\leq \e^{-C'}$,
\item (Constant) There is some $C>0$, $M^{weak}_{\calH}(\e)=M^{hom}_{\calH}(\e)=C$.
\end{enumerate}
\end{theorem}

Theorem \ref{thm:weakhom} is a crucial ingredient in Theorem \ref{thm:strong1}.  We will see in the next section that for the properties $\calH$ where Theorem \ref{thm:weakhom} applies, $T_{\calH}$ is basically controlled by $M_{\calH}^{hom}$ (and thus $M_{\calH}$).

\section{The Slow Growth Rates for $T_{\calH}$}\label{sec:strongvert}

This section contains the main results about properties $\calH$ for which $T_{\calH}$ grows at a sub-wowzer rate, including a proof of Theorem \ref{thm:ingredient1}.  We begin by showing that when $\calH$ satisfies the hypotheses of Theorem \ref{thm:weakhom},  $T_{\calH}$ is roughly the same as $M^{hom}_{\calH}$ from Definition \ref{def:Mhom}.  We first note that homogeneous partitions can be made into $\dev_{2,3}$-regular decompositions. 

\begin{lemma}\label{lem:vd}
For all $\e_1>0$, $\e_2:\mathbb{N}\rightarrow (0,1]$, and $t\geq 1$, the following holds.  Suppose $H=(V,E)$ is a sufficiently large $3$-graph, and $\calQ$ is a partition of $V$ of size $t$ which is $\e_1$-homogeneous with respect to $H$ (in the sense of Definition \ref{def:hom}).  

Then there exists a  $(t,1,\e_1,\e_2(1))$-decomposition $\calP$ of $V$ so that $\calP_1=\calQ$ and so that $\calP$ is $\e_1$-homogeenous with respect to $H$ (in the sense of Definition \ref{def:homdec}). 
\end{lemma}
\begin{proof}
Fix $\e_1>0$, $\e_2:\mathbb{N}\rightarrow (0,1]$, and $t\geq 1$.  Suppose $H=(V,E)$ is a sufficiently large $3$-graph, and $\calQ$ is an $\e_1$-homogeneous partition of $V$ of size $t$. Let $\Sigma_{hom}$ be the set of $\e_1$-homogeneous triples from $\calQ^3$, and let 
 $$
 \calS_{hom}=\bigcup_{(X,Y,Z)\in \Sigma_{hom}}X\times Y\times Z.
 $$
 By assumption, $|\calS_{hom}|\geq (1-\e_1)|V|^3$.  Let $\calP$ be the $(t,1)$-decomposition of $V$ with $\calP_1=\calQ$ and $\calP_2=\{X\times Y: (X,Y)\in \calQ^2\}$.  Clearly each bigraph from $\calP_2$ has $\dev_{2}(\e_2(1),1)$.  Note that any $G\in \triads(\calP)$ has the form $(X,Y,Z; X\times Y, X\times Z, Y\times Z)$ for some $X,Y,Z\in \calP$, and moreover $K_3(G)=X\times Y\times Z$.  To ease notation, we will denote the triad $(X,Y,Z; X\times Y, X\times Z, Y\times Z)$ by $G_{XYZ}$. 

Let $\Omega=\{G_{XYZ}:(X,Y,Z)\in \Sigma_{hom}\}$.  Clearly $G\in \Omega$ is $\e$-homogeneous with respect to $H$ in the sense of Definition \ref{def:homdec}, since
$$
\overline{E}\cap K_3(G)=\overline{E}\cap (X\times Y\times Z).
$$  
Further, we have $|\bigcup_{G\in \Omega}K_3(G)|=|\calS_{hom}|\geq (1-\e_1)|V|^3$.  Thus $\calP$ is $\e_1$-homogeneous in the sense of Definition \ref{def:homdec}.
\end{proof}

Lemma \ref{lem:vd} shows that homogeneous partitions in the sense of Definition \ref{def:hom} give rise to homogeneous decompositions in the sense of Definition \ref{def:homdec} with $\ell=1$.  We next prove a  lemma which says that when a tripartite $3$-graph $H$ omits $k\otimes U(k)$, the densities on certain regular triads in $H$ must be roughly the same.  This was previously shown in \cite{Terry.2021b} (see Section 4.2 there).  We sketch a slightly shorter proof here, which is similar to the proof of Corollary \ref{cor:countingcor}.   

\begin{lemma}\label{lem:otherway}
For all $k\geq 1$ there is a constant $D\geq 1$ and a polynomial $p(x,y)$ so that for all $0<d_2<1/2$, $0<\e_1<2^{-D}$ and $0<\e_2<p(\e_1, d_2)$ the following holds.  

Suppose $G=(U, V, W; E_{UV},E_{UW},E_{VW})$ is a triad where $|U|$, $|V|$, and $|W|$ are sufficiently large, and where each of $G[U,V]$, $G[U,W]$, $G[V,W]$ have $\dev_2(\e_2)$ and density at least $d_2$.  Suppose $G'=(U, V,W; E_{UW},E_{VW}, E_{UV}')$ where $G'[U,V]$ also has $\dev_2(\e_2)$ and density at least $d_2$.

Suppose $H=(U\cup V\cup W,F)$ is a $3$-graph so that $k\otimes U(k)$ is not an induced sub-$3$-graph of $\trip(H)$.  Assume $d_{H}(G),d_{H}(G')\in [0,\e)\cup (1-\e,1]$.  Then $d_{H}(G),d_{H}(G')$ are either both in $[0,\e)$ or both in $(1-\e,1]$.  
\end{lemma}
\begin{proof}
Let $D$ and $p(x,y)$ be as in Corollary \ref{cor:countingcor} for $T=2k+2^{k^2}$. Fix $0<\e_1<2^{-D}$ and $0<\e_2<p(\e_1,d_2)$.  

Suppose $H$, $G$, and $G'$ are as in the hypotheses of Lemma \ref{lem:otherway}.  Suppose towards a contradiction $d_{H}(G)\geq 1-\e$ and $d_{H}(G')\leq \e$.  We next define auxiliary bigraphs by duplicating vertices and edges from $G$.  We begin by defining new vertex sets 
$$
U'=U_1\cup \ldots \cup U_k,\text{ }W'=W_1\cup \ldots \cup W_k,\text{ and }V'=\bigcup_{S\subseteq [k]}V_S,
$$
 where for each $1\leq i\leq k$, $U_i$ is a copy of $U$ and $W_i$ is a copy of $W$, and for each $S\subseteq [k]$, $V_S$ is a copy of $V$.  For each $1\leq i,j\leq k$, and $S\subseteq [k]$, define $G_{U_iW_j}$ to the bigraph with vertex sets $(U_i,W_j)$, and edge set a copy of $E_{UW}$. Similarly, given $1\leq j\leq k$ and $S\subseteq [k]$ define $G_{V_SW_j}$ to be the bigraph with vertex sets $(V_S,W_j)$ and and edge set a copy of $E_{VW}$. Finally, given $1\leq i\leq k$ and $S\subseteq [k]$, define $G_{U_iV_S}$ to be the bigraph with vertex sets $(U_i,V_S)$ and edge set $E_{U_iV_S}$ satisfying
\[
E_{U_iV_S}=\begin{cases}\text{ a copy of }E_{UV}&\text{ if }i\in S\text{ and }\\
\text{ a copy of }E_{UV}'&\text{ if }i\notin S.
\end{cases}\]
For $1\leq a,b\leq k$ and $S\subseteq [k]$, let $G_{U_iW_jV_S}$ be the triad whose component bigraphs are $G_{U_iW_j}, G_{W_jV_S}, G_{U_iV_S}$.  

We now define an auxiliary $3$-graph $H'=(U'\cup V'\cup W', E(H'))$, where we $E(H')$ is determined by the following definition for $\overline{E(H')}$.  For triads of the form $G_{U_iV_SW_j}$ with $i\in S$, we define $\overline{E(H')}\cap K_3(G_{U_iV_SW_j})$ to be a copy of $\overline{E}\cap K_3(G)$.  For triads of the form $G_{U_iV_SW_j}$ with $i\notin S$, we define $\overline{E(H')}\cap K_3(G_{U_iV_SW_j})$ to be a copy of $K_3(G')\setminus \overline{E}$.  We finish the definition of $\overline{E(H')}$ by declaring it contains no other edges. This finishes our definition of $H'$.  For each $1\leq i,j\leq k$ and $S\subseteq [k]$, let 
$$
d_{iSj}=d_{H'}(G_{U_iV_SW_j}).
$$
By construction, each $d_{iSj}\in [0,\e)\cup (1-\e,1]$. Consequently, Proposition \ref{prop:homimpliesrandome},  implies each $(\overline{H'}|G_{U_{i}V_{S}W_j},G_{U_{i}V_{S}W_j})$ has $\dev_{2,3}(\e_2,6\e)$.  But now Corollary \ref{cor:countingcor} (applied with parameters $T$, $\e_1$, $\e_2$, $d_2$ and $d_1:=1/2$) implies $\trip(H)$ contains an induced copy of $k\otimes U(k)$, a contradiction.
\end{proof}

The following averaging lemma is an easy exercise (see part 1 \cite{Terry.2024a} for a proof).  It  shows up in the proof of Proposition \ref{prop:dev} below, and throughout the rest of the paper.

\begin{lemma}\label{lem:averaging}
Let $a,b,\e\in (0,1)$ satisfy $ab=\e$. Suppose $A\subseteq X$ and $|A|\geq (1-\e)|X|$.  For any partition $\calP$ of $X$, if we let $\Sigma=\{Y\in \calP: |A\cap Y|\geq (1-a)|Y|\}$, then $|\bigcup_{Y\in \Sigma}Y|\geq (1-b)|X|$.
\end{lemma}

Our next result, Proposition \ref{prop:dev} below, shows that for a $3$-graph $H$, if $\trip(H)$ omits $k\otimes U(k)$, the $H$ has a homogeneous decompositions in the sense of Definition \ref{def:hom}.  This proof is similar to the proof of Proposition 4.1 in \cite{Terry.2021b}, but the proof is slightly more complicated due to the fact we start with a possibly non-equitable decomposition.

\begin{proposition}\label{prop:dev}
For all $k\geq 1$ there are $K\geq 1$ and a polynomial $q(x,y)$ so that for all $0<\e_1<2^{-K}$, all $\e_2:\mathbb{N}\rightarrow (0,1]$ satisfying $\e_2(x)<q(\e_1,x^{-1})$, and all $T,L\geq 1$, the following holds.

Suppose $H=(V,E)$ is a sufficiently large $3$-graph such that $k\otimes U(k)$ is not an induced sub-$3$-graph of $\trip(H)$.  Suppose $1\leq t\leq T$, $1\leq \ell\leq L$, and $\calP$ is a $(t,\ell,\e_1,\e_2(\ell))$-decomposition of $V$ which is  $\dev_{2,3}(\e_1,\e_2(\ell))$-regular with respect to $H$.  Then $\calP_1$ is a $2\e_1^{1/K}$-homogeneous partition with respect to $H$ (in the sense of Definition \ref{def:hom}).
\end{proposition}
\begin{proof}
Let $D$ and $p(x,y)$ be as in Proposition \ref{prop:suffvc2} for $k$, and set $K=\max\{D,2\}$.  Define $q(x,y)=p(x,xy)(x/4)^8$.  Assume $0<\e_1<2^{-K}$ and $\e_2:\mathbb{N}\rightarrow (0,1]$ satisfies 
$$
\e_2(x)<q(\e_1,x^{-1})=\e_1^8p(\e_1,\e_1x^{-1}).
$$

  Fix $T,L\geq 1$ and suppose $H=(V,E)$ is a sufficiently large $3$-graph such that $k\otimes U(k)$ is not an induced sub-$3$-graph of $\trip(H)$.  Note this implies $H$ has $\VC_2$-dimension less than $k$.  Suppose $1\leq t\leq T$, $1\leq \ell\leq L$, and $\calP$ is a $(t,\ell,\e_1,\e_2(\ell))$-decomposition of $V$ which is  $\dev_{2,3}(\e_1,\e_2(\ell))$-regular with respect to $H$.  We will show $\calP_1$ is $2\e_1^{1/K}$-homogeneous with respect to $H$, in the sense of Definition \ref{def:hom}.

We begin by setting up a  bit of notation.  First, we fix some enumerations 
$$
\calP_1=\{V_1,\ldots, V_t\}\text{ and }\calP_2=\{P_{ij}^{\alpha}: 1\leq i,j\leq t, 1\leq \alpha\leq \ell\}.
$$
Given $1\leq i,j,s\leq t$ and $1\leq \alpha,\beta,\gamma\leq \ell$, we let $G_{ijs}^{\alpha\beta\gamma}$ denote the triad $(V_i,V_j,V_s;P_{ij}^{\alpha},P_{ik}^{\beta},P_{js}^{\gamma})$.  For each $1\leq i,j\leq t$ and $1\leq \alpha\leq \ell$, let $d_{ij}^{\alpha}$ denote the density of the bigraph $(V_i,V_j; P_{ij}^{\alpha})$.

Let $\Omega$ be the set of $\e_1$-nontrivial triads of $\calP$ which are $\dev_{2,3}(\e_1,\e_2(\ell))$-regular with respect to $H$, and set $\calS=\bigcup_{G\in \Omega}K_3(G)$.   Using the assumption that $\calP$ is $\dev_{2,3}(\e_1,\e_2(\ell))$-regular and Lemma \ref{lem:nontrivialtriad}, we see that $|\calS|\geq (1-3\e_1)|V|^3$. By Proposition \ref{prop:suffvc2}, we know that every $G\in \Omega$ is $\e_1^{1/K}$-homogeneous with respect to $H$.  Consequently, we can write $\Omega=\Omega^0\cup \Omega^1$, where 
\begin{align*}
\Omega^0=\{G\in \Omega: d_H(G)\leq \e_1^{1/K}\}\text{ and }\Omega^1=\{G\in \Omega: d_H(G)\geq 1-\e_1^{1/K}\}.
\end{align*}
Later in the proof we will use that, by Proposition \ref{prop:counting}, each $G_{ijs}^{\alpha\beta\gamma}\in \Omega$ satisfies
\begin{align}\label{al:tricount}
|K_3(G_{ijs}^{\alpha\beta\gamma})|=(d_{ij}^{\alpha}d_{is}^{\beta}d_{js}^{\gamma}\pm 4\e_2(\ell)^{1/4})|V_i||V_j||V_s|.
\end{align}
We now define the set of triples from $\calP$ which are almost covered by the set $\calS$.  In particular, we define
\begin{align*}
\Sigma&=\{(X,Y,Z)\in \calP_1^3: |\calS\cap (X\times Y\times Z)|\geq (1-2\sqrt{\e_1})|X||Y||Z|\}.
\end{align*}
Since $|\calS|\geq (1-3\e_1)|V|^2$, Lemma \ref{lem:averaging} implies $|\bigcup_{(X,Y,Z)\in \Sigma}X\times Y\times Z|\geq (1-2\sqrt{\e_1})|V|^3$.  Thus, it suffices to show each triple in $\Sigma$ is $2\e_1^{1/K}$-homogeneous with respect to $H$.  

Fix $(V_i,V_j,V_s)\in \Sigma$. We show $(V_i,V_j,V_s)$ is $2\e_1^{1/K}$-homogeneous with respect to $H$.  Since $(V_i,V_j,V_s)\in \Sigma$ and (\ref{al:tricount}), we have
\begin{align*}
(1-2\sqrt{\e_1})|V_i||V_j||V_s|\leq |\calS\cap (V_i\times V_j\times V_s)|=|V_i||V_j||V_s|\sum_{\{(\alpha,\beta,\gamma)\in [\ell]^3: G_{ijs}^{\alpha\beta\gamma}\in \Omega\}}(d_{ij}^{\alpha}d_{is}^{\beta}d_{js}^{\gamma}\pm 4\e_2(\ell)^{1/4}).
\end{align*}
Consequently, $1-2\sqrt{\e_1}\leq \sum_{\{(\alpha,\beta,\gamma)\in [\ell]^3: G_{ijs}^{\alpha\beta\gamma}\in \Omega\}}d_{ij}^{\alpha}d_{is}^{\beta}d_{js}^{\gamma}+ 4\e_2(\ell)^{1/4}$.  Since $\e_2(\ell)<q(\e_1,\ell^{-1})$,  this implies 
\begin{align}\label{al:lbtri}
1-3\sqrt{\e_1}\leq \sum_{\{(\alpha,\beta,\gamma)\in [\ell]^3: G_{ijs}^{\alpha\beta\gamma}\in \Omega\}}d_{ij}^{\alpha}d_{is}^{\beta}d_{js}^{\gamma}.
\end{align}
We will now mimic the argument appearing in the proof of Proposition 4.1 of \cite{Terry.2021b}, with alterations to account for the fact our decomposition is not necessarily equitable. We begin by defining an auxiliary trigraph $\Gamma=([\ell],[\ell],[\ell]; {\bf E})$, where 
$$
{\bf E}=\{(\alpha,\beta,\gamma)\in [\ell]^3: G_{ijs}^{\alpha\beta\gamma}\in \Omega\}
$$
Given $1\leq \alpha\leq \ell$, we define the following quantity, which acts as a weighted measure of a neighborhood in $\Gamma$.
$$
\calN(\alpha):=\sum_{(\beta,\gamma)\in N_{{\bf E}}(\alpha)}d_{is}^{\beta}d_{js}^{\gamma}.
$$
Similarly, given $1\leq \alpha,\beta\leq \ell$, we define a weighted measure of the neighborhood of $(\alpha,\beta)$ in $\Gamma$ as follows.
$$
\calN(\alpha,\beta):=\sum_{\gamma\in N_{{\bf E}}(\alpha,\beta)}d_{js}^{\gamma}.
$$
Let $\calI_1=\{\alpha: 1\leq \alpha\leq \ell\text{ and }\calN(\alpha)\geq 1-2\e_1^{1/4}\}$.  We claim 
\begin{align}\label{al:ibd}
\sum_{\alpha\in \calI_1}d_{ij}^{\alpha}\geq 1-2\e_1^{1/4}.
\end{align}
  Indeed, suppose this is not the case.  Then 
\begin{align*}
\sum_{(\alpha,\beta,\gamma)\in {\bf E}}d_{ij}^{\alpha}d_{is}^{\beta}d_{js}^{\gamma}=\sum_{\alpha\in [\ell]\setminus \calI_1}\calN(\alpha)+\sum_{\alpha\in \calI_1}\calN(\alpha)&< \sum_{\alpha\in [\ell]\setminus \calI_1}d_{ij}^{\alpha}(1-\e_1^{1/4})+\sum_{\alpha\in \calI_1}d_{ij}^{\alpha}\\
&=(1-2\e_1^{1/4})\sum_{\alpha\in [\ell]}d_{ij}^{\alpha}+2\e_1^{1/4}\sum_{\alpha\in \calI_1}d_{ij}^{\alpha}\\
&=1-2\e_1^{1/4}+2\e_1^{1/4}\sum_{\alpha\in \calI_1}d_{ij}^{\alpha}\\
&<1-2\e_1^{1/4}+2\e_1^{1/4}(1-2\e_1^{1/4})\\
&=1-4\e_1^{1/2},
\end{align*}
contradicting (\ref{al:lbtri}).  This concludes our proof of (\ref{al:ibd}).  Now, for each $\alpha\in \calI_1$, define 
$$
\calI_2^{\alpha}=\{\beta\in [\ell]: 1\leq \beta\leq \ell\text{ and }\calN(\alpha,\beta)\geq 1-2\e_1^{1/8}\}.
$$
We claim that for all $\alpha\in \calI_1$,
\begin{align}\label{al:ibd2}
\sum_{\beta\in \calI_2^{\alpha}}d_{is}^{\beta}\geq 1-2\e_1^{1/8}.
\end{align}
  Indeed, suppose this is not the case.  Then 
\begin{align*}
\calN(\alpha)=\sum_{\beta\in [\ell]\setminus \calI_2^{\alpha}}d_{is}^{\beta}\calN(\alpha,\beta)+\sum_{b_{\beta}\in \calI_2^{\alpha}}d_{is}^{\beta}\calN(\alpha,\beta)&< \sum_{\beta\in [\ell]\setminus \calI_2^{\alpha}}(1-2\e_1^{1/8})d_{is}^{\beta}+\sum_{\beta\in \calI_2^{\alpha}}d_{is}^{\beta}\\
&=(1-2\e_1^{1/8})\sum_{\beta\in [\ell]}d_{is}^{\beta}+2\e_1^{1/8}\sum_{\beta\in  \calI_2^{\alpha}}d_{is}^{\beta}\\
&=1-2\e_1^{1/8}+2\e_1^{1/8}\sum_{\beta\in  \calI_2^{\alpha}}d_{is}^{\beta}\\
&<1-2\e_1^{1/4}+2\e_1^{1/8}(1-2\e_1^{1/8})\\
&=1-4\e_1^{1/8},
\end{align*}
contradicting the assumption that $\alpha\in \calI_1$.  This finishes our verification of (\ref{al:ibd2}).  We now observe that  for all $\alpha\in \calI_1$ and $\beta\in \calI_2^{\alpha}$, by Lemma \ref{lem:otherway}, there is a single value $\sigma(\alpha,\beta)\in \{0,1\}$ so that for all $\gamma\in N_{{\bf E}}(\alpha,\beta)$, $G_{ijs}^{\alpha\beta\gamma}\in \Omega^{\sigma(\alpha,\beta)}$.   In our next claim, we show that for a fixed $\alpha\in \calI_1$, these values must all agree.

\begin{claim}\label{cl:agood}
For all $\alpha\in \calI_1$, there is $\sigma(\alpha)\in \{0,1\}$ so that $\sigma(\alpha)=\sigma(\alpha,\beta)$ for all $\beta\in \calI_2^{\alpha}$.
\end{claim}
\begin{proof}
Fix $\alpha\in \calI_1$, and suppose $\beta,\beta'\in \calI_2^{\alpha}$.  Since $\beta,\beta'\in \calI_2^{\alpha}$, know that $\calN(\alpha,\beta)$ and $\calN(\alpha,\beta')$ are both at least $1-2\e_1^{1/8}$.  Translating the definitions of these quantities, this implies that 
$$
\min\Big\{\Big|\bigcup_{\{P_{js}^{\gamma}:G_{ijk}^{\alpha\beta\gamma}\in \Omega\}}P_{js}^{\gamma}\Big|,\Big|\bigcup_{\{P_{js}^{\gamma}:G_{ijs}^{\alpha\beta'\gamma}\in \Omega\}}P_{js}^{\gamma}\Big|\Big\}\geq (1-2\e_1^{1/8})|V_j||V_s|.
$$
Consequently, there is some $\gamma\in [\ell]$ so that $G_{ijs}^{\alpha\beta\gamma}$ and $G_{ijs}^{\alpha\beta'\gamma}$ are both in $\Omega$.  By Lemma \ref{lem:otherway}, we must have $\sigma(\alpha,\beta)=\sigma(\alpha,\beta')$. 
\end{proof}

Our next claim shows that all the values $\sigma(\alpha)$ from Claim \ref{cl:agood} must all agree.

\begin{claim}\label{cl:bgood}
There is some $\sigma\in \{0,1\}$ so that for all $\alpha\in \calI_1$, $\sigma(\alpha)=\sigma$.
\end{claim}
\begin{proof}
Fix some $\alpha,\alpha'\in \calI_1$.  By (\ref{al:ibd2}), we know that $\sum_{\beta\in \calI_2^{\alpha}}d_{is}^{\beta}$ and $\sum_{\beta\in \calI_2^{\alpha'}}d_{is}^{\beta}$ are both at least $1-2\e_1^{1/4}$.  Translating the definitions, this implies 
$$
\min\Big\{\Big|\bigcup_{\{P_{is}^{\beta}:\beta\in \calI_2^{\alpha}\}}P_{is}^{\beta}\Big|,\Big|\bigcup_{\{P_{is}^{\beta}:\beta\in \calI_2^{\alpha'}\}}P_{is}^{\beta}\Big|\Big\}\geq (1-2\e_1^{1/4})|V_i||V_k|.
$$
Consequently, there is some $P_{is}^{\beta}\in \calI_2^{\alpha}\cap  \calI_2^{\alpha'}$.  By definition, this implies that both $\calN(\alpha,\beta)$ and $\calN(\alpha',\beta)$ are at least $1-2\e_1^{1/8}$.  Translating, this means 
$$
\min\Big\{\Big|\bigcup_{\{P_{js}^{\gamma}:G_{ijs}^{\alpha\beta\gamma}\in \Omega\}}P_{js}^{\gamma}\Big|,\Big|\bigcup_{\{P_{js}^{\gamma}:G_{ijs}^{\alpha'\beta\gamma}\in \Omega\}}P_{js}^{\gamma}\Big|\Big\}\geq (1-2\e_1^{1/8})|V_j||V_s|.
$$
Thus there is some $\gamma\in [\ell]$ so that $G_{ijs}^{\alpha\beta\gamma}$ and $G_{ijs}^{\alpha'\beta\gamma}$ are in $\Omega$.  Lemma \ref{lem:otherway}, we must have $\sigma(\alpha)=\sigma(\alpha,\beta)=\sigma(\alpha',\beta)=\sigma(\alpha')$. 
\end{proof}

Let $\sigma\in \{0,1\}$ be as in Claim \ref{cl:bgood}.  We are now ready to show $(V_i,V_j,V_s)$ is homogeneous.  Indeed, we have the following, where $\overline{E(H)}^1=\overline{E(H)}^1$ and $\overline{E(H)}^0=V(H)^3\setminus \overline{E(H)}$.
\begin{align*}
|\overline{E(H)}^{\sigma}\cap (V_i\times V_j\times V_s)|&\geq \sum_{G_{ijk}^{\alpha\beta\gamma}\in \Omega}(1-\e_1^{1/K})|K_3(G_{ijk}^{\alpha\beta\gamma})|+|(V_i\times V_j\times V_s)\setminus \calS|\\
&=(1-\e_1^{1/K})|(V_i\times V_j\times V_s)\cap\calS|+|(V_i\times V_j\times V_s)\setminus \calS|\\
&\geq (1-2\e_1^{1/K})|V_i||V_j||V_s|,
\end{align*}
where the last inequality uses that $(V_i,V_j,V_s)\in \Sigma$, and the fact that $K\geq 2$.
\end{proof}

We can now prove that when $k\otimes U(k)\notin \trip(\calH)$ for some $k\geq 1$, then $T_{\calH}$ essentially reduces to $M_{\calH}^{hom}$. In particular, Theorem \ref{thm:comparison} below  implies Theorem \ref{thm:ingredient1} from the introduction.

\begin{theorem}\label{thm:comparison}
For all $k\geq 1$ there is a $K\geq 1$ and a polynomial $p(x,y)$ so that the following holds.  Suppose $\calH$ is a hereditary $3$-graph property so that   $k\otimes U(k)\notin \trip(\calH)$.  For all sufficiently small $\e_1>0$ and all $\e_2:\mathbb{N}\rightarrow (0,1]$ satisfying $\e_2(x)\leq p(\e_1,x^{-1})$, 
$$
M^{weak}_{\calH}(2\e^{1/4K})\leq M_{\calH}^{hom}(2\e_1^{1/K})\leq T_{\calH}(\e_1,\e_2)\leq M^{hom}_{\calH}(\e_1/6).
$$
\end{theorem}
\begin{proof}
The fact that  $M^{weak}_{\calH}(2\e_1^{1/4K})\leq M_{\calH}^{hom}(2\e_1^{1/K})$ holds for sufficiently small $\e_1>0$  follows from standard arguments (see e.g. Proposition 4.8 in Part 2 \cite{Terry.2024b}).

We now show  the rightmost inequality  holds for all sufficiently small $\e_1>0$ and all $\e_2:\mathbb{N}\rightarrow (0,1]$.  Suppose $H=(V,E)$ is a sufficiently large element of $\calH$.  Fix $\e_1>0$ sufficiently small and any $\e_2:\mathbb{N}\rightarrow (0,1]$.  Let $\calQ=\{V_1,\ldots, V_t\}$ be an $\e_1/6$-homogeneous partition of $H$ (which exists by Theorem \ref{thm:vdischom}).  By Lemma \ref{lem:vd} there is a  $(t,1,\e_1,\e_2(1))$-decomposition $\calP$ which is $\e_1/6$-homogeneous for $H$ in the sense of Definition \ref{def:homdec} and with $\calP_1=\calQ$.  By Proposition \ref{prop:homimpliesrandome}, $\calP$ is $\dev_{2,3}(\e_1,\e_2(1))$-regular.  This shows $T_{\calH}(\e_1,\e_2)\leq M_{\calH}^{hom}(\e_1/6)$. 

We now show the remaining inequality.  Let $K$ and $q(x,y)$ be as in Proposition \ref{prop:dev} for $k$.  Assume $0<\e_1<2^{-K}$ and $\e_2:\mathbb{N}\rightarrow (0,1)$ satisfies $\e_2(x)<q(\e_1,x^{-1})$. Let $L$ be such that $\psi(\e_1,\e_2,L,T_{\calH}(\e_1,\e_2),\calH)$ holds (recall the definition of this notation from the introduction).   

Assume $H$ is a sufficiently large element of $\calH$.  By assumption, there exists a $\dev_{2,3}(\e_1,\e_2(\ell))$-regular $(t,\ell,\e_1,\e_2(\ell))$-decomposition $\calP$ for $H$ for some $1\leq t\leq T_{\calH}(\e_1,\e_2(\ell))$ and $1\leq \ell \leq L$.  By Proposition \ref{prop:dev}, $\calP_1$ is $2\e^{1/K}_1$-homogeneous in the sense of Definition \ref{def:hom}.  This shows $M_{\calH}^{hom}(2\e_1^{1/K})\leq T_{\calH}(\e_1,\e_2)$ holds for all sufficiently small $\e_1>0$ and $\e_2:\mathbb{N}\rightarrow (0,1]$ satisfying $\e_2(x)\leq q(\e_1,x^{-1})$. 
\end{proof}

Combining Theorem \ref{thm:comparison} with Theorem \ref{thm:weakhom}, we immediately obtain Theorem \ref{thm:main2intro}.  As another corollary of Proposition \ref{prop:homimpliesrandome}, we see that under the same hypotheses at Theorem \ref{thm:main2intro}, $L_{\calH}$ is the constant function $1$.
 
\begin{corollary}\label{cor:Lequals1}
Suppose $\calH$ is a hereditary $3$-graph property and assume that for some $k\geq 1$, $k\otimes U(k)\notin \trip(\calH)$.  There is a polynomial $p(x,y)$ so that for all $\e_1>0$ sufficiently small, and all $\e_2:\mathbb{N}\rightarrow (0,1]$ satisfying $\e_2(x)\leq p(\e_1,x^{-1})$, $L_{\calH}(\e_1,\e_2)=1$.
\end{corollary}
\begin{proof}
Let $\e_1>0$ be sufficiently small, and assume $\e_2:\mathbb{N}\rightarrow (0,1)$ is such that $\e_2(x)< (1/2)^{48}$.  Suppose $H=(V,E)$ is a sufficiently large element of $\calH$.  By Theorem \ref{thm:weakhom}, there is some $1\leq t\leq M_{\calH}^{hom}(\e_1/6)$ and an $\e_1/6$-homogeneous partition $\calQ=\{V_1,\ldots, V_t\}$ for $H$.  By Lemma \ref{lem:vd} there is a  $(t,1,\e_1,\e_2(1))$-decomposition $\calP$ which is $\e_1/6$-homogeneous for $H$ in the sense of Definition \ref{def:homdec} and with $\calP_1=\calQ$.  By Proposition \ref{prop:homimpliesrandome}, $\calP$ is $\dev_{2,3}(\e_1,\e_2(1))$-regular.  This shows $L_{\calH}(\e_1,\e_2)\leq 1$.  By definition, $L_{\calH}(\e_1,\e_2)\geq 1$, so we have $L_{\calH}(\e_1,\e_2)=1$. 
\end{proof}

We will show in part 3 \cite{Terry.2024d}  that the converse of Corollary \ref{cor:Lequals1} also holds (see Corollary 7.2 there).

\section{The jump to wowzer}\label{sec:wowzer}

This section contains the proof of Theorem \ref{thm:strong1}.  In Subsection \ref{ss:CF} we cover needed results due to Conlon and Fox regarding strong graph regularity. In Subsection \ref{ss:std}, we present a series of standard lemmas about regular partitions.  Finally, in Subsection \ref{ss:1.9}, we prove Theorems \ref{thm:main1}, \ref{thm:main0}, and finally, Theorem \ref{thm:strong1}.

\subsection{A theorem of Conlon and Fox.}\label{ss:CF}

This section contains necessary results due to Conlon and Fox about strong regular partitions of graphs.  As these results are stated in terms of regular pairs (rather than $\dev_2$), we begin with some definitions.

\begin{definition}
Suppose $G=(V,E)$ is a graph and $X,Y\subseteq V$.  We say $(X,Y)$ is \emph{$\e$-regular with respect to $G$} if for all $X'\subseteq X$ and $Y'\subseteq Y$ satisfying $|X'|\geq \e|X|$ and $|Y'|\geq \e |Y|$, 
$$
|d_G(X,Y)-d_G(X',Y')|\leq \e.
$$
\end{definition}

\begin{definition}
Suppose $G=(V,E)$ is a graph.  A partition $V=V_1\cup \ldots \cup V_t$ is an \emph{$\e$-regular partition} if for at least $(1-\e)|V|^2$ many pairs $(x,y)\in V^2$, there is an $\e$-regular pair $(V_i,V_j)$ with $(x,y)\in V_i\times V_j$.   
\end{definition}

We will use the following quantitative equivalence between regularity and $\dev_2$-quasirandomenss (see  \cite{Gowers.20063gk} for a proof). 

\begin{theorem}\label{thm:equiv}
Suppose $B=(U\cup W, E)$ is a bipartite graph and $|E|=d|U||W|$.  
\begin{enumerate}
\item If $B$ is $\e$-regular then $(U,W;\overline{E})$ has $\dev_2(\e,d)$.
\item If $(U,W;\overline{E})$ has $\dev_2(\e,d)$, then $(U\cup W, E)$ is $\e^{1/12}$-regular. 
\end{enumerate}
\end{theorem}

We now define the so-called energy of a partition.

\begin{definition}\label{def:energy}
Suppose $G=(V,E)$ is a graph and $\calP=\{V_1,\ldots, V_m\}$ is a partition of $V$.  We define
$$
q_G(\calP)=\sum_{1\leq i<j\leq k}d_G^2(V_i,V_j)\frac{|V_i||V_j|}{|V|^2}.
$$
\end{definition}

A standard energy increment argument using Definition \ref{def:energy} yields the following version of Szemer\'{e}di's regularity lemma (recall also Definition \ref{def:tower} from Subsection \ref{ss:notation}).  We will use the following notation. Given $m\in \mathbb{N}$, we define a function $Tw_m:\mathbb{N}\rightarrow \mathbb{N}$ by setting $Tw_m(1)=m$ and for $x>1$, setting $Tw_m(x+1)=Tw_m(x)2^{Tw_m(x)}$.

\begin{theorem}\label{thm:reg}
For all $m\geq 1$ and $\e>0$, the following holds.  For all finite graphs $G=(V,E)$ and equipartitions $\calP=\{V_1,\ldots, V_m\}$ of $V$, there is an $\e$-regular equipartition $U_1,\ldots, U_M$ of $G$ which refines $\calP$, for some $M\leq Tw_m(\e^{-5})$. 
\end{theorem}

It is well known (see Lemma 3.7 of \cite{Alon.2000}) that if $G=(V,E)$ is a graph and $\calU=\{U_1,\ldots, U_s\}$ and $\calV=\{V_1,\ldots, V_t\}$ are partitions of $V$ so that  $\calU\preceq \calV$ and $q_G(\calV)<q_G(\calU)+\e$, then for almost all pairs $(V_i,V_j)$, $d_G(V_i,V_j)$ is close to $d_G(U_s,U_r)$ where $U_s\supseteq V_i$ and $U_r\supseteq V_j$.  It was pointed out in \cite{Conlon.2012} that a converse to this statement is also true.  Since this is important to our proof, we include a precise statement of this below (Fact \ref{fact:index}), with a proof in the appendix.  We will use the following definition to state Fact \ref{fact:index}.

\begin{definition}\label{def:cons}
Suppose $G=(V,E)$ is a graph and $\calU=\{U_1,\ldots, U_s\}$ and $\calV=\{V_1,\ldots,V_t\}$ are partitions of $V$.  We say $\calV$ is an \emph{$\e$-conservative refinement of $\calU$} if $\calV$ refines $\calU$ and for at least $(1-\e)t^2$ many pairs $1\leq i\neq j\leq t$ it holds that the $d_G(V_i,V_j)$ is within $\e$ of $d_G(U_{i'},U_{j'})$, where $V_i\subseteq U_{i'}$ and $V_j\subseteq U_{j'}$. 
\end{definition}

\begin{fact}\label{fact:index}
Suppose $G=(V,E)$ is a graph and  $\calU=\{U_1,\ldots, U_s\}$ and $\calV=\{V_1,\ldots, V_t\}$ are equipartitions of $V$.  If $\calV$ is a  $\e$-conservative refinement of $\calU$, then $q_G(\calU)<q_G(\calZ)+101\e$.
\end{fact}

We now state the theorem of Conlon and Fox which we will use (see Corollary 1.2 of \cite{Conlon.2012}).

\begin{theorem}[Conlon-Fox \cite{Conlon.2012}]\label{thm:conlonfox}
There is a constant $C>0$ so that for all $0<\e<2^{-700}$ and arbitrarily large $n$, there is a graph $G_n(\e)$ on $n$ vertices so that if $\calU,\calV$ are equitable partition of $G$ satisfying that $|\calU|\geq |\calV|$, that $q(\calU)\leq q(\calV)+\e$,  and that $\calU$ is $\e/|\calV|$-regular, then $|\calU|, |\calV|$ are both at least $W(C\e^{-1/7})$.
\end{theorem}

\subsection{Standard Lemmas}\label{ss:std}

This section contains several standard lemmas about graph and hypergraph regularity.  Throughout we have opted for brevity over optimization of bounds.  Our first lemma says that sub-pairs of regular pairs are still somewhat regular. For a proof, see Lemma 3.1 in \cite{Alon.2000}.
 
\begin{proposition}[Sub-pairs lemma v.1]\label{lem:slreg} Suppose $G=(A\cup B, E)$ is a bipartite graph and $|E|=d|A||B|$.  Suppose $A'\subseteq A$ and $B'\subseteq B$ satisfy $ |A'| \geq \gamma |A|$ and $|B'| \geq \gamma|B|$ for some $\gamma \geq \e$, and $G$ is $\e$-regular. Then $G':=(A'\cup B', E\cap K_2[A',B'])$ is $\e'$-regular with density $d'$, where  $\epsilon'=2\gamma^{-1}\e $ and $d'\in(d -\e,d + \e)$.
\end{proposition}

The following version for $\dev_2$ then follows from Theorem \ref{thm:equiv}.

\begin{proposition}[Sub-pairs lemma v.2]\label{lem:sldev} Suppose $G=(A, B; E)$ is a bigraph with density $d$. Suppose $A'\subseteq A$ and $B'\subseteq B$ satisfy $ |A'| \geq \gamma |A|$ and $|B'| \geq \gamma|B|$ for some $\gamma \geq \e$, and $G$ satisfies $\dev_2(\epsilon, d)$.  Then $G':=(A', B'; E\cap (A'\times B'))$ satisfies $\dev_2(\epsilon',d)$ where $\epsilon'=2\gamma^{-1}\e^{1/12}$.
\end{proposition}

Our next fact tells us we can combine the edges or vertices of disjoint quasirandom graphs and obtain a quasirandom graph. 

\begin{fact}[Unioning Edges]\label{fact:adding}
Suppose $E_1$ and $E_2$ are disjoint subsets of $U\times V$.  Assume $(U,V;E_1)$ has $\dev_2(\e_1,d_1)$, and $(U, V;E_2)$ has $\dev_2(\e_2,d_2)$, then $(U,V;E_1\cup E_2)$ has $\dev_2(\e_1^{1/12}+\e_2^{1/12}, d_1+d_2)$.  
\end{fact}

 Fact \ref{fact:adding} can be deduced from the standard version for regularity (see e.g. Fact 2.23 in \cite{Terry.2021b}) and Theorem \ref{thm:equiv}. One can also show compliments of $\dev_2$-regular bigraphs are regular.
 
 \begin{lemma}[Bigraph compliments]\label{lem:complimentbi}
Suppose $B=(X,Y; E)$ is a bigraph satisfying $\dev_{2}(\e)$.   Let $B'=(X,Y;E')$ where $E'=(X\times Y)\setminus E$.  Then $B'$ also satisfies $\dev_{2}(\e)$.  
\end{lemma}
\begin{proof}
Let $d=|E|/|X||Y$.  Note $|E'|=(1-d)|X||Y|$, and thus, for all $(x,y)\in X\times Y$, $g_{B'}(x,y)=-g_{B}(x,y)$.  Consequently, for each $u_0,u_1\in X$ and $v_0,v_1\in Y$,
$$
\prod_{(i,j)\in \{0,1\}^2}g_{B}(u_i,v_j)=(-1)^{4}\prod_{(i,j)\in \{0,1\}^2}g_{B'}(u_i,v_j)=\prod_{(i,j)\in \{0,1\}^2}g_{B'}(u_i,v_j).
$$
This shows 
\begin{align*}
&\sum_{u_0,u_1\in X}\sum_{v_0,v_1\in Y}\prod_{(i,j)\in \{0,1\}^2}g_{B}(u_i,v_j)=\sum_{u_0,u_1\in X}\sum_{v_0,v_1\in Y}\prod_{(i,j)\in \{0,1\}^2}g_{B'}(u_i,v_j),
\end{align*}
which by definition implies $B'$ has $\dev_{2}(\e)$ (since $B$ does).
\end{proof}

Similarly, complimenting a $\dev_{2,3}$-regular trigraph preserves $\dev_{2,3}$-regularity.

\begin{lemma}[Trigraph compliments]\label{lem:compliment}
Suppose $H=(X,Y,Z; F)$ is a trigraph, $G$ is a triad underlying $H$, and $(H, G)$ has $\dev_{2,3}(\e_1,\e_2)$.  

Let $H'=(X,Y,Z;F')$ where $F'=K_3(G)\setminus F$.  Then $(H',G)$ also has $\dev_{2,3}(\e_1,\e_2)$.  
\end{lemma}
\begin{proof}
Say $G=(X,Y,Z; E_{XY},E_{YZ},E_{XZ})$.  Note $|F'|=(1-d_H(G))|K_3(G)|$, and thus, for all $(x,y,z)\in K_3(G)$, $h_{H',G}(x,y,z)=-h_{H,G}(x,y,z)$.  Consequently, for each $u_0,u_1\in X$, $v_0,v_1\in Y$, and $w_0,w_1\in Z$, 
$$
\prod_{(i,j,k)\in \{0,1\}^3}h_{H,G}(u_i,v_j,w_k)=(-1)^{12}\prod_{(i,j,k)\in \{0,1\}^3}h_{H',G}(u_i,v_j,w_k)=\prod_{(i,j,k)\in \{0,1\}^3}h_{H',G}(u_i,v_j,w_k).
$$
This shows 
\begin{align*}
&\sum_{u_0,u_1\in X}\sum_{v_0,v_1\in Y}\sum_{w_0,w_1\in Z}\prod_{(i,j,k)\in \{0,1\}^3}h_{H,G}(u_i,v_j,w_k)\\
&=\sum_{u_0,u_1\in X}\sum_{v_0,v_1\in Y}\sum_{w_0,w_1\in Z}\prod_{(i,j,k)\in \{0,1\}^3}h_{H',G}(u_i,v_j,w_k),
\end{align*}
which clearly implies $(H',G)$ has $\dev_{2,3}(\e_1,\e_2)$ (since $(H,G)$ does).
\end{proof}

The next fact follows from a completely standard counting argument.  It says, roughly speaking, that given a triple of $\e$-regular pairs, most edges are in the expected number of triangles. This is a standard fact, and we include a proof in the appendix for the sake of completeness.

\begin{lemma}\label{lem:standreg}
Let $d>0$ and assume $0<\e<\min\{d^8, 2^{-8}\}$.  Suppose $G=(U\cup V\cup W, E)$ is a tripartite graph, and assume each of $G[U,V]$, $G[U,W]$, and $G[V,W]$ is $\e$-regular with densities $d_{UV}, d_{UW}, d_{VW}\geq d$.  Let 
$$
Y=\{uv\in E\cap K_2[U,V]: |N_G(u)\cap N_G(v)|=(1\pm \e^{1/8})d_{UW}d_{UV}|W|\}.
$$
Then $|Y|\geq (1- \e^{1/8})|E\cap K_2[U,V]|$.
\end{lemma}

We use the following version for $\dev_2$ which follows from Lemma \ref{lem:standreg} and Theorem \ref{thm:equiv}.

\begin{lemma}\label{lem:standdev}
Let $d>0$ and $0<\e<\min\{d^{800}, 2^{-800}\}$. Suppose $G$ is a triad with vertex sets $U,V,W$, whose component graphs have $\dev_2(\e)$ and densities $d_{UV},d_{UW},d_{VW}\geq d$.  Let 
$$
Y=\{uv\in E_{UV}: |N_{E_{UW}}(u)\cap N_{E_{VW}}(v)|=(1\pm \e^{1/100})d_{UW}d_{VW}|W|\}.
$$
Then $|Y|\geq (1- \e^{1/100})|E_{UV}|$.
\end{lemma}

We will next state a slicing lemma for regular triads.  The author could not find a quantitative version of this in the literature, so we will provide a proof in the appendix.  Said proof uses a recent result of Nagle, R\"{o}dl, and Schacht \cite{NRS} giving a quantitative relationship between $\dev_{2,3}$ and another type of quasirandomness called $\disc_{2,3}$.  We now state the desired slicing lemma (see also the notation defined at the end of Subsection \ref{ss:bigraphs}).

\begin{lemma}\label{lem:sldev23}
There is $K\geq 1$ and a polynomial $p(x,y,z)$ so that the following hold.  Suppose $0<\e_1,\e_2,d_2,\gamma<1$ satisfy $\e_2<p(\e_1,d_2,\gamma)$ and $G=(A,B,C; E_{AB},E_{BC},E_{AC})$ is a triad whose component bigraphs each have density at least $d_2$.  Let $H=(A,B,C; F)$ be a trigraph underlied by $G$, and assume $(H|G)$ satisfies $\dev_{2,3}(\e_1,\e_2)$.  

Suppose $A'\subseteq A$, $B'\subseteq B$, and $C'\subseteq C$ satisfy $|A'|\geq \gamma |A|$, $|B'|\geq \gamma |B|$, and $|C'|\geq \gamma |C|$, and 
$$
H'=H[A',B',C']\text{ and }G'=G[A',B',C'].
$$ 
Then $(H'|G')$ satisfies $\dev_{2,3}(\gamma^{-K}\e_1^{1/K},2\gamma^{-1}\e_2^{1/12})$.
\end{lemma}

We will have frequent use of the following corollary, which allows us to refine the vertex partition of a regular decompositions.

\begin{corollary}\label{cor:sldev23}
For all $C\geq 1$, there exists $K\geq 1$ and a polynomial $q(x,y)$ so that the following hold. Assume $\e_1>0$ is sufficiently small, $\ell,t\geq 1$, and $\e_2:\mathbb{N}\rightarrow (0,1)$ satisfies $\e_2(\ell)<q(\e_1,1/\ell)$. 

Suppose $H=(V,E)$ is a $3$-graph and $\calP$ is an $\dev_{2,3}(\e_1,\e_2(\ell))$-regular $(t,\ell,\e_1,\e_2(\ell))$-decomposition for $H$. Let $\calP'$ be a $(t',\ell)$-decomposition with the following properties.
\begin{enumerate}
\item $\calP_1'\preceq \calP_1$, and each set in $\calP_1$ is the union of at most $C$ elements of $\calP_1'$, 
\item For every $P'\in \calP_2'$, $P'=P\cap (X\times Y)$ for some $P\in \calP_2$ and $X,Y\in \calP_1'$. 
\end{enumerate}
Then $\calP'$ is a $\dev_{2,3}(\e_1',\e_2'(\ell))$-regular $(t',\ell, \e_1',\e_2'(\ell))$-decomposition for $H$, where 
$$
\e_1'=4\e_1^{1/2K}C^{K}\text{ and }\e_2'(\ell)=2C\e_1^{-1/2K^2}\e_2(\ell)^{1/12}.
$$ 
\end{corollary}
\begin{proof}
Fix $C\geq 1$.  Let $K$ and $p(x,y,z)$ be as in Lemma \ref{lem:sldev23} and define $q(x,y)=p(x,xy, x/C)$.  Suppose $\e_1>0$ is sufficiently small, $\ell,t\geq 1$, and $\e_2:\mathbb{N}\rightarrow (0,1)$ satisfies $\e_2(\ell)<q(\e_1,1/\ell)$.

Assume now $H=(V,E)$ is a $3$-graph and $\calP$ is an $\dev_{2,3}(\e_1,\e_2(\ell))$-regular $(t,\ell,\e_1,\e_2(\ell))$-decomposition for $H$, and let $\calP'$ be a $(t',\ell)$-decomposition as in the hypotheses.  Let $\Omega$ be the set of triads from $\calP$ which are $\e_1$-non-trivial and $\dev_{2,3}(\e_1,\e_2(\ell))$-regular with respect to $H$.  Since $\calP$ is $\dev_{2,3}(\e_1,\e_2(\ell))$-regular and by Lemma \ref{lem:nontrivialtriad}, $|\bigcup_{G\in \Omega}K_3(G)|\geq (1-3\e_1)|V|^3$.  Set $\mu=\e^{1/2K^2}_1/C$, and note that, by construction, $\e_2(\ell)<p(\e_1,\e_1/\ell, \mu)$.  Define
$$
\calX=\{X\in \calP': \text{ there is }Y\in \calP\text{ with }X\subseteq Y\text{ and }|X|\geq \mu |Y|\}.
$$
Now let 
\begin{align*}
\Omega'&=\{G'\in \triads(\calP'):\text{ $K_3(G')\subseteq K_3(G)$ for some $G\in \Omega$}\}\\
\Omega''&=\{G'\in \triads(\calP'):\text{ $K_3(G')\subseteq X\times Y\times Z\text{ for some }X,Y,Z\in \calX$}\}.
\end{align*}
Then 
\begin{align*}
|\bigcup_{G'\in \Omega''}K_3(G)|&\geq |\bigcup_{G\in \Omega}K_3(G)|-|\bigcup_{G'\in \Omega'\setminus \Omega''}K_3(G'')|\\
&\geq (1-3\e_1)|V|^3-\sum_{Y\in \calP}\sum_{\{X\in \calP'\setminus \calX: X\subseteq Y\}}|X||V|^2\\
&\geq (1-3\e_1)|V|^3-\sum_{Y\in \calP}|\{X\in \calP': X\subseteq Y\}|\mu |Y||V|^2\\
&\geq (1-3\e_1)|V|^3-C\mu |V|^3\\
&\geq (1-3\e_1-\e_1^{1/2K^2})|V|^3.
\end{align*}
It is not difficult to check that $3\e_1+\e_1^{1/2K^2}\leq \e_1'$ holds by definition of $\e_1'$.  Consequently, it suffices to show that every $G\in \Omega''$ is $\dev_{2,3}(\e_1',\e_2'(\ell))$-regular with respect to $H$.  By Lemma \ref{lem:sldev23}, each $G\in \Omega''$ is $\dev_{2,3}(\mu^{-K}\e_1^{1/K},2\mu^{-1}\e_2(\ell)^{1/12})$-regular with respect to $H$.  By definition of $\e_1',\e_2'(\ell)$, we are done.  
\end{proof}

We end this section with a notion of approximate refinements.  This notion has appeared throughout the related literature (see e.g. \cite{Moshkovitz.2019, Conlon.2012}).

\begin{definition}
Let $\e>0$ and $V$ a set. 
\begin{enumerate}
\item Given sets $X,Y\subseteq V$, we write $X\subseteq_{\e} Y$ if $|X\setminus Y|\leq \e|X|$.  
\item Given two partitions $\calP,\calQ$ of $V$, write $\calQ\preceq_{\e}\calP$ if for all $Q\in \calQ$, there is $P\in \calP$ with $Q\subseteq_{\e}P$.
\end{enumerate}
\end{definition}

We will use the fact that we can always find equitable, approximate refinements. This fact is standard, but we include a proof in the appendix for completeness.

\begin{fact}\label{fact:refinements}
Suppose $\calP,\calP'$ are partition of $V$ of size $s,t$ respectively.  There is an equipartition $\calQ$ of $V$ with size at most $\e^{-1} s t$ so that  $\calQ\preceq_{\e}\calP$ and $\calQ\preceq_{\e}\calP'$.
\end{fact}

\subsection{Proof of the Theorem \ref{thm:main0}}\label{ss:1.9}

In this subsection, we prove Theorem \ref{thm:main0}. We begin with some notation.  Given an integer $n$ and $0<\e<2^{-700}$,  let $G_{\e}(n)=(V(n),E_{\e}(n))$ be the graph in Theorem \ref{thm:conlonfox} on $n$ vertices.  We define three auxiliary vertex sets $A(n)$, $B(n)$ and $C(n)$ as follows.  Let 
$$
A(n)=\{a_v: v\in V(n)\},\text{ and }B(n)=\{b_v: v\in V(n)\},
$$
and let $C(n)$ be any set of size $n$.  We then define the $3$-uniform hypergraph 
$$
H_{\e}(n)=(A(n)\cup B(n)\cup C(n),F_{\e}(n))
$$
 where $F_{\e}(n)=\{a_vb_{v'}z: vv'\in E_{\e}(n), z\in C(n)\}$.  In the notation of the introduction, $H_{\e}(n)=n\otimes \bip(G_{\e}(n))$.  We now verify this hypergraph has $\VC_2$-dimension at most $1$.

\begin{fact}\label{fact:hvc2}
For all $\e>0$ and $n\geq 1$, $H_{\e}(n)$ has $\VC_2$-dimension at most $1$.
\end{fact}
\begin{proof}
Suppose towards a contradiction there exist $\{x_1,x_2,y_1,y_2\}\in A(n)\cup B(n)\cup C(n)$ and for all $S\subseteq \{1,2\}^2$ an element $z_S\in A(n)\cup B(n)\cup C(n)$ so that $z_Sx_iy_j\in F_{\e}(n)$ if and only if $(i,j)\in S$.   Then there is some $z$ so that $zx_1y_1,zx_2y_2\in F_{\e}(n)$ but $zx_2y_1,zx_1y_2\notin F_{\e}(n)$.  

Suppose $z\in C(n)$.  Since $zx_1y_1,zx_2y_2\in F_{\e}(n)$ and $z\in C(n)$, there must be some $vv', ww'\in E_{\e}(n)$, $\{x_1,y_1\}=\{a_v,b_{v'}\}$ and $\{x_2,y_2\}=\{a_w,b_{w'}\}$.   But now there exists no $z'\in A(n)\cup B(n)\cup C(n)$ so that $z'x_1y_1\in F_{\e}(n)$ but $z'x_2y_2\notin F_{\e}(n)$, a contradiction.  

Suppose now $z\in A(n)$.  Since $zx_1y_1\in F_{\e}(n)$, there is some $vv'\in E_{\e}(n)$ so that either 
\begin{enumerate}[(i)]
\item $\{z,x_1\}=\{a_v,b_{v'}\}$ and $y_1\in C(n)$ or 
\item $\{z,y_1\}=\{a_v,b_{v'}\}$ and $x_1\in C(n)$.  
\end{enumerate}
Similarly, since $zx_2y_2\in F_{\e}(n)$, there is some $ww'\in E_{\e}(n)$ so that either 
\begin{enumerate}[(a)]
\item $\{z,x_2\}=\{a_w,b_{w'}\}$ and $y_2\in C(n)$ or 
\item $\{z,y_2\}=\{a_w,b_{w'}\}$ and $x_2\in C(n)$.  
\end{enumerate}
If (i) and (a) both held, then we would have $zx_2y_1\in F_{\e}(n)$, a contradiction.  Similarly, if (ii) and (b) both held, we would have $zx_1y_2\in F_{\e}(n)$, a contradiction.  Thus one of the following hold.
\begin{enumerate}[(I)]
\item (i) and (b) hold: $\{z,x_1\}=\{a_v,b_{v'}\}$, $\{z,y_2\}=\{a_w,b_{w'}\}$, and $y_1,x_2\in C(n)$ or
\item (ii) and (a) hold: $\{z,y_1\}=\{a_v,b_{v'}\}$, $\{z,x_2\}=\{a_w,b_{w'}\}$, and $y_2,x_1\in C(n)$.
\end{enumerate}
In case (I) there exists no $z'$ with $z'x_2y_1\in F_{\e}(n)$, a contradiction, and in case (II) there exists no $z'$ with $z'x_1y_2\in F_{\e}(n)$, a contradiction.

A similar argument arrives at a contradiction if $z\in B(n)$.
\end{proof}

We now state a version Theorem \ref{thm:main0} where the regular decomposition is assumed to refine a distinguished tripartition.  

\begin{theorem}\label{thm:main3}
There are constants $C>0$, $M\geq 1$, and $\e^*>0$ so that for all $0<\e_1<\e^*$, there exists $\e_2:\mathbb{N}\rightarrow (0,1]$ so that for any $\ell\geq 1$ and any $N\geq 1$,  there exists a bipartite graph $G=(A\cup B, E)$ on at least $N$ vertices so that the following holds.

Suppose $H=(A\cup B\cup C, F)$ is $N\otimes G$ and $\calP$ is a $\dev_{2,3}(\e_1,\e_2(\ell))$-regular $(t,\ell,\e_1,\e_2(\ell))$-decomposition for $H$ in which $\calP_1$ refines $\{A,B,C\}$.  Then $t\geq W(C(\e_1^{-1/M}))$.
\end{theorem}

Theorem \ref{thm:main3} is simpler to prove for notational reasons, and suffices for our applications.  Theorem \ref{thm:main0} as stated in the introduction follows immediately from Theorem \ref{thm:main3} and Corollary \ref{cor:sldev23} (see the proof of Corollary \ref{cor:wowzer} later on).  We will now prove Theorem \ref{thm:main3}.

\vspace{2mm}

\noindent{\bf Proof of Theorem \ref{thm:main3}.} 
Our goal is to prove the following: there exist constants $C>0$, $M\geq 1$, and $\e_1^*>0$ so that for and all $0<\e_1<\e_1^*$ there exists a function $\e_2:\mathbb{N}\rightarrow (0,1)$ so that for all $\ell,N\geq 1$, there is $n\geq N$ and a bipartite graph $G=(A\cup B,E)$ on $n$ vertices so that for any $\dev_{2,3}(\e_1,\e_2(\ell))$-regular $(t,\ell,\e_1,\e_2)$-decomposition $\calP$ of $n\otimes G=(A\cup B\cup C, F)$ with $\calP_1$ refining $\{A,B,C\}$  has $t\geq W(C(\e_1^{-1/M}))$.  Suppose towards a contradiction this is false.  Then the following holds.

\begin{align}\label{hyp}
&\text{For all $C>0$, $M\geq 1$, and $\e_1^*>0$, there is $0<\e_1<\e_1^*$ so that for all $\e_2:\mathbb{N}\rightarrow (0,1]$,}\\
&\text{there are $\ell,N\geq 1$ so that for all $n\geq N$, and all bipartite graphs $G=(A\cup B,E)$}\nonumber\\
&\text{on $n$ vertices, there is a $\dev_{2,3}(\e_1,\e_2(\ell))$-regular $(t,\ell,\e_1,\e_2)$-decomposition $\calP$ of}\nonumber\\
&\text{$n\otimes G=(A\cup B\cup C,F)$ with $\calP_1$ refining $\{A,B,C\}$ and $1\leq t< W(C(\e_1^{-1/M}))$.}\nonumber
\end{align}

We assume (1) holds, and after a somewhat lengthy argument, arrive at a contradiction.  We begin by defining several parameters.  Let $K_0=K_0(2)$ be as in Proposition \ref{prop:suffvc2}, and let $C_0$ be as in Theorem \ref{thm:conlonfox}.  Let $C_1=C_0400000^{-1/7}$ and set $K=\max\{K(2),2\}$.  Now choose $\mu_*\in (0,1)$ sufficiently small so that for all $0<x<\mu^*$,
$$
W(C_1x^{-1})>x^{-140K}W(C_1x^{-1}/2)^2.
$$
Set $C=C_1/2$ and define
$$
\e_1^*=\min\{2^{-K_0}, (\mu^*)^{140K}, (2^{700}400000)^{-40K}\}.
$$
Let $\e_1$ be as guaranteed to exist for $C$, $M=280K$, and $\e_1^*$ by (\ref{hyp}).  Let $\e_2^*:\mathbb{N}\rightarrow (0,1)$ be as in Proposition \ref{prop:suffvc2} for $k=2$, $\e_1$ and $\mu=\e_1$.  Define a function $\phi:(0,1)\rightarrow \mathbb{N}$ by setting 
$$
\phi(x)=W(C(x^{-1/280K})),
$$
and let $\e_2:\mathbb{N}\rightarrow (0,1]$ be defined as follows (recall Definition \ref{def:tower}). For all $x\in \mathbb{N}$, set
$$
\e_2(x)=\min\Big\{\e_2^*(x),\frac{\e_1^4}{2\phi(\e_1)^2xTw_{\e_1^{-2}\phi(\e_1)^2}(\e_1^{-4}\phi(\e_1)^2)}, \Big(\frac{\e_1}{4x}\Big)^{1200000}\Big\}.
$$
Let $\ell$ and $N$ be as guaranteed to exist by (\ref{hyp}) for this choice of $\e_2$. Set
$$
\e=400000\e_1^{1/40K}.
$$
Clearly $\e<2^{-700}$, so we can take $G_{\e}(n)=(V(n),E_{\e}(n))$ to be a graph as in Theorem \ref{thm:conlonfox} for this $\e$ and some sufficiently large $n\geq \max\{3\e_1^{-1},N\}$.  Let $H=(A\cup B\cup C,F)$, where $A=\{a_v: v\in  V(n)\}$, $B=\{b_v: v\in V(n)\}$, $C$ is any set of size $n$, and 
$$
F=\{a_xb_yz: xy\in E_{\e}(n), z\in C\}.
$$
By construction, $H=n\otimes \bip(G_{\e}(n))$, where $ \bip(G_{\e}(n))=(A\cup B, \{a_xb_y: xy\in E_{\e}(n)\})$. Therefore our assumption (\ref{hyp}) implies there exists some $1\leq t\leq W(C(\e_1^{-1/M}))=\phi(\e_1)$ and a $\dev_{2,3}(\e_1,\e_2(\ell))$-regular $(t,\ell,\e_1,\e_2)$-decomposition $\calP$ for $H_{\e}(n)$ such that $\calP_1$ refines $\{A,B,C\}$. Fix enumerations for $\calP_1$, $\calP_2$ of the following form (we allow some of the sets below to be empty to ease notation).
$$
\calP_1=\{A_i,B_i,C_i: i\in [t]\}\text{ and }\calP_2=\{P_{A_iB_j}^{\alpha},P_{B_jC_k}^{\alpha},P_{A_iC_k}^{\alpha}: 1\leq i,j,k\leq t, 1\leq  \alpha\leq \ell\}.
$$
We will write $G_{ijk}^{\alpha\beta\gamma}$ for the triad $(A_i,B_j,C_k; P_{A_iB_j}^{\alpha}, P_{A_iC_k}^{\beta}, P_{B_jC_k}^{\gamma})$.  To ease notation, for the rest of the proof, we will write $\e_2$ as shorthand for $\e_2(\ell)$ unless  there is a possibility for confusion.

The idea of our proof is to show the edges of $\bip(G_{\e}(n))$ are extremely well approximated by unions of elements from $\calP_2$.  Because the elements of $\calP_2$ are highly regular graphs, this will allow us to produce an extremely regular partition for $G_{\e}(n)$ with too few parts, contradicting the fact that $G_{\e}(n)$ comes from Theorem \ref{thm:conlonfox}.  

The bulk of the proof is concerned with the first step, namely showing the edges of $\bip(G_{\e}(n))$ are well approximated by unions of elements from $\calP_2$.  This will require several notational definitions, followed by  a series of observations regarding the error triads of $\calP$.  To ease notation, let $W=A\cup B\cup C$ and set
$$
E=\{a_xb_y\in K_2[A,B]: xy\in E_{\e}(n)\}.
$$
Note that by construction $|W|=3n$.   Let 
\begin{align*}
\calP_1^{lg}&=\{X\in \calP_1: |X|\geq \e_1n/t\}\text{ and }\\
\calP_2^{lg}&=\{P\in \calP_2: \text{ for some }X,Y\in \calP_1^{lg}, P\subseteq (X\times Y)\text{ and } |P|\geq \e_1|X||Y|/\ell\}
\end{align*}
Note that a triad $G\in \triads(\calP)$ is $\e_1$-non-trivial if and only if its vertex sets are from $\calP_1^{lg}$ and its edges come from $\calP_2^{lg}$.  Now let
$$
\Gamma:=\{P\in \calP_2^{lg}: (X,Y; P) \text{ has }\dev_2(\e_2)\text{ where }X,Y\in \calP_1\text{ are such that } P\subseteq X\times Y\}.
$$
Now let 
\begin{align*}
\Omega&:=\{G\in \triads(\calP): \text{$G$ is $\e_1$-non-trivial and $\dev_{2,3}(\e_1,\e_2)$-regular with respect to $H$}\}\text{ and }\\
\calS&:=(A\times B\times C)\cap \Big(\bigcup_{G\in \Omega}K_3(G)\Big).
\end{align*}
Because $\calP$ is $\dev_{2,3}(\e_1,\e_2(\ell))$-regular and by Lemma \ref{lem:nontrivialtriad}, 
\begin{align}\label{al:s}
|\calS|\geq |A||B||C|-3\e_1|W|^3\geq (1-9\e_1)|A||B||C|.
\end{align}
For each $P_{A_iB_j}^{\alpha}\in \calP_2$, let $d_{A_iB_j}^{\alpha}=|P_{ij}^{\alpha}|/|A_i||B_j|$.  For $P_{B_jC_k}^{\alpha},P_{A_iC_k}^{\alpha}\in \calP_2$, define the notation $d_{B_jC_k}^{\alpha}$ and $d_{A_iC_k}^{\alpha}$ analogously.   With this notation, we have by Proposition \ref{prop:counting} that for all $G_{ijk}^{\alpha\beta\gamma}\in \Omega$, 
\begin{align}\label{al:triangles}
|K_3(G_{ijk}^{\alpha\beta\gamma})|=(d_{A_iB_j}^{\alpha}d_{B_jC_k}^{\beta}d_{A_iC_k}^{\gamma}\pm 4\e_2^{1/4})|A_i||B_j||C_k|.
\end{align}
 Now define
\begin{align*}
\Sigma&=\{P_{A_iB_j}^{\alpha}\in \Gamma: \text{ there exists $k\in [t]$ and $\beta,\gamma\in [\ell]$ so that $G_{ijk}^{\alpha\beta\gamma}\in \Omega$}\}.
\end{align*}
We now show that if $P_{A_iB_j}^{\alpha}\in \Sigma$, then the density of $\overline{E}$ relative to $P_{A_iB_j}^{\alpha}\cap (A_i\times B_j)$ is close to $0$ or $1$.  We will later use this to show $\overline{E}$ is well approximated by unions of elements from $\Sigma$.

\begin{claim}\label{cl:1}
Suppose $P_{A_iB_j}^{\alpha}\in \Sigma$.  Then there is $\rho_{ij}^{\alpha}\in \{0,1\}$ so that 
$$
|P_{A_iB_j}^{\alpha}\cap  \overline{E}|=(\rho_{ij}^{\alpha}\pm 2\e_1^{1/K})|P_{A_iB_j}^{\alpha}|.
$$
\end{claim}

\begin{proof}
Since $P_{A_iB_j}^{\alpha}\in \Sigma$, there are $k\in [t]$ and $\beta,\gamma\in [\ell]$ such that $G:=G_{ijk}^{\alpha\beta\gamma}\in \Omega$.  By Proposition \ref{prop:suffvc2}, and since $H$ has $\VC_2$-dimension less than $2$ (see Fact \ref{fact:hvc2}), we must have $d_H(G)\in [0,\e_1^{1/K})\cup (1-\e_1^{1/K},1]$.  Suppose first $d_H(G)\geq 1-\e_1^{1/K}$.  Let 
$$
Y=\{(x,y)\in P_{A_iB_j}^{\alpha}:|\{z\in C_k: (x,y,z)\in K_3(G)\}|=(d_{A_iC_k}^{\beta}d_{B_jC_k}^{\gamma}\pm \e_2^{1/200})|C_k|\}.
$$
Note that by our choice of the parameters, $2\e_1^{-1/4}\e_2^{1/12}<\e_2^{1/20}$. Consequently, since each of $P_{A_iB_j}^{\alpha}$, $P_{A_iB_k}^{\beta}$, and $P_{B_jC_k}^{\gamma}$ are in $\Gamma$, Lemma \ref{lem:sldev} implies $G[A_i,B_j]$, $G[A_i,C_k]$, and $G[B_j,C_k]$ have (respectively) $\dev_2(\e_2^{1/20},d_{A_iB_j}^{\alpha})$, $\dev_2(\e_2^{1/20},d_{A_iC_k}^{\beta})$, and $\dev_2(\e_2^{1/20},d_{B_jC_k}^{\gamma})$.   Thus Lemma \ref{lem:standdev} implies 
$$
|Y|=(1\pm \e_2^{1/2000})|P_{A_iB_j}^{\alpha}|.
$$
Combining with the definition of $F$, we have that $|\overline{F}\cap K_3(G)|$ is equal to the following.
\begin{align}\label{F}
&\sum_{(x,y)\in P_{A_iB_j}^{\alpha}\cap Y\cap \overline{E}}|\{z\in C_k: xyz\in K_3(G)\}|+\sum_{(x,y)\in (P_{A_iB_j}^{\alpha}\cap \overline{E})\setminus Y}|\{z\in C_k: xyz\in K_3(G)\}|\nonumber\\
&=|P_{A_iB_j}^{\alpha}\cap \overline{E}\cap Y|(d_{A_iC_k}^{\beta}d_{B_jC_k}^{\gamma}\pm \e^{1/200}_2)|C_k|+\sum_{(x,y)\in (P_{A_iB_j}^{\alpha}\cap \overline{E})\setminus Y}|\{z\in C_k: xyz\in K_3(G)\}|\nonumber\\
&=|P_{A_iB_j}^{\alpha}\cap \overline{E}|(d_{A_iC_k}^{\beta}d_{jk}^{\gamma}\pm \e_2^{1/2000})|C_k|\pm |P_{A_iB_j}^{\alpha}\setminus Y||C_k|\nonumber\\
&=|P_{A_iB_j}^{\alpha}\cap \overline{E}|(d_{A_iC_k}^{\beta}d_{B_jC_k}^{\gamma}\pm \e_2^{1/2000})|C_k|\pm \e_2^{1/2000}|P_{A_iB_j}^{\alpha}||C_k|\nonumber\\
&=|P_{A_iB_j}^{\alpha}\cap \overline{E}|(d_{A_iC_k}^{\beta}d_{B_jC_k}^{\gamma}\pm \e_2^{1/2000})|C_k|\pm\e_2^{1/2000}d_{A_iB_j}^{\alpha}|A_i||B_j||C_k|.
\end{align}
On the other hand, by Proposition \ref{prop:counting},
\begin{align*}
|\overline{F}\cap K_3(G)|=&d_H(G)|K_3(G)|=d_H(G)(d_{A_iC_k}^{\beta}d_{B_jC_k}^{\gamma}d_{A_iB_j}^{\alpha}\pm 4\e_2^{1/4})|A_i||B_j||C_k|.
\end{align*}
Combining with (\ref{F}), we have
\begin{align*}
d_H(G)(d_{A_iC_k}^{\beta}d_{jk}^{\gamma}d_{A_iB_j}^{\alpha}\pm 4\e_2^{1/4})|A_i||B_j||C_k|=&|P_{A_iB_j}^{\alpha}\cap \overline{E}|(d_{A_iC_k}^{\beta}d_{B_jC_k}^{\gamma}\pm \e_2^{1/2000})|C_k|\\
&\pm 2\e_2^{1/2000}d_{A_iB_j}^{\alpha}|A_i||B_j||C_k|.
\end{align*}
Rearranging, this implies 
\begin{align*}
&|P_{A_iB_j}^{\alpha}\cap \overline{E}|(d_{A_iC_k}^{\beta}d_{B_jC_k}^{\gamma}\pm \e_2^{1/2000})\\
&=d_H(G) \Big(d_{A_iC_k}^{\beta}d_{jk}^{\gamma}d_{A_iB_j}^{\alpha}\pm 4\e_2^{1/4}\pm 2d_{A_iB_j}^{\alpha}\e_2^{1/2000}d_H(G)^{-1})|A_i||B_j|.
\end{align*}
Since $d_H(G)\geq 1-\e_1^{1/K}>1/2$, this implies 
\begin{align*}
|P_{A_iB_j}^{\alpha}\cap \overline{E}|&\geq \frac{(1-\e_1^{1/K})(d_{A_iC_k}^{\beta}d_{B_jC_k}^{\gamma}d_{A_iB_j}^{\alpha}- 4\e_2^{1/4}- 4d_{A_iB_j}^{\alpha}\e_2^{1/2000})|A_i||B_j|}{d_{A_iC_k}^{\beta}d_{B_jC_k}^{\gamma}+\e_2^{1/2000}}\nonumber\\
&\geq (1-2\e_1^{1/K})d_{A_iB_j}^{\alpha}|A_i||B_j|\\
&=(1-2\e_1^{1/K})|P_{A_iB_j}^{\alpha}|,
\end{align*}
where second inequality is by assumption on $\e_1,\e_2$, since $G_{ijk}^{\alpha\beta\gamma}$ is $\e_1$-non-trivial, and because $K\geq 2$.   This finishes the case where $d_H(G)\geq 1-\e_1^{1/K}$.  

Suppose now $d_H(G)\leq \e_1^{1/K}$. Consider the trigraph $H'=(A_i, B_j, C_k; K_3(G)\setminus \overline{F})$.  By Lemma \ref{lem:compliment}, $(H',G)$ also has $\dev_{2,3}(\e_1,\e_2)$, and by definition, its density is $1-d_H(G)$.  A symmetric argument to the above then shows
$$
|P_{A_iB_j}^{\alpha}\setminus \overline{E}|\geq (1-2\e_1^{1/K}) |P_{A_iB_j}^{\alpha}|.
$$ 
and consequently, $|P_{A_iB_j}^{\alpha} \cap \overline{E}|\leq 2\e_1^{1/K}|P_{A_iB_j}^{\alpha}|$.  This finishes the proof of Claim \ref{cl:1}.
\end{proof}

 Claim \ref{cl:1} shows that elements in $\Sigma$ behave well with respect to $\overline{E}$.  We now want to show that most pairs in $A\times B$ come from an element in $\Sigma$.  To this end, we define 
\begin{align*}
Z&=\bigcup_{P_{A_iB_j}^{\alpha}\in \Sigma}P_{A_iB_j}^{\alpha}.
\end{align*}
We claim that  
\begin{align}\label{al:z}
|Z|\geq (1-\e_1^{1/10})|A||B|.
\end{align}
To prove (\ref{al:z}), we will use the following auxiliary set. 
\begin{align*}
\Sigma'=\{P_{A_iB_j}^{\alpha}\in \Gamma: & \text{ for at least $(1-\e_1^{1/8})|C|$ many $v\in C$,}\\
&\text{  there is $k\in [t]$ and $\beta,\gamma\in [\ell]$ so that $v\in C_k$ and $G_{ijk}^{\alpha\beta\gamma}\in \Omega$}\}.
\end{align*}
Note that by definition, $\Sigma'\subseteq \Sigma$.  Observe that
$$
\e_1^{1/8}|C||\bigcup_{P\in \calP_2\setminus \Sigma'}P\cap (A\times B)|\leq |(A\times B\times C)\setminus \calS|\leq 9\e_1|A||B||C|,
$$
where the second inequality is from (\ref{al:s}).  Thus, 
$$
|\bigcup_{P\in \calP_2\setminus \Sigma'}P\cap (A\times B)|\leq  9\e^{1/8}_1|A||B|\leq \e_1^{1/10}|A||B|.
$$
Since $\Sigma'\subseteq \Sigma$, this finishes our proof of (\ref{al:z}).     We now define
$$
\calR=\{(i,j)\in [t]^2:  |(A_i\times B_j)\cap Z|\geq (1-\e_1^{1/20})|A_i||B_j|\}.
$$
By (\ref{al:z}),
\begin{align}\label{al:cr}
|(A\times B)\cap \Big(\bigcup_{(i,j)\in \calR}(A_i\times B_j)\Big)|\geq (1-\e_1^{1/20})|A||B|.
\end{align} 
We now show that if $(i,j)\in \calR$, then $\overline{E}\cap (A_i\times B_j)$ can be well approximated using elements of $\Sigma$.  Specifically, for each $(i,j)\in \calR$, define
$$
\Gamma_{ij}=\bigcup_{\{P_{A_iB_j}^{\alpha}\in \Sigma: \rho_{ij}^{\alpha}=1\}}P_{A_iB_j}^{\alpha}.
$$
An immediate corollary of Claim \ref{cl:1} is that for any $(i,j)\in \calR$, 
\begin{align}\label{al:diff}
|(\overline{E}\Delta \Gamma_{ij})\cap (A_i\times B_j))|\leq |(A_i\times B_j)\setminus Z|&+\sum_{\{P_{A_iB_j}^{\alpha}\in \Sigma: \rho_{ij}^{\alpha}=1\}}|P_{A_iB_j}^{\alpha}\setminus \overline{E}|\nonumber\\
&+\sum_{\{P_{A_iB_j}^{\alpha}\in \Sigma: \rho_{ij}^{\alpha}=0\}}|\overline{E}\cap P_{A_iB_j}^{\alpha}|\nonumber\\
&\leq (\e_1^{1/20}+4\e_1^{1/K}) |A_i||B_j|.
\end{align}
Let 
$$
\Sigma_{err}=\Big(\bigcup_{(i,j)\in [t]^2\setminus \calR}(A_i\times B_j)\Big)\cup \Big(\bigcup_{(i,j)\in \calR}(\overline{E}\Delta \Gamma_{ij})\cap (A_i\times B_j)\Big).
$$
By  (\ref{al:cr}) and (\ref{al:diff}), $|\Sigma_{err}|\leq (2\e_1^{1/20}+4\e_1^{1/K})n^2\leq 4\e_1^{1/20K}n^2$. We now show that for all $(i,j)\in \calR$, $\overline{E}$ behaves regularly on sufficiently large sub-pairs of $(A_i,B_j)$ which also avoids $\Sigma_{err}$.

\begin{claim}\label{cl:3}
Suppose $\ell\e_2^{1/12}<\eta<1$, $0<\tau<1$, $(i,j)\in \calR$, and $X\subseteq A_i$, $Y\subseteq B_j$ satisfy $|X|\geq \eta |A_i|$, $|Y|\geq \eta|B_j|$, and $|\Sigma_{err}\cap (X\times Y)|\leq \tau |X||Y|$.  Then
$$
\frac{|\overline{E}\cap (A_i\times B_j)|}{|A_i||B_j|}\approx_{2\ell\e_2^{1/12}+\tau}\frac{|\overline{E}\cap (X\times Y)}{|X||Y|}.
$$
\end{claim}
\begin{proof}
By Theorem \ref{thm:equiv} and Fact \ref{fact:adding}, $\Gamma_{ij}$ is $\ell \e_2^{1/12}$-regular   in $(A_i,B_j)$.  Since $\eta >\ell\e_2^{1/12}$, the assumptions  $|X|\geq \eta |A_i|$ and $|Y|\geq \eta|B_j|$ imply 
$$
\frac{|\Gamma_{ij}\cap (A_i\times B_j)|}{|A_i||B_j|}\approx_{\ell\e_2^{1/12}}\frac{|\Gamma_{ij}\cap (X,Y)|}{|X||Y|}.
$$
Since $|\Sigma_{err}\cap (X\times Y)|\leq \tau|X||Y|$,
$$
\frac{|\Gamma_{ij}\cap (X\times Y)}{|X||Y|}\approx_{\tau}\frac{|\overline{E}\cap (X\times Y)|}{|X||Y|}.
$$
Thus 
$$
\frac{|\overline{E}\cap (A_i\times B_j)|}{|A_i||B_j|}\approx_{\ell\e_2^{1/12}+\tau}\frac{|\overline{E}\cap (X\times Y)|}{|X||Y|}.
$$
This finishes the proof of Claim \ref{cl:3}.
\end{proof}

We are now ready to show there are partitions of $V(n)$ violating Theorem \ref{thm:conlonfox} (recall $G_{\e}(n)$ has vertex set $V(n)$ and edge set $E_{\e}(n)$).  We being by using $\calQ_1$ to generate a partition of $V(G_{\e}(n))$.  In particular, for each $1\leq i,j\leq t$, let 
$$
X_{ij}=\{v\in V(n): a_v\in A_i, b_v\in B_j\}.
$$
Then $\calX=\{X_{ij}: 1\leq i,j\leq t\}$ is a partition of $V(n)$ into at most $t^2$ parts.  By Fact \ref{fact:refinements}, there is an equipartition $\calV=\{V_1,\ldots, V_s\}$ of $V(n)$ with $s\leq \e_1^{-2}t^2$ so that for all $1\leq u\leq s$, there is a pair $1\leq i,j\leq t$ so that  $|V_u\setminus X_{ij}|\leq \e^2_1|V_u|$.  By Theorem \ref{thm:reg}, there is an equipartition $\calU=\{U_1,\ldots, U_S\}$ refining $\calV$ which is $\e_1/s$-regular for $G$, for some 
$$
S\leq Tw_{s}(\e_1^{-1}s)\leq  Tw_{\e^{-2}_1 t^2}(\e_1^{-2}t^2).
$$ 
Our next goal is to show that $\calU$ is an $\e$-conservative refinement of $\calV$ (recall Definition \ref{def:cons}).  To do this, we pass back to $H_{\e}(n)$, where we can compute the density of $E$ using elements from $\calP_2$. To do this, we will need the following notation regarding $\calU$, $\calV$, and $\calX$.  For each $1\leq i\leq s$, let $1\leq f(i),g(i)\leq t$ be such that $|V_i\setminus X_{g(i)f(i)}|\leq \e_1^2|V_i|$.  For each $1\leq u\leq S$, let $1\leq h(u)\leq S$ be such that $U_u\subseteq V_{h(u)}$.  Define
$$
\calU'=\{U_u\in \calU: |U_u\cap X_{g(h(u))f(h(u))}|\geq \e_1|U_u|\}.
$$
One can show by a standard averaging argument that $|\calU'|\geq (1-\e_1)s$.  We now use the partitions above of $V(n)$ to generate partitions of $A$ and $B$.  For each $1\leq i\leq s$, let 
$$
A_i'=\{a_v: v\in V_i\}\text{ and }B_i'=\{b_v: v\in V_i\}.
$$
 Similarly, for each $1\leq u\leq S$, let  
 $$
 A''_u=\{a_v: v\in U_u\}\text{ and }B_u''=\{b_v: v\in U_u\}.
 $$
 We now prove a claim about the sizes of certain intersections involving these new sets.

\begin{claim}\label{cl:AB}
The following hold.  
\begin{enumerate}
\item For all $1\leq i\leq s$, $|A'_i\cap A_{g(i)}|\geq \e_1^{-1}\ell\e_2^{1/12}|A_{g(i)}|$ and $|B'_i\cap B_{f(i)}|\geq \e_1^{-1}\ell\e_2^{1/12}|B_{f(i)}|$.
\item For all $1\leq u\leq S$ such that $U_u\in \calU'$, 
$$
|A_u''\cap A_{g(h(u))}|\geq \e_1^{-1}\ell\e_2^{1/12}|A_{g(h(u))}|\text{ and }|B_u''\cap B_{f(h(u))}|\geq \e_1^{-1}\ell\e_2^{1/12}|B_{f(h(u))}|.
$$
\end{enumerate}
\end{claim}
\begin{proof}
For (a), fix $i\in [s]$.  By definition,
$$
|A'_i\cap A_{g(i)}|\geq |V_i\cap X_{g(i)f(i)}|\geq (1-\e_1^2)|V_i|=(1-\e^2_1)|A_i'|=(1-\e_1^2)n/s\geq \eta |A_{g(i)}|,
$$
where $\eta=\e_1^2(1-\e_1)/t^2$.  Note
$$
\eta=\e_1^2(1-\e_1)/t^2\geq \e_1^2(1-\e_1)/\phi(\e_1)^2>\e_1^{-1}\ell\e_2^{1/12},
$$  
where the first inequality is because $t\leq \phi(\e_1)$ and the last inequality is by definition of the function $\e_2:\mathbb{N}\rightarrow (0,1)$.  A similar argument shows $|B'_i\cap B_{f(i)}|\geq \e_1^{-1}\ell\e^{1/12}_2|B_{f(i)}|$. This finishes the proof of (a).  For (b), fix $1\leq u\leq S$ so that $U_u\in \calU'$. Then
$$
|A_u''\cap A_{g(h(u))}|\geq |U_u\cap X_{g(h(u))f(h(u))}|\geq \e_1|U_u|=\e_1|A_u''|=\e_1n/S\geq \eta'|A_{g(h(u))}|,
$$   
where $\eta'=\e_1^{3}/Tw_{\e_1^{-2}t^{-2}}(\e_1^{-2}t^2)$.  Note
$$
\eta'=\e_1^{3}/Tw_{\e_1^{-2}t^{-2}}(\e_1^{-2}t^2)\geq \e_1^{3}/Tw_{\e_1^{-2}\phi(\e_1)^{2}}(\e_1^{-2}\phi(\e_1)^2)>\e_1^{-1}\ell\e_2^{1/12},
$$
where the first inequality is because $t\leq \phi(\e_1)$ and the last inequality is again by definition of the function $\e_2:\mathbb{N}\rightarrow (0,1)$.   The proof that $|B_u''\cap B_{f(h(u))}|\geq \e_1^{-1}\ell\e_2^{1/12}|B_{f(h(u))}|$ is similar. This concludes the proof of Claim \ref{cl:AB}. 
\end{proof}

We now define $\Theta$ be the set of pairs $(U_u,U_v)$ from $\calU'$ where the following hold.
\begin{enumerate}[(a)]
\item $|\Sigma_{err}\cap (A_u''\times B_v'')|\leq 4\e^{1/40K}_1|A_u''||B_v''|$,
\item $|\Sigma_{err}\cap (A'_{h(u)}\times B'_{h(v)})|\leq 4\e^{1/40K}_1|A'_{h(u)}||B'_{h(v)}|$.
\end{enumerate}
We show $\Theta$ is large.  For (a), since $|\Sigma_{err}|\leq 4\e_1^{1/20K}n^2$ (see the remark following the definition of $\Sigma_{err}$), and because each $|A_u''|=|B_u''|=n/S$, we have 
$$
4\e^{1/40K}_1\Big|\bigcup_{\{uv\in {[S]\choose 2}: (a) \text{ fails}\}}A''_{u}\times B''_{v}\Big|=4\e^{1/40K}_1\Big|\{uv\in {[S]\choose 2}: (a) \text{ fails}\}\Big|\frac{n^2}{S^2}\leq 4\e_1^{1/20K}n^2.
$$
Consequently, 
$$
\Big|\{uv\in {[S]\choose 2}: (a) \text{ fails}\}\Big|\leq \e^{1/40K}_1S^2\leq 3\e_1^{1/40K}{S\choose 2}.
$$
For (b), since each $|A_i'|=|B_i'|=n/s$, and since for each $i\in [s]$, there are $s/S$ many $u\in [S]$ with $i=h(u)$, we have 
$$
4\e^{1/40K}_1\Big|\bigcup_{\{uv\in {[S]\choose 2}: (b) \text{ fails}\}}A'_{h(u)}\times B'_{h(v)}\Big|=4\e^{1/40K}_1\frac{n^2}{s^2}\Big|\{uv\in {[S]\choose 2}: (b) \text{ fails}\}\Big|\frac{S^2}{s^2}\leq 4\e_1^{1/20K}n^2.
$$
Consequently, $|\{uv\in {[S]\choose 2}: (b) \text{ fails}\}|\leq \e_1^{1/40K}S^2\leq 3\e_1^{1/40K}{S\choose 2}$.  Combining these estimates with the fact that $|\calU'|\geq (1-\e_1)S$, we have 
\begin{align}\label{al:theta}
|\Theta|\geq (1-7\e_1^{1/40K}){S\choose 2}.
\end{align}
We now show that for all $(U_u,U_v)\in \Theta$, 
$$
|d_{G_{\e}(n)}(U_u,U_v)-d_{G_{\e}(n)}(V_{h(u)},V_{h(v)})|\leq 10\e_1^{1/40K}.
$$
Fix $(U_u,U_v)\in \Theta$.  Since $A_u''\times B_v''\subseteq A_{f(h(u))}\times B_{g(h(v))}$ and $\bigcup_{(i,j)\in [t]^2\setminus \calR}A_i\times B_j\subseteq \Sigma_{err}$, part (a) of the definition of $\Theta$ implies $(f(h(u)),g(h(v)))\in \calR$. Since $U_u,U_v\in \calU'$, Claim \ref{cl:AB} implies $|A_u''\cap A_{f(h(u))}|>\e_1^{-1}\ell\e_2^{1/12}|A_{f(h(u))}|$ and $|B_v''\cap B_{g(h(v))}|>\e_1^{-1}\ell\e_2^{1/12}|B_{g(h(v))}|$. Therefore, part (a) of the definition of $\Theta$ and Claim \ref{cl:3} imply
\begin{align*}
\frac{|\overline{E}\cap (A_{g(h(u))}\times B_{f(h(u))})|}{|A_{g(h(u))}||B_{f(h(u))}|}\approx_{\ell\e_2^{1/12}+4\e_1^{1/40K}} \frac{|\overline{E}\cap (A''_u\cap A_{g(h(u))})\times (B''_v\cap B_{f(h(u))})|}{|A''_u\cap A_{g(h(u))}||B''_v\cap B_{f(h(u))}|}.
\end{align*}
Further, $U_u,U_v\in \calU'$ implies $|A_u''\setminus A_{g(h(u))}|\leq \e_1|A_u''|$ and $|B_v''\setminus B_{f(h(v))}|\leq \e_1|B_v''|$.  Thus,
$$
\frac{|\overline{E}\cap (A''_u\cap A_{g(h(u))})\times (B''_v\cap B_{f(h(u))})|}{|A''_u\cap A_{g(h(u))}||B''_v\cap B_{f(h(u))}|}\approx_{3\e_1}\frac{|\overline{E}\cap (A''_u\times B''_v)|}{|A''_u||B''_v|}.
$$
Combining these, we have 
\begin{align}\label{al:approx}
\frac{|\overline{E}\cap (A_{g(h(u))}\times B_{f(h(u))})|}{|A_{g(h(u))}||B_{f(h(u))}|}\approx_{\ell\e_2^{1/12}+4\e_1^{1/40K}+3\e_1} \frac{|\overline{E}\cap (A''_u\times B''_v)|}{|A''_u||B''_v|}.
\end{align}
On the other hand,  since $U_u,U_v'\in \calU'$, Claim \ref{cl:AB} implies $|A_{h(u)}'\cap A_{f(h(u))}|>\e_1^{-1}\ell\e_2^{1/12}|A_{f(h(u))}|$ and $|B_{h(v)}'\cap B_{g(h(v))}|>\e_1^{-1}\ell\e_2^{1/12}|B_{g(h(v))}|$.  Thus part (b) of the definition of $\Theta$ and Claim \ref{cl:3} imply
$$
\frac{|\overline{E}\cap (A_{g(h(u))}\times B_{g(h(v))})|}{|A_{g(h(u))}||B_{g(h(v))}|}\approx_{\ell\e^{1/12}_2+4\e^{1/40K}_1}\frac{|\overline{E}\cap ((A_{h(u)}'\cap A_{g(h(u))})\times (B_{h(v)}'\cap B_{f(h(v))}))|}{|A_{h(u)}'\cap A_{g(h(u))}||B_{h(v)}'\cap B_{f(h(v))}|}.
$$
In a similar argument to above, we have that because $|A_{h(u)}'\setminus A_{g(h(u))}|\leq \e^2_1|A_{h(u)}'|$ and $|A_{h(u)}'\setminus B_{f(h(v))}|\leq \e^2_1|B_{h(v)}'|$, 
\begin{align}\label{al:approx2}
\frac{|\overline{E}\cap (A_{g(h(u))}\times B_{g(h(v))})|}{|A_{g(h(u))}||B_{g(h(v))}|}\approx_{\ell\e^{1/12}_2+4\e^{1/40K}_1+3\e^2_1}\frac{|\overline{E}\cap(A_{h(u)}'\times B_{h(v)}')|}{|A_{h(u)}'||B_{h(v)}'|}.
\end{align}
Combining (\ref{al:approx}) and (\ref{al:approx2}), we have 
\begin{align}\label{al:approx3}
\frac{|\overline{E}\cap (A_{h(u)}'\times B_{h(v)}')|}{|A_{h(u)}'||B'_{h(v)}|}\approx_{2\ell\e^{1/12}_2+8\e^{1/40K}_1+6\e_1} \frac{|\overline{E}\cap (A''_u\times B''_v)|}{|A''_u||B''_v|}.
\end{align}
We now observe that 
$$
\frac{|\overline{E(G_{\e}(n))}\cap (U_u\times U_v)|}{|U_u||U_v|}=\frac{|\overline{E}\cap (A''_u\times B''_v)|}{|A''_u||B''_v|}
$$
and 
$$
\frac{|\overline{E(G_{\e}(n))}\cap (V_{h(u)}\times V_{h(v)})|}{|V_{h(u)}||V_{h(v)}|}=\frac{|\overline{E}\cap (A'_{h(u)}\times B'_{h(v)})|}{|A'_{h(u)}||B'_{h(v)}|}
$$
Consequently, (\ref{al:approx3}) implies 
$$
d_{G_{\e}(n)}(U_u,U_v) \approx_{2\ell\e^{1/12}_2+8\e^{1/40K}_1+6\e_1}d_{G_{\e}(n)}(V_{h(u)},V_{h(v)}).
$$
We have now shown that for all $(U_u,U_v)\in \Theta$, 
$$
|d_{G_{\e}(n)}(U_u,U_v)-d_{G_{\e}(n)}(V_{h(u)},V_{h(v)})|\leq 4\ell\e_2+8\e^{1/40K}_1+6\e_1\leq 10\e_1^{1/40K},
$$
where the last inequality is due to our choices of parameters.  Combining with (\ref{al:theta}), we have shown $\calU$ is a $10\e_1^{1/40K}$-conservative refinement of $\calV$. By Fact \ref{fact:index} and the definition of $\e$, this implies 
$$
q_{G_{\e}(n)}(\calU)\leq q_{G_{\e}(n)}(\calV)+(101)10\e_1^{1/40K}\leq q_{G_{\e}(n)}(\calV)+\e,
$$
By our choice of $\mu^*$, and since, by assumption, $\e_1<\e_1^*<(\mu^*)^{140K}$, we have  
\begin{align}\label{al:w}
\e_1^{-2}W(C_1\e_1^{-1/280K}/2)^2<W(C_1\e_1^{-1/280K})=W(C_0(\e^{-1/7})),
\end{align}
where the last inequality is by definition of $C_1$ and $\e$.  By construction,
\begin{align*}
|\calV|=s\leq \e_1^{-2}t^2\leq \e_1^{-2}\phi(\e_1)^2&=\e_1^{-2}W(C_1(\e_1^{-1/280K}))^2<W(C_0(\e^{-1/7})),
\end{align*}
where the last inequality is by (\ref{al:w}).  This contradicts Theorem \ref{thm:conlonfox} which implies $s$ must be greater than $W(C_0(\e^{-1/7}))$ because $q_{G_{\e}(n)}(\calU)\leq q_{G_{\e}(n)}(\calV)+\e$. 
\qed

\vspace{2mm}

We now prove Theorem \ref{thm:main1}, which we restate here for convenience. 

\begin{corollary}\label{cor:wowzer}
Suppose that for all $k\geq 1$, $k\otimes U(k)\in \trip(\calH)$.  Then there is $C'>0$ so that for all sufficiently small $\e_1>0$ there is $\e_2:\mathbb{N}\rightarrow (0,1]$ so that $T_{\calH}(\e_1,\e_2)\geq W(\e_1^{-C'})$.  
\end{corollary}
\begin{proof}
Let $C>0$, $K\geq 1$, and $\e^*>0$ be as in Theorem \ref{thm:main3}.  Fix $0<\e^*_1<\e^*$ and let $\e^*_2:\mathbb{N}\rightarrow (0,1)$ be as in Theorem \ref{thm:main3} for $\e^*_1$.  Let $p(x,y),D$ be as in Corollary \ref{cor:sldev23}. Let $\e_1= ( \e_1^*/12)^{2D^2}$ and let $\e_2:\mathbb{N}\rightarrow (0,1)$ be defined by setting $\e_2(x)=p(\e_1,x)(\e^*_1)^{4D^2}(\e_2^*(x))^{12}/18$.   Let $L$ be any integer so that $\psi(\e_1,\e_2,L,T_{\calH}(\e_1,\e_2),\calH)$ holds (recall the definition of $\psi$ from the introduction).  Let $n$ be sufficiently large compared to these parameters. 

Let $G$ be as in Theorem \ref{thm:main3} for $n,L$, and set $H_1=n\otimes G$.  Say $H_1=(U\cup V\cup W, E_1)$.  Since $k\otimes U(k)\in \trip(\calH)$ for all $k\geq 1$, Observation \ref{ob:universal2}  implies there exists some $H\in \calH$ so that $V(H)=U\cup V\cup W$, so that $E(H)\cap K_3[U,V,W]=E_1\cap K_3[U,V,W]$.  

By definition of $L$, there is a $\dev_{2,3}(\e_1,\e_2(\ell))$-regular $(t,\ell,\e_1,\e_2(\ell))$-decomposition $\calP$ for $H$, for some $1\leq \ell \leq L$ and $1\leq t\leq T_{\calH}(\e_1,\e_2)$.  Say $\calP_1=\{X_1,\ldots, X_t\}$.  Let $\Omega_{reg}$ be the set of $\dev_{2,3}(\e_1,\e_2(\ell))$-regular triads of $\calP$ with respect to $H$.

Define $\calP'_1=\{A_i,B_i,C_i: i\in [t]\}$, where for each $i\in [t]$, $A_i=A\cap X_i$, $B_i=B\cap X_i$, and $C_i=C\cap X_i$, and $\calP'_2=\{P\cap (X\times Y): X,Y\in \calP'_1, P\in \calP_2\}$.  Then $\calP':=(\calP_1',\calP_2')$ is a $(3t,\ell)$-decomposition of $V(H)$, satisfying the hypotheses of Corollary \ref{cor:sldev23} with constant $3$.  Let $\Omega_{reg}'$ be the set of $\dev_{2,3}(\e_1^*,\e_2^*(\ell))$-regular triads of $\calP'$ with respect to $H$.  By our choices of parameters and Corollary \ref{cor:sldev23}, $|\bigcup_{G\in \Omega_{reg}'}K_3(G)|\geq (1-\e_1^*)|V(H)|^3$.

Call a triad $G$ of $\calP'$ a \emph{crossing triad} if its vertex sets are of the form $(A_i,B_j,C_k)$. Note that if $G\in \Omega_{reg}'$ is a crossing triad, then it is $\dev_{2,3}(\e_1^*,\e_2^*(\ell))$ with respect to $H_1$, since in that case, $E(H_1)\cap K_3(G)=E(H)\cap K_3(G)$.  Since $H_1$ is tripartite, it is clear that any non-crossing triad of $\calP'$ whose component bigraphs have $\dev_{2}(\e_2^*(\ell))$ will also has $\dev_{2,3}(\e_1^*,\e_2^*(\ell))$ (because $d_{H_1}(G)=0$).  Thus, if $\Omega_{reg}''$ is the set of triads of $\calP'$ which have $\dev_{2,3}(\e_1^*,\e_2^*(\ell))$ with respect to $H_1$, then 
$$
|\bigcup_{G\in \Omega_{reg}''}K_3(G)|\geq |\bigcup_{G\in \Omega_{reg}'}K_3(G)|\geq (1-\e^*_1)|V(H)|^3=(1-\e_1^*)|V(H_1)|^3.
$$
Thus, Theorem \ref{thm:main3} implies $3t\geq W(C(\e_1^*)^{-1/K})\geq W(C((12\e_1)^{1/2D^2})^{-1/K}))$.  Clearly there is some $C'$, depending only on $K,C,D$ so that this is at least $W(\e_1^{-C'})$.  This shows $T_{\calH}(\e_1,\e_2)\geq W(\e_1^{-C'})$, as desired.
\end{proof}

We can now prove Theorem \ref{thm:strong1}.

\vspace{2mm}

\noindent{\bf Proof of Theorem \ref{thm:strong1}.}
If $n\otimes U(n)\in \trip(\calH)$ for all $n$, then by Corollary \ref{cor:wowzer},  there is $C>0$ so that for all sufficiently small $\e_1>0$ there is $\e_2:\mathbb{N}\rightarrow (0,1]$ so that  $T_{\calH}(\e_1,\e_2)\geq W(\e_1^{-C})$.  Suppose now there is $n$ so that $n\otimes U(n)\notin \trip(\calH)$.  By Theorem \ref{thm:main2intro}, we have the desired conclusions. 
\qed

\vspace{2mm}

\section{Comparison with Moshkovitz-Shapira}\label{sec:weakreg}

In this section, we discuss the results of Moshkovitz and Shapria from \cite{Moshkovitz.2019}, and how they relate to the $3$-graph used to prove Theorem \ref{thm:main0}.  We have adapted their notation slightly to fit with ours, but the changes are all cosmetic.  We begin by defining a weak version of regularity for bigraphs used in \cite{Moshkovitz.2019}.

\begin{definition}
A bipartite $G=(V_1, V_2; E)$ is a bigraph is \emph{$\langle \delta \rangle$-regular} if for all $V_1'\subseteq V_1$ and $V_2'\subseteq V_2$ with $|V_1'|\geq \delta |V_1|$ and $|V_2'|\geq \delta |V_2|$,
$$
d_G(V_1',V_2')\geq d_G(V_1,V_2)/2.
$$
\end{definition}

\begin{definition}\label{def:deltareg}
Given a graph $G=(V,E)$, a partition $\calP$ of $V$ is \emph{$\langle \delta \rangle$-regular for $G$} if there is a graph $G'=(V,E')$ so that $|E\Delta E'|\leq \delta |E|$ and so that for all $A\neq B\in \calP$, $(A,B; \overline{E'}\cap (A\times B))$ is $\langle \delta \rangle$-regular. 
\end{definition}

 Definition \ref{def:deltareg} gives rise to the following notion of a regular decomposition (recall Definition \ref{def:decomp}).

\begin{definition}
A $(t,\ell)$-decomposition $\calP$ of $V$ is called $\langle \delta \rangle$-good if all elements in $\calP_2$ are $\langle \delta \rangle$-regular.
\end{definition}

The notion of hypergraph regularity considered in \cite{Moshkovitz.2019} is defined in terms of the following auxiliary bigraphs associated to a $3$-partite $3$-graph.

\begin{definition}
Let $H=(V_1\cup V_2\cup V_3, F)$ be a tripartite $3$-graph. Given $\{i_1,i_2,i_3\}=\{1,2,3\}$ with $i_2<i_3$, define $G_H^{i_1}$ to be the bigraph with vertex sets $(V_{i_1}, (V_{i_2}\times V_{i_3}))$ and edge set 
$$
E_H^{i_1}=\{(u(v,w)): uvw\in F\}.
$$
\end{definition}

If $H=(V_1\cup V_2\cup V_3, F)$ is a tripartite $3$-graph, then any $(t,\ell)$-decomposition $\calP$ of $V$ with $\calP_1\preceq\{V_1,V_2,V_3\}$ naturally induces a vertex partition for each of the auxiliary bigraphs $G_H^1,G_H^2,G_H^3$, as we now explain.  Since $\calP_1$ refines $\{V_1,V_2,V_3\}$, there exist enumerations of the following forms (where we allow some sets to be empty for notational convenience).
\begin{align*}
\calP_1&=\{V_{ij}: 1\leq i\leq 3, 1\leq j\leq t\}\text{ and }\\
\calP_2&=\{P_{ij,i'j'}^{\alpha}: 1\leq i,i'\leq 3, 1\leq j,j'\leq t,\alpha\leq \ell\}.
\end{align*}
  Then for any choice of $\{i_1,i_2,i_3\}=\{1,2,3\}$ with $i_2<i_3$, this induces a partition of $V(G_H^{i_1})=V_{i_1}\cup (V_{i_2}\times V_{i_3})$ consisting of 
$$
\calP_1^{i_1}=\{V_{i_11},\ldots, V_{i_1t}\}\cup \{P_{i_2j,i_3j'}^{\alpha}: 1\leq j,j'\leq t, \alpha\leq \ell\}.
$$
Note $|\calP_1^{i_1}|\leq t+t^2\ell$.  We can now define the notion of hypergraph regularity used in \cite{Moshkovitz.2019}.

\begin{definition}
Suppose $H=(V_1\cup V_2\cup V_3, F)$ is a tripartite $3$-graph satisfying $|V_1|=|V_2|=|V_3|=n$.  Assume $\calP$ is a $(t,\ell)$-decomposition of $V$ with $\calP_1\preceq \{V_1,V_2,V_3\}$. We say $\calP$ is \emph{$\langle \delta \rangle$-regular for $H$} if it is $\langle \delta \rangle$-good and for each $1\leq i_1\leq 3$, $\calP_1^{i_1}$ is a $\langle \delta \rangle$-regular partition of $G_H^{i_1}$. 
\end{definition}

We now restate the theorem of Moshkovitz and Shapira.

\begin{theorem}[Theorem 4,  \cite{MoshkovitzSimple}]\label{thm:MSrestate}
For all $s\geq 1$ there exists a $3$-partite $3$-uniform hypergraph $H$ with edge density at least $2^{-s-3}$ and a partition $\calV_0$ of $V(H)$ of size at most $3^{200}$ so that if $\calP$ is a $\langle 2^{-73}\rangle$-regular partition with vertex partition $\calP_1\preceq \calV_0$, then $|\calP_1|\geq W(s)$.
\end{theorem}

We will show that in contrast to the hypergraph constructed by Moshkovitz and Shapira to prove Theorem \ref{thm:MSrestate}, the $3$-graph used in Theorem \ref{thm:main0} admits $\langle \delta \rangle$-regular decompositions with a relatively small (specifically tower) bound on the number of vertex parts.    In fact we show this is true for any $3$-partite $3$-graph whose ternary edges are built from triangles of an underlying tripartite graph (as is the case for the hypergraph in Theorem \ref{thm:main0}).  This shows the example of Moshkovitz and Shapira is accomplishing something more complicated than the example used to prove Theorem \ref{thm:main0}.

\begin{theorem}\label{thm:weakregthm}
Let $0<\delta<1/8$, $0<\rho<1$, and assume $n\geq 1$ is sufficiently large.  Suppose $V_1,V_2,V_3$ are disjoint sets of size $n$, $G=(V_1, V_2, V_3;E)$ is a triad, and $H=(V_1\cup V_2\cup V_3,F)$ is a tripartite $3$-graph where $\overline{F}=K_3(G)$ and $|\overline{F}|=\rho|V_1||V_2||V_3|$.  Suppose $\calV_0$ is a partition of $V(H)$ refining $\{V_1,V_2,V_3\}$ of size at most $3^{200}$.  Then there exists a decomposition $\calP=(\calP_1,\calP_2)$ of $V(H)$, with $\calP_1\preceq \calV_0$, which is $\langle \delta \rangle$-regular for $H$, and with $|\calP_1|\leq 3^{200}Tw_3(\delta^{-1000}\rho^{-1000})$.  
\end{theorem}
\begin{proof}
Fix $0<\delta<1/8$, $0<\rho<1$, and $n\geq 1$ sufficiently large.  Suppose $V_1,V_2,V_3$, $G=(V_1,V_2,V_3;E)$, and $H=(V_1\cup V_2\cup V_3; F)$ satisfy the hypotheses of Theorem \ref{thm:weakregthm}.    Let $\calV_0=\{X_1,\ldots, X_s\}$ be a partition of $V(H)$ refining $\{V_1,V_2,V_3\}$ with $|\calV_0|:=C\leq 3^{200}$ parts.  

Let $\e=\rho^{100}\delta^{100}/C$. By Theorem \ref{thm:reg}, there exists an $\e^2$-regular partition $\calQ$ for $G$ with $\calQ\preceq\{V_1,V_2,V_3$ and so that  $|\calQ|:=t\leq Tw_3(\e^{-10})=Tw_3(\rho^{-1000}\delta^{-1000})$ parts.   Fix an enumeration
$$
\calQ=\{V_{ij}: 1\leq i\leq 3, 1\leq j\leq t\},
$$
 where for each $i\in [3]$, $V_i=V_{i1}\cup \ldots \cup V_{it}$. Let $\calP_1$ be the common refinement of $\calQ$ and $\calV_0$ and fix an enumeration
 $$
 \calP_1=\{V_{iju}: 1\leq i\leq 3,1\leq j\leq t, 1\leq u\leq C\},
 $$
 where, for each $i\in [3]$, $j\in [t]$, and $u\in [C]$, $V_{iju}=V_{ij}\cap X_u$.  For each $1\leq i_1,i_2\leq 3$ and $1\leq j_1,j_2\leq t$, let 
 $$
 d^E_{i_1j_1u_1,i_2j_2u_1}:=\frac{|E\cap (V_{i_1j_1u_1}\times V_{i_2j_2u_2})|}{|V_{i_1j_1u_1}||V_{i_2j_2u_2}|}.
 $$
 We now define three sets of ``bad pairs" from $\calP_1$. 
 \begin{align*}
\Theta_1&=\{(V_{i_1j_1u_1},V_{i_2j_2u_2})\in \calP_1^2: (V_{i_1j_1},V_{i_2j_2})\text{ is not  $\e^2$-regular with respect to $G$}\},\\
\Theta_2&=\{(V_{i_1j_1u_1},V_{i_2j_2u_2})\in \calP_1^2: d^E_{i_1j_1,i_2j_2}< \delta \rho\},\text{ and }\\
\Theta_3&= \{(V_{i_1j_1u_1},V_{i_2j_2u_2})\in \calP_1^2: \text{ either $|V_{i_1j_1u_1}|\leq \e^{1/2}|V_{i_1j_1}|/C$ or $|V_{i_2j_2u_2}|\leq \e^{1/2}|V_{i_2j_2}|/C$}\}.
\end{align*}
Since $\calQ$ is $\e^2$-regular, $ |\bigcup_{(X,Y)\in \Theta_1}X\times Y|\leq \e^2|V(H)|^2\leq 9\e^2n^2$.  By definition of $\Theta_3$, $|\bigcup_{(X,Y)\in \Theta_3}X\times Y|\leq \sqrt{\e}n^2$.    Consequently, if we define $\Theta=\Theta_1\cup \Theta_2\cup \Theta_3$ and set 
$$
E_{err}=\bigcup_{(X,Y)\in \Theta}E\cap (X\times Y),
$$    
then 
\begin{align*}
|E_{err}|\leq 9\e^2n^2+\sum_{(X,Y)\in \Theta_2}\delta\rho|X||Y|+\sqrt{\e}n^2&= 9\e^2n^2+\delta\rho\sum_{(X,Y)\in \Theta_2}|X||Y|+\sqrt{\e}n^2\\
&\leq (9\e^2+3\delta \rho+\sqrt{\e})n^2.
\end{align*}
 Note that by Lemma \ref{lem:slreg}, if $(X,Y)\in \calP_1^2\setminus \Theta$, then $(X,Y)$ is $2C\e^{3/2}$-regular in $G$.   
 
 We next define the $\calP_2$ part of the our desired decomposition.  In particular, given $(V_{i_1j_1u_1},V_{i_2j_2u_2})\in \calP_1^2$, we will define a partition $V_{i_1j_1u_1}\times V_{i_2j_2u_2}=P_{i_1j_1u_1,i_2j_2u_2}^1\cup P_{i_1j_1u_1,i_2j_2u_2}^0$.  First, for all $(V_{i_1j_1u_1},V_{i_2j_2u_2})\in \calP_1^2$ satisfying either $(V_{i_1j_1u_1},V_{i_2j_2u_2})\in \Theta$ or $V_{i_1j_1u_1}=V_{i_2j_2u_2}$, define 
 \begin{align*}
P_{i_1j_1u_1,i_2j_2u_2}^1&=(V_{i_1j_1u_1}\times V_{i_2j_2u_2})\text{ and }\\
P_{i_1j_1u_1,i_2j_2u_2}^2&=\emptyset.
\end{align*}
Now for each distinct pair $(V_{i_1j_1u_1},V_{i_2j_2u_2})\in  \calP_1^2\setminus \Theta$, set 
\begin{align*}
P_{i_1j_1u_1,i_2j_2u_2}^1&=E\cap (V_{i_1j_1u_1}\times V_{i_2j_2u_2})\text{ and }\\
P_{i_1j_1u_1,i_2j_2u_2}^0&=(V_{i_1j_1u_1}\times V_{i_2j_2u_2})\setminus E.
\end{align*}
Now define 
$$
\calP_2=\{P_{i_1j_1u_1,i_2j_2u_2}^{\alpha}: \alpha\in \{0,1\}, V_{i_1j_1u_1}, V_{i_2j_2u_2}\in \calP_1\}.
$$
Then $\calP:=(\calP_1,\calP_2)$ a $(3t, 2)$-decomposition of $V_1\cup V_2\cup V_3$.  It is easy to see that by construction, $\calP$ is a $\langle \delta \rangle$-good $(3t,2)$-decomposition of $V$ (since $\e\ll 1/2, \delta,\rho$).   We now show $\calP$ is $\langle \delta \rangle$-regular for $H$.  We will just prove the induced partition on $V(G_H^1)$ is $\langle \delta \rangle$-regular for $G_H^1$, as the proofs for $G_H^2$ and $G_H^3$ are then identical.

Recall $G_H^1$ has vertex sets $(V_1, (V_2\times V_3))$ and edge set $E_H^1=\{(u,(vw)): uvw\in K_3^{(2)}(G)\}$.  Note $|E(G_H^1)|=|\overline{F}|=\rho |V_1||V_2||V_3|$, and the partition $\calP$ induces $V(G_H^1)$ is  
$$
\calP_1^1=\{V_{1ju}:1\leq j\leq t, 1\leq u\leq C\}\cup \{P^{\alpha}_{2j_2u_2,3j_3u_3}: 1\leq j_2,j_3\leq t, 1\leq u_2,u_3\leq C,\alpha\in \{0,1\}\}.
$$
We show this partition is $\langle \delta \rangle$-regular for $G_H^1$.  First, let $\Sigma_1$ be the set of pairs $(V_{1j_1u_1},P_{2j_2u_2,3j_3u_3}^{\alpha})$ from $\calP_1^1$ with the property that $\{V_{1j_1u_2},V_{2j_2u_3},V_{3j_3u_3}\}^2\cap \Theta=\emptyset$.   We show all pairs from $\Sigma_1$ are $\langle \delta \rangle$-regular.  To this end, fix $(V_{1j_1u_1},P_{2j_2u_2,3j_3u_3}^{\alpha})\in \Sigma_1$.  If $\alpha=0$, then by definition,
$$
E_H^1\cap (V_{1j_1u_1}\times P_{2j_2u_2,3j_3u_3}^{\alpha})=\emptyset,
$$
so $\langle \delta \rangle$-regularity holds trivially.  So assume $\alpha=1$. In this case, by definition of $\Sigma_1$ definition of $\Theta$, and Proposition \ref{prop:counting},
\begin{align*}
|E_H^1\cap (V_{1j_1u_1}\times P_{2j_2u_2,3j_3u_3}^{\alpha})|&=|K_3(G[V_{1j_1u_1},V_{2j_2u_2},V_{3j_3u_3}])|\\
&=(d^E_{1j_1,2j_2}d^E_{1j_1,3j_3}d^E_{2j_2,3j_3}\pm 4\e^{3/8})|V_{1j_1u_1}||V_{2j_2u_2}||V_{3j_3u_3}|\\
&=(d^E_{1j_1,2j_2}d^E_{1j_1,3j_3}\pm \e^{3/8}(d^E_{2j_2,3j_3})^{-1})|V_{1j_1u_1}||P_{2j_2u_2,3j_3u_3}^{\alpha}|\\
&\leq (d^E_{1j_1,2j_2}d^E_{1j_1,3j_3}\pm \e^{1/4}(\delta \rho)^{-1})|V_{1j_1u_1}||P_{2j_2u_2,3j_3u_3}^{\alpha}|\\
&\leq (d^E_{1j_1,2j_2}d^E_{1j_1,3j_3}+(\delta\rho)^{24})|V_{1j_1u_1}||P_{2j_2u_2,3j_3u_3}^{\alpha}|,
\end{align*}
where the second to last inequality uses the definition of $\Theta$, and the last uses the definition of $\e$. Therefore 
\begin{align}\label{al:h1}
\frac{|E(G_H^1)\cap (V_{1j_1u_1}\times P_{2j_2u_2,3j_3u_3}^{\alpha})|}{|V_{1j_1u_1}||P_{2j_2u_2,3j_3u_3}^{\alpha}|}\leq d^E_{1j_1,2j_2}d^E_{1j_1,3j_3}+\delta\rho^{24}.
\end{align}
  Suppose $V'\subseteq V_{1j_1u_1}$ and $P'\subseteq P_{2j_2u_2,3j_3u_3}^{\alpha}$ have size at least $\delta |V_{1j_1u_1}|$ and $\delta |P_{2j_2u_2,3j_3u_3}^{\alpha}|$ respectively.  By Lemma \ref{lem:slreg}, $G[V',V_{2j_2u_2}]$ and $G[V', V_{3j_3u_3}]$ are $2C\e^{3/2}$-regular.  By definition of $\e$ and $\Theta$, $2C\e^{3/2}<\min\{2^{-8},d^8\}$ where $d=\min\{d^E_{1j_1,2j_2},d^E_{1j_1,3j_3},d^E_{2j_2,3j_3}\}$.  Therefore, by Lemma \ref{lem:standreg}, if we set 
$$
Y=\{(x,y)\in P_{2j_2u_2,3j_3u_3}^{\alpha}: |N_G(x)\cap N_G(y)\cap V_{1j_1'}|=d^E_{1j_1,2j_2}d^E_{1j_1,3j_3}(1\pm (2\e)^{1/8})|V'|\},
$$
then we have $|Y|\geq (1-(2\e)^{1/8})|P_{2j_2u_3,3j_3u_3}^{\alpha}|$.  Thus, 
$$
|P'\cap Y|\geq |P'|-(2\e)^{1/8}|P_{2j_2u_2,3j_3u_3}^{\alpha}|\geq (1-\delta^{-1}(2\e)^{1/8})|P|.
$$
Consequently,
\begin{align*}
|E_H^1\cap (V'\times P')|&\geq \sum_{xy\in P'\cap Y}d^E_{1j_1,2j_2}d^E_{1j_1,3j_3}(1- (2\e)^{1/8})|V'|\\
&\geq d^E_{1j_1,2j_2}d^E_{1j_1,3j_3}(1- (2\e)^{1/8})(1-\delta^{-1}(2\e)^{1/8})|P'||V'|\\
&\geq (d^E_{1j_1,2j_2}d^E_{1j_1,3j_3}-2\delta^{-1}(2\e)^{1/8})|P'||V'|.
\end{align*}
This shows the density of edges in $G_H^1$ on the sub-pair $(V',P')$ is at least
\begin{align*}
d^E_{1j_1,2j_2}d^E_{1j_1,3j_3}-2\delta^{-1}(2\e)^{1/8}&\geq  (d^E_{1j_1,2j_2}d^E_{1j_1,3j_3}+\delta\rho^{24})-(\delta\rho^{24}+2\delta^{-1}(2\e)^{1/8})\\
&\geq (d^E_{1j_1,2j_2}d^E_{1j_1,3j_3}+\delta\rho^{24})/2,
\end{align*}
where the last inequality is by definition of $\e$ and since $\delta<1/8$.   By (\ref{al:h1}) this is at least half the density of $G_H^1$ on the original pair, $(V_{1j_1u_1}, P_{2j_2u_2,3j_3u_3}^{\alpha})$.  This concludes the verification that every pair from $\Sigma_1$ is $\langle \delta \rangle$-regular.  

Let $(G_H^1)'=(V(G_H^1), (E_H^1)')$ be the graph obtained from $G_H^1$ by deleting all edges appearing in pairs outside $\Sigma_1$.  Clearly every pair in $\calP^1_1$ is now $\langle \delta \rangle$-regular with respect to $(G_H^1)'$. So we just need to check that $|E_H^1\Delta (E_H^1)'|\leq \delta |E(G_H^1)|$. We have that $|E_H^1\Delta (E_H^1)'|$ is at most
\begin{align*}
\sum_{(V_{1j_1u_1},V_{2j_2u_2})\in \Theta}\sum_{j_3\in [t]}\sum_{\alpha\in \{0,1\}}|V_{1j_1u_1}||P_{2j_2u_2,3j_3u_3}^{\alpha}|&+\sum_{(V_{1j_1u_1},V_{3j_3u_3})\in \Theta}\sum_{ j_2\in [t]}\sum_{\alpha\in \{0,1\}}|V_{1j_1u_1}||P_{2j_2u_2,3j_3u_3}^{\alpha}|\\
&+\sum_{(V_{2j_2u_2},V_{3j_3u_3})\in \Theta}\sum_{j_1\in [t]}\sum_{\alpha\in \{0,1\}}|V_{1j_1u_1}||P_{2j_2u_2,3j_3u_3}^{\alpha}|.
\end{align*}
This is at most
\begin{align*}
\sum_{(V_{1j_1u_1},V_{2j_2u_2})\in \Theta}|V_{1j_1u_1}||V_{2j_2u_2}||V_3|&+\sum_{(V_{1j_1u_1},V_{3j_3u_3})\in \Theta}|V_{1j_1u_1}||V_{3j_3u_3}||V_2|\\
&+\sum_{(V_{2j_2u_2},V_{3j_3u_3})\in \Theta}|V_1||V_{2j_2u_2}||V_{3j_3u_3}|,
\end{align*}
which we can bound from above by 
$$
3n|\Sigma_{err}|\leq 3 (9\e^2+3\delta \rho)n^3<\delta (\rho n^3)=\delta |E(G_H^1)|.
$$
This finishes the proof.
\end{proof}

\appendix

\section{ }

This appendix contains the proofs of several ingredients needed for this paper, as well as for Part 2 \cite{Terry.2024b}. 

\subsection{A proof of Theorem \ref{thm:reg2intro} with explicit bounds}\label{app:bounds}

In this section we provide an exposition of Gowers' proof from \cite{Gowers.2007} of Theorem \ref{thm:reg2intro} with explicit bounds.  We begin with several definitions and lemmas from \cite{Gowers.2007}.

\begin{definition}
Suppose $X$ is a set, $f:X\rightarrow [-1,1]$ is a function, and $\calP=\{X_1\cup \ldots \cup X_r\}$ is a partition of $X$. For each $i\in [r]$, let $\gamma_i$ and $d_i$ be such that $|X_i|=\gamma_i|X|$ and $d_i=\frac{1}{|X_i|}\sum_{x\in X_i}f(x)$.  Define the \emph{mean-square density of $f$ with respect to $\calP$} to be
$$
msd_f(\calP):=\sum_{i=1}^r\gamma_id_i^2.
$$ 
\end{definition}

\begin{definition}
Let $X_1,X_2,X_3$ be sets and $f:X_1\times X_2\times X_3\rightarrow [-1,1]$ a function.  For each $1\leq i<j\leq 3$, suppse $\calP_{ij}=\{P_{ij}^1,\ldots, P_{ij}^{m_{ij}}\}$ is a partition of $X_i\times X_j$. Let $\calP$ be the decomposition of $X_1\cup X_2\cup X_3$ consisting of $\calP_1=\{X_1,X_2,X_3\}$ and $\calP_2=\calP_{12}\cup \calP_{23}\cup \calP_{13}$.  The \emph{mean-square density of $f$ relative to $\calP$}, denoted $msd_f(\calP)$, is the mean square density  of $f$ relative to the partition $\{K_3(G): G\in Triads(\calP)\}$ of $X_1\times X_2\times X_3$.
\end{definition}

For a trigraph $H=(X_1, X_2, X_3; E(H))$, define $msd_H(\calP)$ to be $msd_f(\calP)$, where $f$ is the indicator function of $E(H)$.

\begin{lemma}[See Lemmas 8.4 and 8.9 in \cite{Gowers.2007}]\label{lem:mainlem}
Suppose 
$$
G=(X_1, X_2, X_3;P_{12},P_{23}, P_{13})
$$
is a triad, and $H=(X_1,X_2, X_3; E)$ is a trigraph underlied by $G$. For each $1\leq i<j\leq 3$, let $\delta_{ij}$ be such that $|P_{ij}|=\delta_{ij}|X_i||X_j|$ and assume $(V_i,V_j; P_{ij})$ satisfies $\dev_2(\e_2)$. Define $\delta=\delta_{12}\delta_{13}\delta_{23}$, and let $d$ be such that $|E|=d|K_3(G)|$.  Suppose $(H,G)$ fails $\dev_{2,3}(\e_1,\e_2)$ and $2^{21}\e_2^{1/4}\leq \delta^7$.  

Then for each $1\leq i<j\leq 3$, there is an integer $m_{ij}\leq 3^{\delta^{-4}}$ and a partition $\calP(P_{ij})=\{P_{ij}^1, \ldots, P_{ij}^{m_{ij}}\}$ of $P_{ij}$ so that the following holds.  For any decomposition $\calP$ of $X_1\cup X_2\cup X_3$ satisfying $\calP_1\preceq\{X_1,X_2,X_3\}$ and $\calP_2\preceq\calP(P_{12})\cup \calP(P_{13})\cup \calP(P_{23})$, we have
$$
msd_H(\calP)\geq d^2+2^{-10}\e_1^2.
$$
\end{lemma}

We now state a define a type of multi-colored regular partition.

\begin{definition}\label{def:mcreg}
Suppose $G=(V; E_1,\ldots, E_k)$ where $X,Y,Z$ are sets and $V\times V=E_1\cup \ldots \cup E_k$ is a partition.  We say a partition $\calP$ of $V$ \emph{has $\dev_2(\e)$  with respect to $G$} if there is a set $\Sigma\subseteq \calP^2$ so that $|\bigcup_{(X,Y)\in \Sigma}X\times Y|\geq (1-\e)|V|^2$ and for all $(X,Y)\in \Sigma$, the bigraph $(V,V; E_i)$ has $\dev_2(\e)$ for each $1\leq i\leq k$.
\end{definition}

We will use Theorem \ref{thm:mcreg} below, which is a multi-colored version of Szemer\'{e}di's regularity lemma.  For original references on this theorem, see  \cite{Komlos.1996, Frankl.2002}. We also refer the reader to the discussion following Theorem B.1 \cite{Terry.2021b}, where it is observed that the bound $M$ in Theorem \ref{thm:mcreg} in does not depend on $r$, the number of colors. 

\begin{theorem}[Multi-colored regularity lemma]\label{thm:mcreg}
For all $r\geq 1$, $\e>0$, and $m\geq 1$, there is $M=M(\e,m)\leq Tw_m(\e^{-5})$ such that the following holds.  Suppose $G=(V;E_1,\ldots, E_r)$ where $|V|\geq M$ and $V=E_1\cup \ldots\cup E_r$ is a partition of $V\times V$. Then for any partition $\calP=\{ V_1, \ldots, V_m\}$ of $V$, there is is a partition, $\calP'$ refining $\calP$ so that $|\calP'|\leq M$ and $\calP'$ has $\dev_2(\e)$ with respect to $G$.  
\end{theorem}

We now give a proof of the regularity lemma, Theorem \ref{thm:reg2intro}. We point out this is merely a more rigorous account of the argument given in \cite{Gowers.2007}.

\vspace{2mm}

\noindent{\bf Proof of Theorem \ref{thm:reg2intro}.}
Let $\e_1>0$ and assume $\e_2:\mathbb{N}\rightarrow (0,1]$ satisfies $\e_2(x)\leq (\frac{1}{2^{21}x})^{12}$.  Let $p(x,y)=x^{7}(100)^{12}y^{12}$ and let $f(x)=Tw(p(\e_2(x)^{-1},\e_1^{-1}))$. 

The proof strategy is as follows.  We will define a sequence of decompositions $\calP(0),\calP(1),\ldots$, so that if $(\ell_i,t_i)$ is the complexity of $\calP(i)$, and $(\ell_{i+1},t_{i+1})$ the complexity of $\calP(i+1)$, then $t_{i+1}\leq \ell_{i+1}\leq f(\ell_i)$.  We will then argue that after some bounded number steps $M$ in the sequence, the resulting decomposition $\calP(M)$ will have the desired regularity properties.

We being by defining $\calP(0)$ to be the decomposition of $V$ consisting of $\calP_1(0)=\{V\}$ and $\calP_2(0)=\{V\times V\}$.  Set $\ell_0=t_0=1$, and note $\calP(0)$ is a $(\ell_0,t_0)$-decomposition.

Suppose now $i\geq 0$, and assume we have defined a $(\ell_i,t_i)$-decomposition $\calP(i)$ of $V$, consisting of  $\calP_1(i)$ and $\calP_2(i)$, where $t_i\leq \ell_i$.

By Theorem \ref{thm:mcreg}, there exists $\calP_1(i+1)$, a refinement of $\calP_1(i)$, so that $\calP_1(i+1)$ satisfies $\dev_2(\e_2(\ell_i))$ with respect to the tuple $(V,(P)_{P\in \calP_2(i)})$, in the sense of Definition \ref{def:mcreg}. Let $t_{i+1}=|\calP_1(i+1)|$.  By Theorem \ref{thm:mcreg}, we may assume $t_{i+1}\leq Tw_{t_i}(\e_2(\ell_i)^{-5})$.  Fix an enumeration $\calP_1(i+1)=\{V_1,\ldots, V_{t_{i+1}}\}$.  Let 
$$
\calQ_2(i)=\{P\cap (V_a\times V_b): P\in \calP_2(i), 1\leq a,b\leq t_{i+1}\},
$$
and fix an enumeration of the form 
$$
\calQ_2(i)=\{P_{ab}^{\alpha}: 1\leq a,b\leq t_{i+1},1\leq \alpha\leq \ell_i\},
$$
where for each $1\leq a,b\leq t_{i+1}$, $V_a\times V_b=\bigcup_{\alpha\leq \ell_i}P_{ab}^{\alpha}$.  Given $1\leq a,b\leq t_{i+1}$ and $1\leq \alpha\leq \ell_i$, let $d_{ab}^{\alpha}$ be such that $|P_{ij}^{\alpha}|=d_{ab}^{\alpha}|V_a||V_b|$.  

Note $\calQ(i):=(\calP_1(i+1),\calQ_2(i))$ is a decomposition of $V$ of complexity $(t_{i+1},\ell_{i})$.  Given a triad $G_{abc}^{\alpha\beta\gamma}:=(V_a,V_b,V_c; P_{ab}^{\alpha}, P_{ac}^{\beta}, P_{bc}^{\gamma})$ from $Triads(\calQ(i))$, set 
$$
d_{abc}^{\alpha\beta\gamma}=|E\cap K_3(G_{abc}^{\alpha\beta\gamma})|/|K_3(G_{abc}^{\alpha\beta\gamma})|.
$$
Let $\Omega$ be the set of triads $(V_a,V_b,V_c; P_{ab}^{\alpha}, P_{ac}^{\beta}, P_{bc}^{\gamma})$ from $Triads(\calQ(i))$ satisfying
$$
\min\{|V_a|,|V_b|,|V_c|\}\geq \e_1 n/48t_{i+1}
$$
 and 
$$
P_{ab}^{\alpha}\geq \e_1|V_a||V_b|/48\ell_i, \text{ }P_{ac}^{\beta}\geq \e_1|V_a||V_c|/48\ell_i,\text{ and }P_{bc}^{\gamma}\geq \e_1|V_b||V_c|/48\ell_i.
$$
By Lemma \ref{lem:nontrivialtriad}, $|\bigcup_{G\in \Omega}K_3(G)|\geq (1-\e_1/48)|V|^3$.  Further, by Proposition \ref{prop:counting}, we have that for each $G\in \Omega$, 
$$
|K_3(G)|=(1\pm \e_2(\ell_i))d^{\alpha}_{ab}d^{\beta}_{ac}d^{\gamma}_{bc}|V_i||V_j||V_s|.
$$

Now define $\Omega_{good}$ to be the set of triads in $\Omega$ which satisfy $\dev_{2,3}(\e_1, \e_2(\ell_i))$ with respect to $H$, and let $\Omega_{\textrm{bad}}=\Omega\setminus \Omega_{good}$.  If $|\bigcup \Omega_{good}|\geq (1-\e_1)|V|^3$, then let $M=i$ and end the construction.  Otherwise, we assume $|\bigcup \Omega_{good}|< (1-\e_1)|V|^3$.  This implies 
$$
|\bigcup \Omega_{\textrm{bad}}|\geq |\bigcup \Omega|-(1-\e_1)|V|^3\geq (1-\e_1/48)|V|^3-(1-\e_1)|V|^3\geq \e_1|V|^3/2.
$$
For each triad $G_{abc}^{\alpha\beta\gamma}\in \Omega_{\textrm{bad}}$, apply Lemma \ref{lem:mainlem} to obtain integers $m_{abc,\alpha\beta\gamma}(ab)$, $m_{abc,\alpha\beta\gamma}(ac)$, and $m_{abc,\alpha\beta\gamma}(ac)$ satisfying
$$
\max\{m_{abc,\alpha\beta\gamma}(ab),m_{abc,\alpha\beta\gamma}(ac),m_{abc,\alpha\beta\gamma}(ac)\}\leq 3^{(d_{ab}^{\alpha}d_{ac}^{\beta}d_{bc}^{\gamma})^{-4}},
$$
and a partition $\calP_{abc,\alpha\beta\gamma}(P_{ab}^{\alpha})$ of $P_{ab}^{\alpha}$ of size $m_{abc,\alpha\beta\gamma}(ab)$,  a partition $\calP_{abc,\alpha\beta\gamma}(P_{ac}^{\beta})$ of $P_{ac}^{\beta}$ of size $m_{abc,\alpha\beta\gamma}(ac)$, and a partition $\calP_{abc,\alpha\beta\gamma}(P_{bc}^{\gamma})$ of $P_{bc}^{\gamma}$ of size $m_{abc,\alpha\beta\gamma}(bc)$ so that for any decomposition $\calR$ of $V_a\cup V_b\cup V_c$ with $\calR_1$ refining $\{V_a,V_b,V_c\}$ and with $\calR_2$ refining  $\calP_{abc,\alpha\beta\gamma}(P_{ab}^{\alpha})\cup \calP_{abc,\alpha\beta\gamma}(P_{ac}^{\beta})\cup \calP_{abc,\alpha\beta\gamma}(P_{bc}^{\gamma})$, we have
$$
msd_{f|_{K_3(G_{abc}^{\alpha\beta\gamma})}}(\calR)\geq (d_{abc}^{\alpha\beta\gamma})^2+2^{-10}\e_1^2.
$$

Define $\calP_2(i+1)$ be the common refinement of all partitions of the form $\calP_{abc,\alpha\beta\gamma}(P_{ab}^{\alpha})\cup \calP_{abc,\alpha\beta\gamma}(P_{ac}^{\beta})\cup \calP_{abc,\alpha\beta\gamma}(P_{bc}^{\gamma})$ for $G_{abc}^{\alpha\beta\gamma}\in \Omega_{\textrm{bad}}$. By above, $msd_H(\calP(i+1))$ is at least
\begin{align*}
&\sum_{G_{abc}^{\alpha\beta\gamma}\in \Omega_{good}}(d_{abc}^{\alpha\beta\gamma})^2|K_3(G_{abc}^{\alpha\beta\gamma})|n^{-3}+\sum_{G_{abc}^{\alpha\beta\gamma}\in \Omega_{\textrm{bad}}}((d_{abc}^{\alpha\beta\gamma})^2+ 2^{-10}\e_1)|K_3(G_{abc}^{\alpha\beta\gamma})|n^{-3}\\
&=msd_H(\calP(i))+\sum_{G\in \Omega_{\textrm{bad}}} 2^{-10}\e_1|K_3(G)|n^{-3}\\
&= msd_H(\calP(i))+2^{-10}\e_1|\bigcup_{G\in \Omega_{\textrm{bad}}}K_3(G)|n^{-3}\\
&\geq msd_H(\calP(i))+2^{-10}\e_1^2.
\end{align*}
Note for each $1\leq a,b\leq t_{i+1}$, $\calP_2(i+1)$ partitions $V_a\times V_b$ into at most the following number of parts.
\begin{align}\label{111}
\sum_{\alpha=1}^{\ell_i}\prod_{\{(c,\beta,\gamma)\in [t_{i+1}]\times [\ell_i]\times [\ell_i]: G_{abc}^{\alpha\beta\gamma}\in \Omega_{\textrm{bad}}\}} 3^{(d_{ab}^{\alpha}d_{ac}^{\beta}d_{bc}^{\gamma})^{-4}}.
\end{align}
Recall that for each $G_{abc}^{\alpha\beta\gamma}\in \Omega_{\textrm{bad}}$, since $G_{abc}^{\alpha\beta\gamma}\in \Omega$, we have 
$$
d_{ab}^{\alpha}d_{ac}^{\beta}d_{bc}^{\gamma}\geq (\e_1/48\ell_i)^3.
$$
Combining with (\ref{111}), we have that for each $1\leq a,b\leq t_{i+1}$, $\calP_2(i+1)$ partitions $V_a\times V_b$ into at most the following number of parts.
\begin{align*}
\ell_{i+1}:= \ell_i(3^{(\e_1^{-1}48\ell_i)^{12}})^{t_{i+1}\ell_i^2}=\ell_i3^{\e_1^{-12}t_{i+1}(48)^{12}\ell_i^{14}}&=\ell_i3^{\e_1^{-12}Tw_{t_i}(\e_2(\ell_i)^{-5})(48)^{12}\ell_i^{14}}\\
&\leq 3^{\e_1^{-12}Tw_{t_i}(\e_2(\ell_i)^{-5})(48)^{12}\ell_i^{15}}\\
&\leq Tw_{t_i}(\e_1^{-12}\e_2(\ell_i)^{-5}(100)^{12}\ell_i^{15})\\
&\leq Tw(t_i\e_1^{-12}\e_2(\ell_i)^{-5}(100)^{12}\ell_i^{15})\\
&\leq Tw(\e_1^{-12}\e_2(\ell_i)^{-5}(100)^{12}\ell_i^{16})\\
&\leq Tw(\e_1^{-12}\e_2(\ell_i)^{-7}(100)^{12})\\
&=f(\ell_i),
\end{align*}
where the last two inequalities use that $t_i\leq \ell_i$ and our assumption on $\e_2(x)$.  This concludes step $i+1$ of our construction.  

We note that because the mean square density goes up each step by $2^{-10}\e^2_1$, this can proceed at most some $M\leq 2^{10}\e_1^{-2}$ many times.  The resulting decomposition $\calP(M)$ will satisfy $\dev_{2,3}(\e_1,\e_2(\ell_M))$, and will have complexity $(\ell_M,t_M)$ satisfying
$$
t_M\leq \ell_M\leq f^{(M)}(1).
$$
\qed

\subsection{Basic facts from Section \ref{ss:CF}}

This section contains proofs of routine facts from Section \ref{ss:CF}.  We begin with the proof of Fact \ref{fact:index}, which says that conservative refinements have similar index values.

\vspace{2mm}

\noindent{\bf Proof of Fact \ref{fact:index}.} Say $|V|=n$.  Fix a new enumeration $\calU=\{U_{ij}: 1\leq i\leq m, 1\leq j\leq t\}$ so that $Z_i=U_{i1}\cup \ldots \cup U_{it}$.  Note $s=tm$.  Given $1\leq i,j\leq t$, let $\Sigma_{ij}$ be the set of pairs $u,v\in [t]$ so that  
$$
d_G(U_{iu},U_{jv})\approx_{\e}d_G(Z_{i},Z_{j}).
$$
By assumption, $\sum_{ij\in {[m]\choose 2}}|\Sigma_{ij}|\geq (1-\e){s\choose 2}$.  Note this implies $\sum_{i\in [m]}{[t]\choose 2}=m{t\choose 2}<2\e {s\choose 2}$.  Since $\calU$ is an equipartition, $n$ is large, and $m{t\choose 2}<2\e {s\choose 2}$, we have
\begin{align*}
q_G(\calU)&=\frac{1}{n^2}\sum_{ij\in {[m]\choose 2}} \sum_{u,v\in [t]} d_G^2(U_{iu}, U_{jv})|U_{iu}||U_{jv}|+\sum_{i\in [m]}\sum_{uv\in {[t]\choose 2}}d_G^2(U_{iu}, U_{iv})|U_{iu}||U_{iv}|\\
&\leq \frac{1+ \e^9}{s^2}\Big(\sum_{ij\in {[m]\choose 2}} \sum_{u,v\in [t]} d_G^2(U_{iu}, U_{jv})+\sum_{i\in [m]}\sum_{uv\in {[t]\choose 2}}d_G^2(U_{iu}, U_{iv})\Big)\\
&\leq \frac{1+ \e^9}{s^2}\Big(\sum_{ij\in {[m]\choose 2}} \sum_{u,v\in [t]} d_G^2(U_{iu}, U_{jv})+ 2\e{s\choose 2}\Big)\\
&\leq \Big(\frac{1+\e^9}{s^2}\sum_{ij\in {[m]\choose 2}} \sum_{u,v\in [t]} d_G^2(U_{iu}, U_{jv})\Big)+ 3\e.
\end{align*}
Now observe
\begin{align*}
\sum_{ij\in {[m]\choose 2}} \sum_{u,v\in [t]} d_G^2(U_{iu}, U_{jv})&\leq \Big(\sum_{ij\in {[m]\choose 2}} \sum_{uv\in \Sigma_{ij}} d_G^2(U_{iu}, U_{jv})\Big)+\Big(\sum_{ij\in {[m]\choose 2}}|[t]^2\setminus \Sigma_{ij}|\Big) \\
&\leq \Big(\sum_{ij\in {[m]\choose 2}} \sum_{uv\in \Sigma_{ij}} d_G^2(U_{iu}, U_{jv})\Big)+ 2\e{s\choose 2},
\end{align*}
where the second equality is by assumption.  We then have
\begin{align*}
\sum_{ij\in {[m]\choose 2}} \sum_{uv\in \Sigma_{ij}} d_G^2(U_{iu}, U_{jv})\leq \sum_{ij\in {[m]\choose 2}}|\Sigma_{ij}|(d_G^2(Z_i,Z_j)+ \e)&\leq \sum_{ij\in {[m]\choose 2}}t^2(d_G^2(Z_i,Z_j)+ \e)\\
&= t^2\sum_{ij\in {[m]\choose 2}}(d_G^2(Z_i,Z_j)+ \e)\\
&=t^2\sum_{ij\in {[m]\choose 2}}d_G^2(Z_i,Z_j)+\e t^2{m\choose 2}
\end{align*}

Combining all this we have
\begin{align*}
q_G(\calU)&\leq 3\e+\Big(\frac{1+\e^9}{s^2}\Big)2\e{s\choose 2}+\Big(\frac{1+\e^9}{s^2}\Big)t^2\sum_{ij\in {[m]\choose 2}}d_G^2(Z_i,Z_j)+\Big(\frac{1+\e^9}{s^2}\Big)\e t^2{m\choose 2}\\
&\leq 100\e+\Big(\frac{1+\e^9}{s^2}\Big)t^2\sum_{ij\in {[m]\choose 2}}d_G^2(Z_i,Z_j).
\end{align*}
Now observe, since $\calZ$ is an equipartition, $n$ is large, and $m=s/t$,
\begin{align*}
\frac{t^2}{s^2}\sum_{ij\in {[m]\choose 2}}d_G^2(Z_i,Z_j)\leq (1+\e^9)\frac{1}{n^2}\sum_{ij\in {[m]\choose 2}}d_G^2(Z_i,Z_j)|Z_i||Z_j|=(1+\e^9)q_G(\calZ).
\end{align*}
Combining we have $q_G(\calU)\leq 100\e+(1+\e^9)^2q_G(\calZ)\leq 101\e+q_G(\calZ)$, as desired.
\qed

\vspace{2mm}

We now prove Lemma \ref{lem:standreg}, which says that in a quasirandom tripartite graph, most edges are in the expected number of triangles. 

\vspace{2mm}

\noindent{\bf Proof of Lemma \ref{lem:standreg}.} Let $d_{UV}$ be such that $|E\cap K_2[U,V]|=d_{UV}|U||V|$, and define $d_{UW},d_{VW}$ analogously.  By assumption, $d_{UV},d_{UW},d_{VW|}\in (d-\e,d+\e)$.   Set
\begin{align*}
U'=\{u\in U: |N_G(u)\cap V|=(1\pm \e^{1/2})d_{UV}|V|\text{ and }|N_G(u)\cap W|=(1\pm \e^{1/2})d_{UW}|W|\}.
\end{align*}
By standard arguments, $|U'|\geq (1-2\e^{1/2})|U|$.  Now fix $u\in U'$ and set $V_u=N_G(u)\cap V$ and $W_u=N_G(u)\cap W$.  By definition of $U'$, $|V_u|=(1\pm \e^{1/2})d_{UV}|V|$ and $|W_u|=(1\pm \e^{1/2})d_{UW}|W|$.  By Lemma \ref{lem:sldev}, $G[V_u,W_u]$ has density within $\e$ of $d_{VW}$ and is $2d_{VW}^{-1}\e$-regular.  By assumption, $2d_{VW}^{-1}\e\leq \e^{1/2}$, so $G[V_u,W_u]$ is $\e^{1/2}$-regular. Let 
$$
V_u'=\{v\in V_u: |d_G(v)\cap W_u|=(1\pm 2\e^{1/4})d_{VW}|W_u|\}.
$$
Since $G[V_u,W_u]$ has density within $\e$ of $d_{VW}$ and is $\e^{1/2}$-regular, standard arguments again show that $|V_u'|\geq (1-2\e^{1/4})|V_u|$.  Note that for all $v\in V_u'$, 
\begin{align*}
|N_G(u)\cap N_G(v)|&=(1\pm 2\e^{1/4})d_{VW}|W_u|\\
&=(1\pm 2\e^{1/4})d_{VW}(1\pm \e^{1/2})d_{UW}|W|\\
&=(1\pm 2\e^{1/4})(1\pm \e^{1/2})(1\pm \e)^2d_{VW}d_{UW}|W|\\
&=(1\pm \e^{1/8})d_{VW}d_{UW}|W|,
\end{align*}
where the last inequality uses the assumptions on $\e$ and $d$. Thus for all $u\in U'$ and $v\in V_u$, $uv\in Y$.  By our size estimates above this, shows 
$$
|Y|\geq (1-2\e^{1/2})|U|(1- 2\e^{1/4})d_{UV}|V|\geq(1-2\e^{1/2})(1- 2\e^{1/4})(1-\e)d_{UV}|U||V|\geq (1-\e^{1/8})|E\cap K_2[U,V]|,
$$
where the last inequality is by assumptions on $\e$ and $d_{UV}$.
\qed

\vspace{2mm}

We now prove Fact \ref{fact:refinements}, which says we can find approximate refinements.

\vspace{2mm}

\noindent{\bf Proof of Fact \ref{fact:refinements}.} Let $\calP=\{X_1,\ldots, X_s\}$,  $\calP'=\{Y_1,\ldots, Y_t\}$, and let $n=|V|$.  For each $1\leq i\leq s$, $1\leq j\leq t$, let $Z_{ij}=X_i\cap Y_j$.  For each $Z_{ij}$, let $Z_{ij}=Z_{ij}^0\cup Z_{ij}^1\cup \ldots \cup Z_{ij}^{m_{ij}}$ be any partition with the property that for each $1\leq u\leq m_{ij}$, $|Z_{ij}^u|=\lceil \e n/ts \rceil $ and $|Z_{ij}^0|<\e n/ts$.  Let
$$
W_0=\bigcup_{1\leq i\leq s}\bigcup_{1\leq j\leq t}Z_{ij}^0.
$$ 
Note $|W_0|\leq \sum_{1\leq i\leq s}\sum_{1\leq j\leq t}|Z_{ij}^0|< ts(\e n/ts)=\e n$.  Let $W_1,\ldots, W_m$ be an enumeration of the set $\{Z_{ij}^u: 1\leq i\leq s, 1\leq j\leq t, 1\leq u\leq m_{ij}\}$.  Now let $W_0=W_1'\cup \ldots \cup W_m'$ be any equipartition, and for each $1\leq v\leq m$, define $R_v=W_v\cup W_v'$. Clearly $V=R_1\cup \ldots \cup R_m$ is an equipartition of $V$ and $m\leq \e^{-1} st$.  By construction, for all $1\leq v\leq m$, there some $Z_{ij}^u$ so that $|R_v\setminus Z_{ij}^u|\leq |W_v'|\leq \e n /m\leq \e |R_v|$.   Thus $\calQ\preceq_{\e}\calP$ and $\calQ\preceq_{\e}\calP'$, as desired.
\qed

\vspace{2mm}

\subsection{Slicing Lemma for $\dev_{2,3}$} In this subsection, we prove Lemma \ref{lem:sldev23}, which is a slicing lemma for $\dev_{2,3}$-regularity.  We begin with setting notation for the copies of $K_{2,2}$ in a bigraph.  Given a bigraph $G=(A,B; E)$, let 
$$
K_{2,2}^G=\{(a_0,a_1,b_0,b_1)\in A^2\times B^2: \text{ for each }(i,j)\in \{0,1\}^2, (a_i,b_j)\in E\}.
$$
The following fact is well known.

\begin{fact}\label{fact:c4count}
Suppose $G=(A,B; E)$ is a bigraph and $E=d|A||B|$.  Then $|K_{2,2}^G|\geq d^4|A|^2|B|^2$.
\end{fact}

Similarly, if $H=(A,B,C;F)$ is a trigraph, we define
$$
K_{2,2,2}^H=\{(a_0,a_1,b_0,b_1,c_0,c_1)\in A^2\times B^2\times C^2: \text{ for each }(i,j,k)\in \{0,1\}^3, (a_i,b_j,c_k)\in F\}.
$$
We will use the following analogue of Fact \ref{fact:c4count}, which follows from Fact \ref{fact:c4count} and several applications of Cauchy-Schwarz.
\begin{fact}\label{fact:k222count}
Let $H=(A,B,C;F)$ be a trigraph with $E=d|A||B|$.  Then $|K_{2,2,2}^G|\geq d^8|A|^2|B|^2|C|^2$.
\end{fact}
%
One important ingredient in our proof of the slicing lemma, is the following quantitative proof that $\dev_{2,3}$-quasirandomenss implies so-called ``disc" quasirandomenss, due to Nagle, R\"{o}dl, and Schacht.  To state this precisely, we need the following definition.

\begin{definition}
Let $\e_1,\e_2>0$. Suppose $H=(V_1,V_2,V_3; F)$ is a trigraph and $G=(V_1,V_2,V_3;E)$ is a triad underlying $H$.
 We say that \emph{$(H,G)$ has $\disc_{2,3}(\e_1,\e_2)$} if for each $1\leq i,j\leq 3$, $G[V_i,V_j]$ is $\e_2$-regular, and for every subgraph $G'\subseteq G$,
\[||E(H)\cap K_3(G')|-d_G(H)|K_3(G')||\leq \e_1 d_{12}d_{13}d_{23}|V_1||V_2||V_3|,\]
where for each $1\leq i<j\leq 3$, $d_{ij}=d_G(V_i,V_j)$
\end{definition}

We now state a  theorem of Nagle, R\"{o}dl, and Schacht, which gives an improved quantitative proof that $\dev_{2,3}$-quasirandomness implies $\disc_{2,3}$-quasirandomness.  

\begin{theorem}[Proposition 1.4 in \cite{NRS}]\label{thm:translateapp}
There is a polynomial $p(x)$ so that the following holds.  For all $0<\e_1,d_1,d_2<1$ and all $0<\e_2<p(d_2)$,  the following holds.  Suppose $H=(A,B,C;F)$ is a trigraph,  $G$ is a triad whose component bigraphs are all $\e_2$-regular with density at least $d_2$.  Assume $G$ underlies $H$, $d_H(G)=d_1$, and  
$$
\sum_{u_0,u_1\in A}\sum_{w_0,w_1\in B}\sum_{z_0,z_1\in C}\prod_{(i,j,k)\in \{0,1\}^3}h_{H,G}(u_i,w_j,z_k)\leq \e_1d_2^{12}|A|^2|B|^2|C|^2,
$$
where $h_{H,G}(x,y,z)=1-d_1$ if $(x,y,z)\in E\cap K_3(G)$, $h_{H,G}(x,y,z)=-d_1$ if $(x,y,z)\in K_3(G)\setminus E$, and $h_{H,G}(x,y,z)=0$ if $(x,y,z)\notin K_3(G)$.  Then $(H,G)$ has $\disc_{2,3}((2\e_1)^{1/8},\e_2)$.
\end{theorem}

\vspace{2mm}

The last lemma we need is another corollary of the counting lemma (see e.g. \cite{Gowers.20063gk}).

\begin{lemma}\label{lem:k222count}
There exists a constant $S\geq 1$ so that the following holds. Suppose $0<\e<1$ and $A,B,C$ are sets, and  $G$ is a triad with vertex sets $A,B,C$ whose component bigraphs each have $\dev_2(\e)$ and densities $d_{AB},d_{AC},d_{BC}$, respectively.  Then 
$$
||K_{2,2,2}^G|-d_{AB}^2d_{AC}^2d_{BC}^2|A|^2|B|^2|C|^2|\leq \e^{1/S}|A|^2|B|^2|C|^2.
$$
\end{lemma}

\vspace{2mm}

\noindent{\bf Proof of Lemma \ref{lem:sldev23}}
Let $S$ be as in Lemma \ref{lem:k222count} and set $K=200000000S$.  Define $p(x,y,z)=(xyz/100)^{20000000K}$.  Assume $0<\e_1,\e_2,d_2,\gamma<1$ satisfy $\e_2<p(\e_1,d_2,\gamma)$, and $G=(A,B,C;E)$, and $H=(A,B,C; F)$ satisfy the hypotheses of Lemma \ref{lem:sldev23}.

Assume $A'\subseteq A$, $B'\subseteq B$, and $C'\subseteq C$ satisfy $|A'|\geq \gamma |A|$, $|B'|\geq \gamma |B|$ and $|C'|\geq \gamma |C|$. Let $d_1=d_H(G)$. To ease notation set
\begin{align*}
A_0=A',\text{ }B_0=B',\text{ and }C_0=C'\\
A_1=A\setminus A_0,\text{ }B\setminus B_0,\text{ and }C\setminus C_0.
\end{align*}
We now fix some more notation. Given $i,j,k\in \{0,1\}$, define $d_{A_iB_j},d_{A_iC_k},d_{B_iC_k}$ be the densities of $G[A_i,B_j]$, $G[A_i,C_k]$ and $G[B_j,C_k]$ respectively, and let $G_{ijk}$ be the triad whose component bigraphs are $G[A_i,B_j]$, $G[A_i,B_j]$ and $G[B_j,C_k]$, and set $d_{ijk}=d_H(G_{ijk})$.   We now define two auxiliary parameters, 
$$
\mu=\e_1^{1/100}d_2^{6}\gamma^3\text{ and }\e_2':=2\mu^{-1}\e_2^{1/12}.
$$
We next prove a claim which says that $d_{ijk}$ is close to $d_1$ when $A_i,B_j,C_k$ are not too small.

\begin{claim}\label{cl:dense1}
Assume $i,j,k\in \{0,1\}$ are such that $|A_i|\geq \mu |A|$, $|B_j|\geq \mu |B|$ and $|C_k|\geq \mu |C|$.  Then $|d_{ijk}-d_1|\leq 2\mu^{-3}(2\e_1)^{1/8}$.
\end{claim}
\begin{proof}
By Lemma \ref{lem:sldev}, each of $G[A_i,B_j]$, $G[A_i,B_j]$ and $G[B_j,C_k]$ satisfy $\dev_2(\e_2')$ and have densities $d_{A_iB_j}=d_{AB}\pm \e_2'$, $d_{A_iC_k}=d_{AC}\pm \e_2'$, and $d_{B_jC_k}=d_{BC}\pm \e_2'$.  Combining with Proposition \ref{prop:counting}, we have
\begin{align}\label{g'}
|K_3(G_{ijk})|=(d_{A_iB_j}d_{B_jC_k}d_{A_iC_k}\pm (\e_2')^{1/4})|A_i||B_j||C_k|=(d_{AB}d_{BC}d_{AC}\pm 8(\e_2')^{1/4})|A_i||B_j||C_k|.
\end{align}
By Theorem \ref{thm:translateapp}, 
\begin{align*}
||F\cap K_3(G_{ijk})|-d_1|K_3(G_{ijk})||&=|d_{ijk}-d_1||K_3(G_{ijk})|\\
&\leq (2\e_1)^{1/8}d_{AB}d_{BC}d_{AC}|A||B||C|\\
&\leq (2\e_1)^{1/8}d_{AB}d_{BC}d_{AC}\mu^{-3}|A_i||B_j||C_k|.
\end{align*}
Combining with (\ref{g'}) this implies
$$
|d_{ijk}-d_3|\leq \frac{(2\e_1)^{1/8}d_{AB}d_{BC}d_{AC}\mu^{-3}|A_i||B_j||C_k|}{(d_{AB}d_{BC}d_{AC}- 8(\e_2')^{1/4})|A_i||B_j||C_k|}\leq 2\mu^{-3}(2\e_1)^{1/8},
$$
where the last inequality is by our assumptions on our parameters.
\end{proof}

We now set up some more notation.  First, define $h_{H,G}(x,y,z)=1-d_1$ if $(x,y,z)\in F\cap K_3(G)$, $h_{H,G}(x,y,z)=-d_1$ if $(x,y,z)\in K_3(G)\setminus F$, and $h_{H,G}(x,y,z)=0$ if $(x,y,z)\notin K_3(G)$.  Then, given $X\subseteq A$, $Y\subseteq B$, and $Z\subseteq C$, set 
$$
f_{H,G}(X,Y,Z)= \sum_{u_0,u_1\in X}\sum_{w_0,w_1\in Y}\sum_{z_0,z_1\in Z}\prod_{(a,b,c)\in \{0,1\}^3}h_{H,G}(u_a,w_b,z_c).
$$

For each $(i,j,k)\in \{0,1\}^3$, define $h_{ijk}(x,y,z)=1-d_{ijk}$ if $(x,y,z)\in F\cap K_3(G_{ijk})$, $h_{ijk}(x,y,z)=-d_{ijk}$ if $(x,y,z)\in K_3(G_{ijk})\setminus F$, and $h_{ijk}(x,y,z)=0$ if $(x,y,z)\notin K_3(G_{ijk})$, and set 
$$
f_{ijk}= \sum_{u_0,u_1\in A_i}\sum_{w_0,w_1\in B_j}\sum_{z_0,z_1\in C_k}\prod_{(a,b,c)\in \{0,1\}^3}h_{ijk}(u_a,w_b,z_c).
$$
By Fact \ref{fact:k222count}, each $f_{ijk}\geq 0$.  Given $0\leq \alpha\leq 8$ and $i,j,k\in \{0,1\}$ define 
$$
c_{\alpha}^{ijk}=\Big(\frac{d_1}{d_{ijk}}\Big)^{\alpha}\Big(\frac{1-d_1}{1-d_{ijk}}\Big)^{8-\alpha}
$$
 and set 
$$
c^{ijk}=\max\{|1-c_{\alpha}^{ijk}|: 0\leq \alpha\leq 8\}.
$$
For each $0\leq \alpha\leq 8$ and $X\subseteq A$, $Y\subseteq B$, and $Z\subseteq C$, define
\begin{align*}
I_{\alpha}[X,Y,Z]=\{(x_0,x_1,y_0,y_1,z_0,z_1)\in X^2\times &Y^2\times Z^2: \{x_0,x_1\}\times \{y_0,y_1\}\times \{z_0,z_1\}\subseteq K_3(G)\\
&\text{ and }|\{(i,j,k)\in \{0,1\}^3: (x_i,y_j,z_k)\in F\}|=\alpha\}.
\end{align*} 

We now prove that for all $i,j,k\in \{0,1\}$, $f_{H,G}(A_i,B_j,C_k)$ is not too small. 

\begin{claim}\label{cl:dense2}
For all $i,j,k\in \{0,1\}$,  $f_{H,G}(A_i,B_j,C_k)\geq -300 \mu^{-6}(2\e_1)^{1/50}d_{AB}^2d_{AC}^2d_{BC}^2|A|^2|B|^2|C|^2$.
\end{claim}
\begin{proof}
If one of $|A_i|<\mu|A|$, $|B_j|<\mu|B|$, or $|C_k|<\mu|C|$ holds, then
$$
 |f_{H,G}(A_i,B_j,C_k)|\leq |A_i|^2|B_j|^2|C_k|^2\leq \mu^2 |A|^2|B|^2|C|^2\leq 300 \mu^{-6}(2\e_1)^{1/50}d_{AB}^2d_{AC}^2d_{BC}^2|A|^2|B|^2|C|^2,
 $$
where the last inequality is by definition of $\mu$.   Thus in this case we are done. 

Assume now $|A_i|\geq \mu|A|$, $|B_j|\geq \mu|B|$, and $|C_k|\geq \mu|C|$. By definition,
\begin{align*}
f_{H,G}(A_i,B_j,C_k)&=\sum_{u_0,u_1\in A_i}\sum_{w_0,w_1\in B_j}\sum_{z_0,z_1\in C_k}\prod_{(a,b,c)\in \{0,1\}^3}h_{H,G}(u_a,w_b,z_c)\\
&= \sum_{\alpha=0}^{8}c_{\alpha}^{ijk}\Big(\sum_{\{(u_0,u_1,w_0,w_1,z_0,z_1)\in I_{\alpha}[A_i,B_j,C_k]\}}h_{ijk}(u_a,w_b,z_c)\Big).
\end{align*}
Thus, by definition of $f_{ijk}$ and $c_{ijk}$, we have
\begin{align*}
\Big|f_{ijk}-f_{H,G}(A_i,B_j,C_k)\Big|&=\Big| \sum_{\alpha=0}^{8}(1-c_{\alpha}^{ijk})\Big(\sum_{\{(u_0,u_1,w_0,w_1,z_0,z_1)\in I_{\alpha}[A_i,B_j,C_k]\}}h_{ijk}(u_a,w_b,z_c)\Big)\Big|\\
&\leq c^{ijk}|K_{2,2,2}(G_{ijk})|.
\end{align*}
Since $f_{ijk}\geq 0$, this implies $f_{H,G}(A_i,B_j,C_k)\geq -c^{ijk}|K_{2,2,2}(G_{ijk})|$.  By  Claim \ref{cl:dense1}, 
$$
|c^{ijk}|\leq 1-(1-(2\mu^{-3}(2\e_1)^{1/8}))^8<128\mu^{-3}(2\e_1)^{1/8}.
$$
Combining the above with Lemmas \ref{lem:k222count} and \ref{lem:sldev}, we have the following.
\begin{align*}
f_{H,G}(A_i,B_j,C_k)&\geq -128\mu^{-3}(2\e_1)^{1/8}|K_{2,2,2}(G_{ijk})|\\
&\geq  -128\mu^{-3}(2\e_1)^{1/8}(d^2_{A_iB_j}d^2_{A_iC_k}d^2_{B_jC_k}-(\e_2')^{1/4S})|A_i|^2|B_j|^2|C_k|^2\\
&\geq -128\mu^{-3}(2\e_1)^{1/8}(d_{AB}^2d_{AC}^2d_{BC}^2-(\e_2')^{1/8S})|A_i|^2|B_j|^2|C_k|\\
&\geq -128\mu^{-6}(2\e_1)^{1/8}(d_{AB}^2d_{AC}^2d_{BC}^2-(\e_2')^{1/8S})|A|^2|B|^2|C|,
\end{align*}
where the last inequality is by the lower bound assumptions on $|A_i|,|B_j|,|C_k|$.  By our choice of parameters, we obtain the desired inequality.
\end{proof}

We now finish the proof. By assumption, 
\begin{align*}
0\leq f_{H,G}(A,B,C)=\sum_{(i,j,k)\in \{0,1\}^3}f_{H,G}(A_i,B_j,C_k)\leq \e_1d_{AB}^4d_{AC}^4d_{BC}^4|A|^2|B|^2|C|^2,
\end{align*}
Combining with Claim \ref{cl:dense2}, we have that 
\begin{align*}
f_{H,G}(A_0,B_0,C_0)&=f_{H,G}(A,B,C)- \sum_{(i,j,k)\in \{0,1\}^3\setminus \{(0,0,0)\}}f_{H,G}(A_i,B_j,C_k)\\
&\leq \e_1d_{AB}^4d_{AC}^4d_{BC}^4|A|^2|B|^2|C|^2-\sum_{(i,j,k)\in \{0,1\}^3\setminus \{(0,0,0)\}}f_{H,G}(A_i,B_j,C_k)\\
&\leq \e_1d_{AB}^4d_{AC}^4d_{BC}^4|A|^2|B|^2|C|^2+8\cdot 300 \mu^{-3}(2\e_1)^{1/50}d_{AB}^2d_{AC}^2d_{BC}^2|A||B||C|\\
&\leq  9\cdot 300 \mu^{-3}(2\e_1)^{1/50}d_{AB}^4d_{AC}^4d_{BC}^4|A|^2|B|^2|C|^2\\
&\leq 9\cdot 300 \mu^{-3}(2\e_1)^{1/50}\gamma^{-6}d_{A_0B_0}^4d_{A_0C_0}^4d_{B_0C_0}^4|A_0|^2|B_0|^2|C_0|^2,
\end{align*}
where the last inequality is by the lower bound assumption on $A_0,B_0,C_0$ and Lemma \ref{lem:sldev}.  Our choices of parameters then finish the proof.
\qed

\vspace{2mm}

\subsection{Proof of a result used in Part 2}

In this section, we next sketch proofs of a two results used in Part 2 \cite{Terry.2024b} of this series.  The results, Theorems 3.8 and 4.6 there, characterize when a hereditary property is ``far" from omitting a given configuration. We refer the reader to \cite{Terry.2024b} for any undefined terminology.   The theorems we wish to prove are the following.

\begin{theorem}\label{thm:blowupthm}
Suppose $\calH$ is a hereditary graph property and $\calF$ is a finite collection of finite graphs. The following are equivalent.
\begin{enumerate}
\item $\calH$ is close to some $\calH'$ such that $\calH'\cap \calF=\emptyset$. 
\item There is some $n$ so that for all $H\in \calF$, $\calH$ contains no $n$-blowup of $H$.
\end{enumerate}
\end{theorem}

\begin{theorem}\label{thm:blowupthm3graphs}
Suppose $\calH$ is a hereditary $3$-graph property and $\calF$ is a finite collection of finite graphs. The following are equivalent.
\begin{enumerate}
\item $\calH$ is close to some $\calH'$ such that $\calH'\cap \calF=\emptyset$. 
\item There is some $n$ so that for all $H\in \calF$, $\calH$ contains no $n$-blowup of $H$.
\end{enumerate}
\end{theorem} 

We provide sketches of the proofs of both these results in this paper.  We do this here rather than in Part 2 of the series, as the proof sketch of Theorem \ref{thm:blowupthm3graphs} requires the hypergraph regularity of Section \ref{sec:regularity}, which is otherwise not needed in Part 2. 

We begin with the graphs case.  The main ingredients for Theorem \ref{thm:blowupthm} are the following induced removal lemma of \cite{Alon.2000} as well as the standard graph counting lemma (see \cite{regularitysurvey}).

\begin{theorem}\label{thm:indrem}
For all $\eta>0$ and $N$, there are $c>0$, $C>0$, and $n_0$ such that the following holds.  Suppose $\calF$ is a family of graphs, each with at most $N$ vertices, $H$ is a graph on $n\geq n_0$ vertices, and for every $G\in \calF$, $H$ contains at most $c n^{v(G)}$ induced copies of $G$. Then $H$ is $\eta$-close to some $H'$ which contains no induced copies of $G$.
\end{theorem}

\begin{theorem}\label{thm:counting}
For every $t\geq 2$ and $\gamma>0$ there exists $\e>0$ such that the following holds. Let $G=(V,E)$ be a graph, and suppose $V_1, \ldots ,V_t \subseteq V$ are such that for each $1\leq i\neq j\leq t$, $(V_i,V_j)$ is an $\e$-regular pair with density $d_{ij}$.  Then the number of $(v_1,\ldots, v_t)\in \prod_{i=1}^tV_i$ such that $v_iv_j\in E$ for all $1\leq i\neq j\leq t$ is at least $(\prod_{1\leq i<j\leq t}d_{ij}- \gamma)\prod_{i=1}^t|V_i|$.
\end{theorem}

We now give a sketch of the proof of Theorem \ref{thm:blowupthm}.

\vspace{2mm}

\noindent{\bf Proof of Theorem \ref{thm:blowupthm}.} Given a sufficiently large $n$, it is a counting exercise to see there is some $\delta>0$ (depending on the number of vertices in $H$) so that if $G$ is an $n$-blow up of $H$, then one must change at least $\delta |V(G)|^2$ pairs from $G$ to remove all induced copies of $H$.  This yields the implication ``not (1)" implies ``not (2)" from Theorem \ref{thm:blowupthm}.  

For the other direction, one assumes (1) and shows (2) as follows.  Assume $\calH$ contains no $n$-blow up of $H$.   Fix a sufficiently large graph $G$ from $\calH$. Apply Theorem \ref{thm:reg} to obtain an $\e$-regular equipartition $\calP=\{V_1,\ldots, V_t\}$ for $G$.  Let $\Sigma_{reg}$ be the set of regular pairs from $\calP$, and let $\Sigma_{err}$ the set of irregular pairs.  We now define three subcollections of the regular pairs. 
\begin{align*}
E_0&=\{(X,Y)\in \Sigma_{reg}: d_G(X,Y)\leq \e\},\\
E_1&=\{(X,Y)\in \Sigma_{reg}: d_G(X,Y)\geq 1-\e\},\text{ and }\\
E_2&=\{(X,Y)\in \Sigma_{reg}: \e<d_G(X,Y)<1-\e\}.
\end{align*}

Our next claim shows that there are restrictions on certain kinds of embeddings involving elements of $\calF$ and the pairs in $E_0,E_1$.

\begin{claim}\label{cl:blw} For all $H\in \calF$, there exists no injective map $f:V(H)\rightarrow \calP$ so that the following holds.
\begin{itemize}
\item $uv\in E(H)$ implies $f(u)f(v)\in E_1\cup E_2$,
\item $uv\notin E(H)$ implies $f(u)f(v)\in E_0\cup E_2$.
\end{itemize}
\end{claim}
\begin{proof}
Suppose towards a contradiction there is some $H\in \calF$ for which such a map $f$ exists. We define an auxiliary graph $G'$ by duplicating vertices and edges from $G$ as follows.  We begin by defining $G'$ to have vertex set $V'=\bigcup_{u\in V(H)}\bigcup_{i=1}^nV_{ui}$, where each $V_{ui}$ is a copy of $V(G)$.  We then define $G'$ to have edge set $E'$, where for each $u,v\in V(H)$ and $1\leq i,j\leq n$, $E'\cap K_2[V_{ui},V_{vj}]$ is a copy of $G[f(u),f(v)]$.  By Theorem \ref{thm:counting}, $G'$ contains an induced $n$-blowup of $H$.  By construction of $G'$, this implies $G$ also contains an induced $n$-blowup of $H$, a contradiction. This conclude the proof of Claim \ref{cl:blw}.
\end{proof}

Note that Claim \ref{cl:blw} implies $E_2=\emptyset$.  Using this and Claim \ref{cl:blw}, we can show $G$ has few copies of  $H$, for all $H\in \calF$. Indeed, set 
$$
\Gamma=\Big(\bigcup_{(X,Y)\in \Sigma_{err}}(X\times Y)\Big)\cup \Big(\bigcup_{(X,Y)\in E_1}(X\times Y)\setminus \overline{E(G)}\Big)\cup \Big(\bigcup_{(X,Y)\in E_0}(X\times Y)\cap \overline{E(G)}\Big).
$$
It is not difficult to check $|\Gamma|\leq 3\e_1|V(G)|^2$.  Suppose now $H\in \calF$ and $f:V(H)\rightarrow V(G)$ is such that $vv'\in E(H)$ if and only if $f(v)f(v')\in E(G)$.  By Claim \ref{cl:blw}, one of the following hold.
\begin{enumerate}
\item there is $uv\in {V(H)\choose 2}$ so that for some $i\in [t]$, $|\{u,v\}\cap V_i|>1$,
\item There is some $uv\in {V(H)\choose 2}$ so that $f(u)f(v)$ is in an irregular pair
\item there is some $uv\in {V(H)\choose 2}$ so that $f(u)f(v)\in \Gamma$.
\end{enumerate}
Since $\calP$ is a regular equipartition and $t$ is large, and since $\Gamma$ is small, it is a standard counting argument to show that the number of such $f$ is at most $\delta |V(G)|^{v(H)}$. Thus there are few induced copies of $H$ in $G$ for all $H\in \calF$.  By Theorem \ref{thm:indrem}, $G$ is close to some $G'$ which is induced $H$-free for all $H\in \calF$.

\qed

We now turn to $3$-graphs.  We will use the following analogue of Theorem \ref{thm:indrem} due to Kohayakawa, Nagle, and R\"{o}dl \cite{KNR} in addition to the $3$-graph analogue Theorem \ref{thm:counting} (namely Proposition \ref{prop:countinggowers}). 

\begin{theorem}[Kohayakawa, Nagle, and R\"{o}dl \cite{KNR}]\label{thm:indrem3graph}
For all $\eta>0$ and $N$, there are $c>0$, $C>0$, and $n_0$ such that the following holds.  Suppose $\calF$ is a family of $3$-graphs, each with at most $N$ vertices, $H$ is a $3$-graph on $n\geq n_0$ vertices, and for every $G\in \calF$, $H$ contains at most $c n^{v(G)}$ induced copies of $G$. Then $H$ is $\eta$-close to some $H'$ which contains no induced copies of $G$.
\end{theorem}

\noindent{\bf Proof of Theorem \ref{thm:blowupthm3graphs}.} Given a sufficiently large $n$, it is a counting exercise to see there is some $\delta>0$ (depending on the number of vertices in $H$) so that if $G$ is an $n$-blow up of $H$, then one must change at least $\delta |V(G)|^3$ pairs from $G$ to remove all induced copies of $H$.  This yields the implication ``not (1)" implies ``not (2)" follows from Theorem \ref{thm:blowupthm3graphs}.  

For the other direction, we assume (1) and show (2).  Assume $\calH$ contains no $n$-blow up of $H$.   Choose $\e_1$ sufficiently small and $\e_2:\mathbb{N}\rightarrow (0,1)$ tending to $0$ sufficiently fast as $n$ tends to infinity.  Given a sufficiently large $3$-graph $H'$ from $\calH$, apply Theorem \ref{thm:reg2} to obtain an equitable, $\dev_{2,3}(\e_1,\e_2(\ell))$-regular $(t,\ell,\e_1,\e_2(\ell))$-decomposition $\calP$ for $H'$.  Say $\calP_1=\{V_1,\ldots, V_t\}$ and $\calP_2=\{P_{ij}^{\alpha}: 1\leq i,j\leq t, 1\leq \alpha\leq \ell\}$.  Let $G_{ijk}^{\alpha\beta\gamma}$ denote the triad $(V_i,V_j,V_k; P_{ij}^{\alpha},P_{ik}^{\beta},P_{jk}^{\gamma})$.  Let $\Sigma_{reg}$ be the set of regular triads from $\calP$, and let $\Sigma_{err}$ be the set of irregular triads from $\calP$.  We now define three distinguished subsets of the regular triads. 
\begin{align*}
E_0&=\{G\in \Sigma_{reg}: d_{H'}(G)\leq \e\},\\
E_1&=\{G\in \Sigma_{reg}: d_{H'}(G)\geq 1-\e\},\text{ and }\\
E_2&=\{G\in \Sigma_{reg}: \e<d_{H'}(G)< 1-\e\}.
\end{align*}
Our next claim shows that there are restrictions on certain embeddings from elements in $\calF$ using $E_1,E_0$.

\begin{claim}\label{cl:blw3}
For all $H\in \calF$, there exists no pair of maps $f:V(H)\rightarrow [t]$ and $g:{V(H)\choose 2}\rightarrow [\ell]$ so that the following holds.
\begin{itemize}
\item $uvw\in E(H)$ implies $G_{f(u)f(v)f(w)}^{g(uv),g(uw),g(vw)}\in E_1\cup E_2$,
\item $uvw\notin E(H)$ implies $G_{f(u)f(v)f(w)}^{g(uv),g(uw),g(vw)}\in E_0\cup E_2$.
\end{itemize}
\end{claim}
\begin{proof}
Suppose towards a contradiction there is some $H\in \calF$ for which such a pair of maps $f$ and $g$ exist. We will next define auxiliary bigraphs and an auxiliary $3$-graph by duplicating elements from $\calP$ and $H'$.  We begin by defining an auxiliary vertex set $\bigcup_{u\in V(H)}\bigcup_{i=1}^nW_{ui}$, where for each $u\in V(H)$ and $i\in [n]$, $W_{ui}$ is a copy of $V_{f(u)}$.  For each $u,v\in V(H)$ and $1\leq i,j\leq n$, define a bigraph $G_{ui,vj}=(W_{ui},W_{vj}; E_{ui,vj})$, where $E_{ui,vj}$ is a copy of $P_{f(u)f(v)}^{g(uv)}$.  Finally, define an auxiliary $3$-graph $H_0$ with vertex set $\bigcup_{u\in V(H)}\bigcup_{i=1}^nW_{ui}$ and edge set $E(H_0)$ defined as follows.  For each $uvw\in E(H)$ and $1\leq i,j,k\leq n$, let $\overline{E(H_0)}\cap (W_{ui}\times W_{vj}\times W_{wk})$ be a copy of $\overline{E(H')}\cap K_3(G_{f(u)f(v)f(w)}^{g(uv),g(uw),g(vw)})$, and for each $uvw\notin E(H)$ and $1\leq i,j,k\leq n$, let $\overline{E(H_0)}\cap (W_{ui}\times W_{vj}\times W_{wk})$ be a copy of $K_3(G_{f(u)f(v)f(w)}^{g(uv),g(uw),g(vw)})\setminus \overline{E(H')}$.  By Corollary \ref{cor:countingcor}, $H_0$ contains an induced $n$-blowup of $H$.  By definition of $H_0$, this implies $H'$ also contains an $n$-blowup of $H$, a contradiction.
\end{proof}

Note that Claim \ref{cl:blw3} implies $E_2=\emptyset$.  We can use this fact and Claim \ref{cl:blw3} to show $G$ has few copies of  $H$, for all $H\in \calF$. Let 
$$
\Gamma=\bigcup_{G\in \Sigma_{err}}K_3(G)\cup \bigcup_{G\in E_1}(K_3(G)\setminus E(G))\cup \bigcup_{G\in E_0}K_3(G)\cap E(G)
$$
It is not difficult to check $|\Gamma|\leq 3\e_1|V(H')|^3$.  Suppose now $H\in \calF$ and $f:V(H)\rightarrow V(H')$ is such that $vv'v''\in E(H)$ if and only if $f(v)f(v')f(v'')\in E(G)$.  By Claim \ref{cl:blw3}, one of the following hold.
\begin{enumerate}
\item there is $uvw\in {V(H)\choose 3}$ so that for some $i\in [t]$, $|\{u,v,w\}\cap V_i|>1$,
\item There is some $uvw\in {V(H)\choose 3}$ so that $f(u)f(v)f(w)$ is in an irregular triad,
\item there is some $uvw\in {V(H)\choose 3}$ so that $f(u)f(v)f(w)\in \Gamma$.
\end{enumerate}
Since $\calP_1$ is an equitable $\dev_{2,3}(\e_1,\e_2(\ell))$-regular $(t,\ell,\e_1,\e_2(\ell))$-decomposition $t$ is large, and since $\Gamma$ is small, we have that the number of such maps is at most $\delta |V(H')|^{v(H)}$. Thus there are few induced copies of $H$ in $H'$.  By Theorem \ref{thm:indrem3graph}, $H'$ is close to some $H''$ which is induced $H$-free for all $H\in \calF$.
\qed

\bibliography{growth.bib}
\bibliographystyle{amsplain}

\end{document}